\documentclass[10pt]{article}

\usepackage{amscd,amsfonts,amsmath,amssymb,amsthm}
\usepackage[section]{algorithm}
\usepackage[noend]{algpseudocode}
\usepackage[title]{appendix}
\usepackage[affil-it]{authblk}
\usepackage{boldtensors}
\usepackage{booktabs}
\usepackage{fancyhdr}
\usepackage{float}
\usepackage{graphicx}
\usepackage{ifpdf}
\usepackage{mathtools}
\usepackage[frozencache,cachedir=.]{minted}
\usepackage{soul}
\usepackage{subcaption}
\usepackage{url}
\usepackage{xcolor}
\usepackage{hyperref}
\usepackage[nameinlink]{cleveref}

\ifpdf
\hypersetup{
    pdftitle={Tensor parametric Hamiltonian operator inference},
    pdfauthor={Arjun Vijaywargiya, Shane A. McQuarrie, Anthony Gruber}
}
\fi

\DeclareMathOperator*{\argmin}{arg\,min}
\DeclareMathOperator{\baroplus}{\bar{\oplus}}
\renewcommand{\epsilon}{\varepsilon}
\newcommand{\trp}{^{\mathsf{T}}}
\newcommand{\invtrp}{^{-\mathsf{T}}}
\newcommand{\lr}[1]{\left(#1\right)}
\newcommand{\IP}[2]{\left\langle #1,#2\right\rangle}
\newcommand{\nn}[1]{\left\|#1\right\|}
\newcommand{\rvec}{\mathrm{rvec}}
\newcommand{\cvec}{\mathrm{cvec}}
\newcommand{\rmat}{\mathrm{rmat}}
\newcommand{\cmat}{\mathrm{cmat}}
\newcommand{\vect}{\mathrm{vec}}

\newcommand{\tens}[1]{\mathbf{#1}}
\newcommand{\Mnabla}{\nabla^{M}\!}
\let\originalleft\left
\let\originalright\right
\renewcommand{\left}{\mathopen{}\mathclose\bgroup\originalleft}
\renewcommand{\right}{\aftergroup\egroup\originalright}

\hypersetup{
    colorlinks=true,
    linkcolor=black,
    citecolor=black,
    urlcolor=black,
}

\crefformat{equation}{#2(#1)#3}
\numberwithin{equation}{section}

\theoremstyle{plain}
\newtheorem{theorem}{Theorem}
\newtheorem{proposition}{Proposition}
\newtheorem{lemma}{Lemma}
\newtheorem{corollary}{Corollary}

\theoremstyle{remark}
\newtheorem{remark}{Remark}

\theoremstyle{definition}
\newtheorem{definition}{Definition}

\numberwithin{theorem}{section}
\numberwithin{proposition}{section}
\numberwithin{lemma}{section}
\numberwithin{corollary}{section}
\numberwithin{remark}{section}
\numberwithin{definition}{section}

\crefname{theorem}{Theorem}{Theorems}
\Crefname{theorem}{Theorem}{Theorems}
\crefname{proposition}{Proposition}{Propositions}
\Crefname{proposition}{Proposition}{Propositions}
\crefname{lemma}{Lemma}{Lemmata}
\Crefname{lemma}{Lemma}{Lemmata}
\crefname{corollary}{Corollary}{Corollaries}
\Crefname{corollary}{Corollary}{Corollaries}
\crefname{algorithm}{Algorithm}{Algorithms}
\Crefname{algorithm}{Algorithm}{Algorithms}
\crefname{appendix}{Appendix}{Appendices}
\Crefname{appendix}{Appendix}{Appendices}

\setminted{
    style=pastie,
    fontsize=\normalsize,
    linenos,
    frame=none
}

\begin{document} 

\pdfinfo{
   /Author (Arjun Vijaywargiya, Shane A. McQuarrie, Anthony Gruber)
   /Title (Tensor parametric Hamiltonian operator inference)
   /Keywords (parametric model reduction, tensor calculus, Hamiltonian systems, structure preservation, scientific machine learning)
}

\title{\textbf{Tensor parametric Hamiltonian operator inference}}

\author[1,2]{Arjun~Vijaywargiya}
\author[3]{Shane~A.~McQuarrie}
\author[2]{Anthony Gruber\thanks{Corresponding author. E-mail: \href{mailto:adgrube@sandia.gov}{adgrube@sandia.gov}.}}

\affil[1]{\normalsize Department of Applied and Computational Mathematics and Statistics, University of Notre Dame}
\affil[2]{\normalsize Computational Mathematics, Center for Computing Research, Sandia National Laboratories}
\affil[3]{\normalsize Scientific Machine Learning, Center for Computing Research, Sandia National Laboratories}

\date{}

\maketitle

\vspace{-.25in}

\begin{abstract}
\noindent
This work presents a tensorial approach to constructing data-driven reduced-order models corresponding to semi-discrete partial differential equations with canonical Hamiltonian structure. By expressing parameter-varying operators with affine dependence as contractions of a generalized parameter vector against a constant tensor, this method leverages the operator inference framework to capture parametric dependence in the learned reduced-order model via the solution to a convex, least-squares optimization problem. This leads to a concise and straightforward implementation which compactifies previous parametric operator inference approaches and directly extends to learning parametric operators with symmetry constraints\textemdash a key feature required for constructing structure-preserving surrogates of Hamiltonian systems. The proposed approach is demonstrated on both a (non-Hamiltonian) heat equation with variable diffusion coefficient as well as a Hamiltonian wave equation with variable wave speed.
 \end{abstract}

\noindent
\textbf{Keywords}: parametric model reduction, tensor calculus, Hamiltonian systems, structure preservation, scientific machine learning

\section{Introduction}
Mathematical models based on partial differential equations (PDEs) are often formed through the combination of terms describing locally distinct physical processes.  For example, an advection-diffusion equation describing the evolution of a scalar quantity $q$ on a domain $\Omega\subset\mathbb{R}^d$, given by $\partial_t q = \nabla\cdot\lr{~D\nabla q - ~v q}$ for appropriate diffusion field $~D$ and velocity $~v$, combines the term $\nabla\cdot\lr{~D\nabla q}$ expressing diffusion with a corresponding term $~v\cdot\nabla q$ describing material transport.
The practical utility of PDE-based models typically relies on the calibration of a number of parameters ($~D,~v$ in the example above), which are critical to the behavior of solutions, hence also to the physical realism of the relevant model. Moreover, these parameters are manifestly application dependent, meaning that any sufficiently general technique for constructing data-driven surrogate models for physical systems must be amenable to parametric variability.
This paper develops a data-driven method of constructing reduced-order models (ROMs)---low-dimensional dynamical systems for approximating the dominant dynamics of a large-scale physical model---which seamlessly encodes parametric dependencies while simultaneously preserving other structural features of interest.

The development of parametric ROMs is straightforward and well understood in the case of intrusive, projection-based methods \cite{Benner2015,benner2017model,Grepl05,RozHP08,Veroy05,Veroy03}.
Given a numerical simulator---termed the full-order model (FOM) due to it usually having a large number of degrees of freedom---for a physical process, intrusive methods define a ROM through Galerkin projection onto a low-dimensional manifold which, ideally, captures the dominant state behavior the system.
In this setting, any parametric dependence appearing in the FOM is automatically inherited by the resulting ROM.
However, intrusive methods require direct manipulation of the FOM equations as well as access to the underlying operators in the FOM source code. This may not always be possible, for instance when using proprietary or complex legacy codes whose implementations are inaccessible or not amenable to modifications. Even when codes are available, the direct manipulations required by intrusive model reduction can be highly non-trivial and error prone.

In response to these challenges, many non-intrusive model reduction strategies have recently been developed and successfully applied to a variety of large-scale systems. In contrast to generating ROM equations through the direct projection of FOM equations, non-intrusive model reduction techniques typically rely on the use of a pre-specified model form, along with snapshot data of FOM solutions, to infer a ROM which can be simulated without access to FOM code.
Several approaches, including as dynamic mode decomposition (DMD)~\cite{SCHMID_2010,schmid2022dmd} and operator inference (OpInf)~\cite{Ghattas_Willcox_2021,kramer2024opinfsurvey,PEHERSTORFER2016196}, define a convex linear least-squares regression to infer ROM components. In their standard (non-parametric) forms, these approaches tend to ``average out'' the effects of parameter dependence. To better accommodate parametric problems, extensions have been developed to DMD \cite{DUAN2024highorderPDMD,huhn2023pDMD,sayadi2015pDMD} and OpInf \cite{mcquarrie2023popinf,Yildiz:2024} that build known parametric structure into the pre-specified model form.
Other strategies use more general regression techniques, such as Gaussian process regression \cite{guo2018gpstructural,guo2019pgps} or deep learning \cite{bhattacharya2021romwithnets,franco2022deepppdes,GAO2020132614}.

An additional challenge in model reduction generally---and one that can be especially detrimental to data-driven methods in particular if left unaddressed---is to ensure that a ROM possesses certain desired properties of the dynamics being approximated. For instance, this paper focuses on (canonical) Hamiltonian systems, which feature a symplectic structure and conserve the total energy of solutions \cite{Marsden_BOOK_1998}. Several strategies exist to guarantee such properties in surrogates learned from data (see, e.g., \cite{bertalan2019hamiltoniandata,greydanus2019hamiltonian,JIN2020166}); in particular, the OpInf framework has enjoyed several recent extensions specifically catered to gradient flows and Hamiltonian systems \cite{geng2025portHamiltonianOpInf,geng2024gradopinf,GRUBER2023116334,Sharma:2024,Sharma:2024_2,sharma2022hamiltonian}. The OpInf formulations that impose structural constraints also tends to negate the need for regularization techniques to stabilize the dynamics of the resulting ROMs \cite{mcquarrie2021regopinf,mcquarrie2023popinf,qian2022pdeopinf}.
However, enforcing structure also tends to increase the complexity of the associated inference procedure mathematically, computationally, and in terms of the difficulty of implementation.
Each of these issues are further compounded by the presence of parametric variations.

The purpose of the present work is to develop a simple and general way to incorporate parametric dependence into the OpInf framework and, simultaneously, the preservation of structural properties in the resulting ROM.
The following simple observation is key: when a matrix operator $~A(~\mu):\mathbb{R}^N\to \mathbb{R}^N$ depends linearly on a known function $~\mu' = ~\theta(~\mu) \in \mathbb{R}^{p'}$ of the parameter vector $~\mu\in\mathbb{R}^p$, then $~A(~\mu) = \tens{T}~\mu'$ can be written as the contraction of a \textit{constant} tensor $\tens{T}\in\mathbb{R}^{N\times N\times p'}$ against the transformed parameters $~\mu'\in\mathbb{R}^{p'}$.  Said differently, the adjoint relationship $A\mapsto (B\mapsto C) \cong (A\otimes B)\mapsto C$ between continuous linear mappings $\mapsto$ of vector spaces and the tensor product $\otimes$ yields a description of the ``affine'' parametric dependence in $~A(~\mu)$ in terms of the linear action of a constant tensor $\tens{T}$ on $~\mu'$.  This notion is general enough to encompass a wide range of parametric behavior and will be shown to enable an OpInf learning problem which does not rely on local approximation techniques such as Taylor expansion~\cite{Farcas2023ParametricNR} or linear interpolation~\cite{PEHERSTORFER2016196}.
This approach builds on previous OpInf methods for affine parametric system~\cite{mcquarrie2023popinf,YGBK2021} while also enabling a novel extension to parametric Hamiltonian systems, in which the preservation of certain structure information is critical to stability and predictive accuracy. The extension is natural given the mathematical concision and streamlined implementation brought about by the tensorized approach.
To summarize, the contributions of this article are twofold: first, a tensor-based approach to parametric OpInf which is convex, mathematically concise, simple to implement, and recovers previous work as a special case; and second, an extension of parametric OpInf methodology to systems with Hamiltonian structure, leading to non-intrusive parametric ROMs which guarantee symplecticity and the production of conservative Hamiltonian dynamics.

The remainder of the work is structured as follows. \Cref{sec:generic} presents the proposed tensor parametric approach in the structure-agnostic case by defining a data-driven tensor inference problem, expressing its closed form solution, analyzing uniqueness conditions, and detailing concise numerical implementations.
\Cref{sec:Hamiltonian} discusses Hamiltonian systems and extends the tensor parametric inference to operators with symmetry appropriate for reduced-order models of Hamiltonian dynamics.
\Cref{sec:numerics} presents numerical results for parametric heat and wave equations to test the approaches of \Cref{sec:generic} and \Cref{sec:Hamiltonian}, respectively, demonstrating the performance of the proposed method.
Finally, \Cref{sec:conclusion} offers some concluding remarks and avenues for future work.

 \section{Tensor parametric operator inference}\label{sec:generic}
This section focuses on systems of ordinary differential equations (ODEs) with affine parametric dependence but without symmetry or conservation properties. \Cref{sec:affine-setting} defines the problem setting and introduces tensorized forms of both the original and approximate reduced dynamics. In \Cref{sec:tpopinf-unconstrained}, a strategy is presented and analyzed for learning reduced tensorized parametric dynamics from state data. \Cref{sec:heat-equation} then outlines an example where the methodology can be applied, with numerical results deferred to \Cref{sec:numerics}.

\subsection{Affine parametric linear systems} \label{sec:affine-setting}
Consider a linear parametric system of ODEs in which the dynamics depends affinely (i.e., generalized linearly) on a vector $~\mu\in\mathbb{R}^{p}$ of parameters:
\begin{align}
    \label{eq:ODE}
    ~M\dot{~q}
    &= ~A(~\mu)~q,
    &
    ~q(0,~\mu)
    &= ~q_0(~\mu),
    &
    ~A(~\mu)
    &= \sum_{x=1}^{p'} \theta_x(~\mu)~A_x,
\end{align}
where $~q=~q(t,~\mu)\in\mathbb{R}^N$ is the ODE state vector, $~q_0(~\mu)\in\mathbb{R}^N$ is a given initial condition, $~M\in\mathbb{R}^{N\times N}$ is symmetric positive definite (e.g., a mass matrix in a finite element method), and $~A:\mathbb{R}^{p}\to\mathbb{R}^{N\times N}$ is a matrix-valued function defined by the (possibly nonlinear) scalar-valued functions $\theta_1,\ldots,\theta_{p'}:\mathbb{R}^{p}\to\mathbb{R}$ and constant matrices $~A_1,\ldots,~A_{p'}\in\mathbb{R}^{N\times N}$. Such systems frequently arise from a method of lines semi-discretization of time-dependent PDEs (see, e.g., \cite{YUAN1999375}).
As the discrete dimension $N$ is typically very large, \cref{eq:ODE} is called the \emph{full-order model} (FOM), which serves as the best available discretization of the dynamics under consideration.

When the matrices $~A_1,\ldots,~A_{p'}$ can be accessed intrusively, it is straightforward to construct a computationally efficient \emph{reduced-order model} (ROM) for~\cref{eq:ODE} through projection. This paper considers projection onto $~M$-orthonormal linear reduced bases $~U\in\mathbb{R}^{N\times r} (r\ll N)$, though nonlinear approximations such as in \cite{barnett2022quadmanifold,geelen2024learning,geelen2023quadmanifold,jain2017quadratic,schwerdtner2024greedy} may also be considered with minor alterations to the approach.
More precisely, given a trial space basis $~U \in \mathbb R^{N \times r}$ satisfying $~U\trp~M~U = ~I$ and a low-rank  approximation $\tilde{~q} \coloneqq ~U\hat{~q} \approx ~q$ defined in terms of an unknown coefficient vector $\hat{~q} = \hat{~q}(t,~\mu) \in \mathbb R^r$, Galerkin projection of the FOM~\cref{eq:ODE} onto the column space of $~U$ gives the following reduced ODE system \cite{Benner2015}:
\begin{align}\label{eq:rODE}
    \dot{\hat{~q}}
    &= ~U\trp~A(~\mu)~U\hat{~q}
= \sum_{x=1}^{p'}\theta_x(~\mu)\hat{~A}_{x}\hat{~q},
    \qquad
    \hat{~q}(0,~\mu) = ~U\trp~M ~q_0(~\mu),
\end{align}
where $\hat{~A}_x \coloneqq ~U\trp~A_x~U\in\mathbb{R}^{r\times r}$ for each $x=1,\ldots,p'$.
The system~\cref{eq:rODE} is a low-dimensional approximation to the FOM~\cref{eq:ODE} which, importantly, exhibits the same affine-parametric structure as the FOM. Provided the trial space basis $~U$ is well chosen (e.g., $L^2$-optimal over some data range), the ROM~\cref{eq:rODE} can accurately recover the dominant behavior of the original system at a significantly reduced computational cost, since the matrices $\hat{~A}_1,\ldots,\hat{~A}_{p'}$ can be pre-computed once and reused for multiple values of the parameter $~\mu$. However, computing each $\hat{~A}_x$ involves direct matrix-matrix multiplications with the full-order object $~A_x$, an intrusive procedure that requires access to the code implementing the FOM~\cref{eq:ODE}.
The goal of parametric OpInf is to non-intrusively construct a ROM similar to~\cref{eq:rODE} by inferring low-dimensional operators from observations of the state $~q$.

\begin{remark}[Basis orthonormality]
The reduced-order matrix $\hat{~A}(~\mu) = ~U\trp~A(~\mu)~U$ is symmetric (or skew symmetric) whenever the full-order matrix $~A(~\mu)$ is.
If a basis $~U$ is chosen which is not $~M$-orthonormal, Galerkin projection results in the slightly different ROM
\begin{align*}
    \dot{\hat{~q}}
    = (~U\trp~M~U)^{-1}~U\trp~A(~\mu)~U\hat{~q},
\end{align*}
which does not inherit (anti-)symmetry from the FOM.
This paper focuses on $~M$-orthonormal bases because the goal of \Cref{sec:Hamiltonian} is to construct ROMs which preserve system symmetry without direct access to $~A(~\mu)$.
\end{remark}

Before proceeding,  it is convenient to reformulate \cref{eq:ODE} and \cref{eq:rODE} in a way which exposes their tensorial structure.  Let $~\mu'= ~\theta(~\mu) \coloneqq (\theta_1(~\mu),\ldots,\theta_{p'}(~\mu))\trp\in\mathbb{R}^{p'}$ denote the vector of scalar coefficients in the affine expansion of $~A(~\mu)$.
Then the FOM~\cref{eq:ODE} can be written in the compact form
\begin{align}
    \label{eq:ODE-tensorized}
    ~M\dot{~q}
    = \lr{\tens{T}~\mu'}~q,
    \qquad
    ~q(0,~\mu)
    = ~q_0(~\mu),
\end{align}
where $\tens{T}\in\mathbb{R}^{N \times N \times p'}$ is the order-3 tensor satisfying $\tens{T}~\mu' = ~A(~\mu)\in\mathbb{R}^{N\times N}$, obtained from the matrix-valued function $~A(~\mu)$ by tensor-Hom adjunction (i.e., adjointness).
In a similar fashion, the ROM~\cref{eq:rODE} can be expressed as
\begin{subequations}
\label{eq:rODE-tensorized}
\begin{align}
    \dot{\hat{~q}}
    = \lr{\hat{\tens{T}}~\mu'}\hat{~q},
    \qquad
    \hat{~q}(0,~\mu)
    = ~U\trp ~M~q_0(~\mu),
\end{align}
where $\hat{\tens{T}}\in\mathbb{R}^{r \times r \times p'}$ is a reduced order-$3$ tensor with entries
\begin{align}\label{eq:reducedT}
    \hat{\mathrm{T}}_{abx}
    = \sum_{i,j=1}^{N} U_{ia} \mathrm{T}_{ijx} U_{jb}
    = \sum_{i,j=1}^{N} \lr{~U\trp}_{ai}(~A_x)_{ij} \lr{~U}_{jb}
    = (\hat{~A}_x)_{ab},
\end{align}
\end{subequations}
where $\mathrm{T}_{ijx} = (~A_x)_{ij}$ denotes the $(i,j,x)$-th entry of $\tens{T}$, $U_{ia}$ denotes the $(i,a)$-th entry of $~U$, and $(\hat{~A}_x)_{ab}$ denotes the $(a,b)$-th entry of $\hat{~A}_x$.
Note that $\hat{\tens{T}}$ can be viewed as the projection of the full-order tensor $\tens{T}$ or, equivalently, as the tensor-Hom adjoint of the projected matrix-valued function $\hat{~A}(~\mu) \coloneqq \sum_{x=1}^{p'}\theta_x(~\mu)\hat{~A}_x$. The advantage of the expression \cref{eq:rODE-tensorized} lies primarily in its compactness: because the explicit tensor contraction $\hat{\tens{T}}~\mu'$ exposes higher-level structure and removes the need for bookkeeping in terms of the matrices $\hat{~A}_1,...,\hat{~A}_{p'}$, additional considerations such as operator symmetries can be more easily included in data-driven ROM construction. Besides leading to a streamlined implementation of the method in \cite{mcquarrie2023popinf,YGBK2021}, this approach also enables simplified proofs of established results (e.g.,~\Cref{cor:uniqueness}) and facilitates an extension which preserves Hamiltonian structure information (c.f.~\Cref{sec:Hamiltonian}).

\subsection{Tensor parametric operator inference}\label{sec:tpopinf-unconstrained}
The ROM~\cref{eq:rODE} is fully and uniquely determined by the entries of the reduced-order matrices $\hat{~A}_1,\ldots,\hat{~A}_{p'}$, which in turn depend on the full-order matrices $~A_1,\ldots,~A_{p'}$. When these matrices are not available for computation due to, e.g., proprietary or legacy software implementations, an alternative to direct projection is to choose the matrices that minimize the residual of \cref{eq:rODE} with respect to a collection of training states. This is the Operator Inference (OpInf) strategy introduced in~\cite{PEHERSTORFER2016196} and extended to affine-parametric systems in~\cite{mcquarrie2023popinf,YGBK2021}.
This section develops a corresponding OpInf procedure for the tensorized ROM~\cref{eq:rODE-tensorized} and presents some analysis of its basic properties.

Suppose access is granted to fixed-time snapshots of the full-order state---solutions to the FOM~\cref{eq:ODE}---at $N_s$ training parameter values $~\mu_1,\ldots,~\mu_{N_s}\in\mathbb{R}^{p}$ and $N_t$ time instances $t_1,\ldots,t_{N_t}\in\mathbb{R}$, organized into the state snapshot matrices
\begin{align*}
    ~Q_s
    = \left[\begin{array}{ccc}
        ~q(t_1,~\mu_s)
        & \cdots &
        ~q(t_{N_t},~\mu_s)
    \end{array}\right]\in\mathbb{R}^{N \times N_t},
    \quad
    s = 1, \ldots, N_s.
\end{align*}
Given an $~M$-orthonormal basis matrix $~U\in\mathbb{R}^{N\times r}$, define the reduced state snapshot matrices $\hat{~Q}_s \coloneqq ~U\trp~M~Q_s\in\mathbb{R}^{r\times N_t}$, $s=1,\ldots,N_s,$ as well as corresponding reduced state time derivative matrices $\dot{\hat{~Q}}_s = ~U\trp~M D_t(~Q_s)\in\mathbb{R}^{r\times N_t}$, which can be estimated, e.g., via finite difference approximation.
Specifically, the $\alpha$-th column of $\dot{\hat{~Q}}_s$ is an estimate for $\frac{d}{dt}~U\trp~M~q(t,~\mu_s)|_{t=t_\alpha}$.
A reduced tensor $\bar{\tens{T}}$ which approximates the intrusively constructed $\hat{\tens{T}}$ in~\cref{eq:rODE-tensorized} can then be inferred by solving the convex minimization problem
\begin{align}
    \label{eq:genericminprob}
    \argmin_{\bar{\tens{T}}}\frac{1}{2}\sum_{s=1}^{N_s}\nn{\dot{\hat{~Q}}_s - \lr{\bar{\tens{T}}~\mu_s'}\hat{~Q}_s }^2,
\end{align}
where $\|~F\|^2=\sum_{ij}F_{ij}F_{ij}$ denotes the Frobenius matrix norm.

Conveniently, convexity implies that the inference of $\bar{\tens{T}}$ in~\cref{eq:genericminprob} is equivalent to a solving a particular linear system.
The precise statement of this result is simplified with the following notation: for an order-$n$ tensor $\tens{B}\in\mathbb{R}^{N_1\times N_2\times...\times N_n}$, let $\cvec_{ij}\,\tens{B}, \rvec_{ij}\,\tens{B}\in \mathbb{R}^{N_1\times...\times N_iN_j\times ...\times N_n}$ denote the partial vectorizations that decrement the tensor degree of $\tens{B}$ by unrolling the ordered indices $i<j$ ``column-wise'' (in the case of $\cvec_{ij}$) or ``row-wise'' (in the case of $\rvec_{ij}$) into a single index located at the $i$-th position. More precisely, indexing the $m$-th dimension with $k_m$,
\begin{subequations}
\label{eq:vectorize}
\begin{align}
&\lr{\cvec_{ij}\,\tens{B}}_{k_1,\ldots,k_{i-1},(k_j - 1)N_i + k_i,k_{i+1},\ldots,k_{j-1},k_{j+1},\ldots,k_n}
    = \mathrm{B}_{k_1,\ldots,k_{i},\ldots,k_{j},\ldots,k_n}, \\
    &\lr{\rvec_{ij}\,\tens{B}}_{k_1,\ldots,k_{i-1},(k_i - 1)N_j + k_j,k_{i+1},\ldots,k_{j-1},k_{j+1},\ldots,k_n}
    = \mathrm{B}_{k_1,\ldots,k_{i},\ldots,k_{j},\ldots,k_n}.
\end{align}
\end{subequations}
Additionally, let $\cmat_{ij}$ and $\rmat_{ij}$ be the natural inverses of $\cvec_{ij}$ and $\rvec_{ij}$, i.e., $$\cmat_{ij}\lr{\cvec_{ij}\,\tens{B}} = \rmat_{ij}\lr{\rvec_{ij}\,\tens{B}} = \tens{B},$$ and let $\otimes$ denote the outer product: for an order-$m$ tensor $\tens{A}\in\mathbb{R}^{M_1\times\cdots\times M_{m}}$ and an order-$n$ tensor $\tens{B}\in\mathbb{R}^{N_1\times\cdots\times N_{n}}$, $\tens{A}\otimes\tens{B}\in\mathbb{R}^{M_1\times\cdots\times M_{m}\times N_1\cdots\times N_{n}}$ is a tensor of order $m+n$ with entries
\begin{align*}
    (\tens{A}\otimes\tens{B})_{i_1,\ldots,i_m,k_1,\ldots,k_n}
    = \mathrm{A}_{i_1,\ldots,i_m}\mathrm{B}_{k_1,\ldots,k_n}.
\end{align*}
With this notation in place, the following result formulates the inference of $\bar{\tens{T}}$ in \cref{eq:genericminprob} as a system of linear equations.

\begin{theorem}\label{thm:unstructured}
Let $\hat{~Y}_s, \hat{~Z}_s \in \mathbb{R}^{r \times N_t}$ and $~\nu_s\in\mathbb{R}^{p'}$ for $s=1,\ldots,N_s$. The tensor $\bar{\tens{T}}\in\mathbb{R}^{r\times r \times p'}$ minimizes the Lagrangian
\begin{align}
    \label{eq:lagrangian-unstructured}
    L\lr{\bar{\tens{T}}} \coloneqq \frac{1}{2}\sum_{s=1}^{N_s}
    \nn{\hat{~Z}_s - \lr{\bar{\tens{T}}~\nu_s}\hat{~Y}_s}^2
\end{align}
if and only if $\bar{~T} \coloneqq \cvec_{23}\,\bar{\tens{T}}\in\mathbb{R}^{r\times rp'}$ satisfies the linear system
\begin{align}
    \label{eq:normal-equations}
    \hat{~B}\,\bar{~T}\trp
    = \hat{~C}\trp,
\end{align}
where $\hat{~B}\in\mathbb{R}^{rp'\times rp'}$ and $\hat{~C}\in\mathbb{R}^{r\times rp'}$ denote the vectorizations
\begin{align*}
    \hat{~B} &\coloneqq
    \rvec_{12}\cvec_{34}
    \lr{\sum_{s=1}^{N_s}~\nu_s\otimes\hat{~Y}_s\hat{~Y}_s\trp\otimes~\nu_s},
    \quad
    \hat{~C} \coloneqq
    \cvec_{23}\,\lr{\sum_{s=1}^{N_s}\hat{~Z}_s\hat{~Y}_s\trp\otimes~\nu_s}.
\end{align*}
\end{theorem}

\begin{proof}
The proof is a direct calculation of the first-order optimality conditions for the Lagrangian~\cref{eq:lagrangian-unstructured}, appealing to a few minor algebraic results detailed in \Cref{appendix:tensoralgebra}.
Differentiating $L$ with respect to $\bar{\tens{T}}$ and applying the definition of the gradient, it follows that
\begin{align*}
    dL\lr{\bar{\tens{T}}} &= \sum_{s=1}^{N_s}\IP{\lr{d{\bar{\tens{T}}}\,~\nu_s}\hat{~Y}_s}{\hat{~Z}_s - \lr{\bar{\tens{T}}~\nu_s}\hat{~Y}_s} \\
    &= \IP{d\bar{\tens{T}}}{\sum_{s=1}^{N_s}\left[\hat{~Z}_s - \lr{\bar{\tens{T}}~\nu_s}\hat{~Y}_s\right]\hat{~Y}_s\trp\otimes~\nu_s} = \IP{d\bar{\tens{T}}}{\nabla L\lr{\bar{\tens{T}}}},
\end{align*}
where $\IP{\cdot}{\cdot}$ denotes the Frobenius inner product and in which the standard properties listed in \Cref{thm:frobeniusproperties} have been used. Hence,
the stationarity condition $\nabla L\lr{\bar{\tens{T}}}=~0$ implies the tensorial equation
\begin{align}\label{eq:tensoreq}
    \sum_{s=1}^{N_s}\lr{\bar{\tens{T}}~\nu_s}\hat{~Y}_s\hat{~Y}_s\trp\otimes~\nu_s
    = \sum_{s=1}^{N_s}\hat{~Z}_s\hat{~Y}_s\trp\otimes~\nu_s.
\end{align}
From~\Cref{thm:contractionproduct}, the left-hand side can be expressed as the tensor contraction
\begin{align*}
    \sum_{s=1}^{N_s}
    \lr{\bar{\tens{T}}~\nu_s}\hat{~Y}_s\hat{~Y}_s\trp\otimes~\nu_s
    = \bar{\tens{T}}:\sum_{s=1}^{N_s}
    ~\nu_s\otimes\hat{~Y}_s\hat{~Y}_s\trp\otimes~\nu_s,
\end{align*}
with the convention that $\tens{X}_1:\tens{X}_2$ operates over the last two indices of $\tens{X}_1$ and the first two indices of $\tens{X}_2$ in reverse order, i.e., $\tens{X}_1:\tens{X}_2 = \sum_{i_{k-1},i_k}\lr{\tens{X}_1}_{...i_{k-1}i_k}\lr{\tens{X}_2}_{i_ki_{k-1}...}$. By \Cref{thm:contractionvectorization}, vectorizing the contraction column-wise in $\bar{\tens{T}}$ and similarly vectorizing the last two indices of the result transforms the tensorial equation~\cref{eq:tensoreq} into the linear system
\begin{align*}
    \lr{\cvec_{23}\,\bar{\tens{T}}}\,\rvec_{12}\cvec_{34}
    \lr{\sum_{s=1}^{N_s} ~\nu_s\otimes\hat{~Y}_s\hat{~Y}_s\trp\otimes~\nu_s}
    = \cvec_{23}\lr{\sum_{s=1}^{N_s}\hat{~Z}_s\hat{~Y}_s\trp\otimes~\nu_s},
\end{align*}
which, using the definitions in the statement, is the system $\bar{~T}\hat{~B} = \hat{~C}$. Transposition and the symmetry relationship $\hat{~B}\trp = \hat{~B}$ then yield the linear system~\cref{eq:normal-equations}.
\end{proof}

\Cref{thm:unstructured} suggests a straightforward numerical method for solving the tensor parametric OpInf problem~\cref{eq:genericminprob}: use the reduced state snapshot matrices to form the linear system~\cref{eq:normal-equations}, solve for $\bar{~T}$, then set $\bar{\tens{T}} = \cmat_{23}\,\bar{~T}$. The resulting non-intrusively obtained OpInf ROM for the approximate state $\tilde{~q} = ~U\hat{~q}$ is given by
\begin{align}
    \label{eq:rODE-opinf}
    \dot{\hat{~q}}
    = \lr{\bar{\tens{T}}~\mu'}\hat{~q},
    \qquad \hat{~q}(0,~\mu) = ~U\trp~M~q_0(~\mu).
\end{align}
This inference procedure can be accomplished with only a few lines of Python code (see \Cref{alg:normal_eqns_unstructured}) and
can be carried out even when data for only $N_s=1$ parameter sample are available.
In such data-scarce settings, however, the uniqueness of the solution is not guaranteed; conditions for uniqueness are established in \Cref{cor:uniqueness}.

\begin{algorithm}[t]
\caption{A NumPy/SciPy implementation for inferring a tensor $\bar{\tens{T}}$ using the normal equations \cref{eq:normal-equations} from  \Cref{thm:unstructured}.}
\label{alg:normal_eqns_unstructured}
\vspace{.25cm}
\begin{minted}{python}
import numpy as np
import scipy.linalg as la

def infer_Tbar_via_normal_eqns(nus, Ys, Zs):
    p, Ns = nus.shape
    r, Nt, Ns = Ys.shape  # or Zs.shape

    Btsr = np.einsum("xs,ias,jas,ys->xijy", nus, Ys, Ys, nus)
    Bhat = Btsr.transpose(0, 1, 3, 2).reshape((r*p, r*p), order="C")

    Ctsr = np.einsum("ias,jas,xs->ijx", Zs, Ys, nus)
    Chat = Ctsr.reshape((r, r*p), order="F")

    Tmat = la.solve(Bhat, Chat.T, assume_a="sym").T
    return Tmat.reshape((r, r, p), order="F")
\end{minted}
\vspace{.25cm}
\end{algorithm}

The linear system~\cref{eq:normal-equations} can be interpreted as the (vectorized) normal equations for the least-squares problem \cref{eq:genericminprob}, hence \Cref{alg:normal_eqns_unstructured} may be disadvantageous if $\hat{~B}$ is poorly conditioned.
The following result provides an alternative inference procedure with improved conditioning.

\begin{proposition}\label{prop:tsrlstsqprop}
Let $L:\mathbb{R}^{r\times r\times p'}\to\mathbb{R}$ be the Lagrangian function~\cref{eq:lagrangian-unstructured} with argument $\bar{\tens{T}}\in\mathbb{R}^{r\times r \times p'}$, defined by the matrices $\hat{~Y}_s, \hat{~Z}_s \in\mathbb{R}^{r\times N_t}$ and vectors $~\nu_{s}\in\mathbb{R}^{p'}$, $s=1,\ldots,N_s$.
Consider the tensors $\hat{\tens{Y}},\hat{\tens{Z}}\in\mathbb{R}^{r\times N_t\times N_s}$ which result from stacking these matrices
along the parameter index, i.e., $\hat{\mathrm{Y}}_{ijk} = (\hat{~Y}_{k})_{ij}$ and $\hat{\mathrm{Z}}_{ijk} = (\hat{~Z}_{k})_{ij}$.
Writing $\bar{~O} = \cvec_{23}\bar{\tens{T}}\in\mathbb{R}^{r\times rp'}$, this Lagrangian has the equivalent expression
\begin{align}
    \label{eq:least-squares-formulation}
    L(\bar{\tens{T}})
    = \ell\lr{\bar{~O}}
    \coloneqq \nn{~D\bar{~O}\trp-~R\trp}^2,
\end{align}
where $~R = \cvec_{23}\,\hat{\tens{Z}}\in\mathbb{R}^{r\times N_tN_s}$ and $~D = \lr{\cvec_{13}\cvec_{24}\,\tens{K}}\trp\in\mathbb{R}^{N_tN_s\times rp'}$ is defined in terms of the order-4 tensor $\tens{K}\in\mathbb{R}^{r\times N_t\times p'\times N_s}$ with components $\mathrm{K}_{b\alpha xs} = (\hat{~Y}_s)_{b\alpha}(~\nu_s)_x$.
Moreover, the problem of minimizing the vectorized Lagrangian $\ell$ decouples over the rows of $\bar{~O}$.
\begin{proof}
This can be accomplished by vectorizing the parameter index in $L$ (which was previously summed) and is easiest to express with explicit index manipulations.  To that end, define the index ranges
\begin{align*}
    1&\leq a,b\leq r,
    &
    1&\leq \alpha\leq N_t,
    &
    1&\leq x\leq p',
    \\
    1&\leq s\leq N_s,
    &
    1&\leq I\leq N_tN_s,
    &
    1&\leq Y\leq rp'.
\end{align*}
Adopting the Einstein convention where repeated sub- and super-indices imply a summation and letting $\Sigma$ denote the sum over all remaining free indices, the Lagrangian is expressible as
\begin{align*}
    L\lr{\bar{\tens{T}}} &= \sum\left[(\hat{~Z}_s)_{a\alpha} - \bar{\mathrm{T}}_{abx}\lr{~\nu_s}^x(\hat{~Y}_s)^b_\alpha\right]^2 = \sum\left[ \hat{\mathrm{Z}}_{a\alpha s} - \bar{\mathrm{T}}_{abx}\nu^x_s\hat{\mathrm{Y}}^b_{\alpha s}\right]^2 \\
    &= \sum \left[ \lr{\cvec_{23}\,\hat{\tens{Z}}}_{aI} - \lr{\cvec_{23}\,\bar{\tens{T}}}_{aY} \lr{\cvec_{13}\cvec_{24}\,\tens{K}}^Y_I \right]^2 = \ell\lr{\bar{~O}},
\end{align*}
where the first equality is the definition of the Lagrangian, the second uses the definition of the tensors $\hat{\tens{Y}},\hat{\tens{Z}}$, the third vectorizes the double contraction and its result while applying the definition of $\tens{K}$, and the fourth transposes the expression under the norm.
Applying the definitions of $~D$, $\bar{~O}$, and $~R$, it follows that
\[\ell\lr{\bar{~O}} = \nn{~R - \bar{~O}~D\trp}^2 = \nn{~D\bar{~O}\trp-~R\trp}^2,\]
as desired.
Moreover, the minimization problem decouples over the rows $\bar{~o}^a\in\mathbb{R}^{rp'}$ of $\bar{~O}$ and $~r^a\in\mathbb{R}^{N_t N_s}$ of $~R$, $a = 1,\ldots, r$, into $r$ independent least-squares problems, i.e.,
\begin{align}\label{eq:vectorizedpOpInf-unstructured}
    \argmin_{\bar{~O}}\,\ell\lr{\bar{~O}} = \left\{\argmin_{\bar{~o}^a} \, \nn{~D\bar{~o}^a - ~r^a}^2\right\}_{a=1}^r.
\end{align}
where each sub-problem has $N_t N_s$ data points and $rp'$ unknowns.
\end{proof}
\end{proposition}

\Cref{prop:tsrlstsqprop}
shows that the least-squares problem~\cref{eq:least-squares-formulation} is equivalent to the system of normal equations~\cref{eq:normal-equations}. Indeed, $~D\trp~D=\hat{~B}$, which can be seen by
considering the tensor $\tens{X}\in\mathbb{R}^{r\times p\times r \times p}$ with components
\begin{align*}
    \mathrm{X}_{axby}
    = \sum_{\alpha,s}\mathrm{K}_{a\alpha xs}\mathrm{K}_{b\alpha ys}
    = \sum_{\alpha,s}(\hat{~Y}_s)_{a\alpha}(~\nu_s)_x(\hat{~Y}_s)_{b\alpha}(~\nu_s)_y
    = \sum_s (~\nu_s)_x(\hat{~Y}_s\hat{~Y}_s\trp)_{ab}(~\nu_s)_y.
\end{align*}
The matrix $\hat{~B}$ is a partial vectorization of this quantity; specifically,
\begin{align*}
    \hat{~B}
    = \rvec_{12}\cvec_{34}\lr{\mathrm{perm}_{12}\,\tens{X}}
= \cvec_{12}\cvec_{34}\,\tens{X}
    =~D\trp~D,
\end{align*}
since $\rvec_{ij}\mathrm{perm}_{ij} = \cvec_{ij}$, where $\mathrm{perm}_{ij}$ denotes a permutation of indices $i$ and $j$.
This implies that the condition number of $~D$ is the square root of the condition number of $\hat{~B}$, hence the direct minimization of~\cref{eq:least-squares-formulation} via, e.g., QR factorization of $~D$, is more numerically viable than solving~\cref{eq:normal-equations}.
Additionally, iterative approaches based on low-rank factorizations of $~D$ can be used to implicitly regularize and efficiently solve the inference problem. A na\"{i}ve implementation of this procedure using direct minimization is presented in \Cref{alg:lstsq_unstructured}.

\begin{algorithm}[t]
\caption{A NumPy implementation for inferring a tensor $\bar{\tens{T}}$ using the least-squares problem \cref{eq:least-squares-formulation} from \Cref{prop:tsrlstsqprop}.}
\label{alg:lstsq_unstructured}
\vspace{.25cm}
\begin{minted}{python}
import numpy as np

def infer_Tbar_with_lstsq(nus, Ys, Zs):
    p, Ns = nus.shape
    r, Nt, Ns = Ys.shape  # or Zs.shape

    K  = np.einsum("ias,xs->iaxs", Ys, nus)
    Dt = K.transpose((0, 2, 1, 3)).reshape((r*p, Nt*Ns), order="F")
    R  = Zs.reshape((r, Nt*Ns), order="F")

    Obar = np.linalg.lstsq(Dt.T, R.T)[0].T
    return Obar.reshape((r, r, p), order="F")
\end{minted}
\vspace{.25cm}
\end{algorithm}

\begin{remark}[Regularization]\label{rem:regularization}
Note that the optimization problems underpinning \Cref{alg:normal_eqns_unstructured} and \Cref{alg:lstsq_unstructured} are solved without regularization. While the selection strategies of \cite{mcquarrie2021regopinf,mcquarrie2023popinf} could be readily incorporated to induce greater numerical stability in the learned operators, our preference is to apply structural constraints when possible, as these are inherently interpretable and implicitly regularize the learning problem through the restriction biases they place on the space of allowable solutions.  \Cref{sec:Hamiltonian} will provide one example of this approach tailored to operators with symmetry.
\end{remark}

The expressions in \Cref{prop:tsrlstsqprop} involving $~D$, $\bar{~O}$, and $~R$ exactly correspond to the state-linear parts of the inference problem~(2.13) in previous work \cite{mcquarrie2023popinf}. The advantage of \Cref{alg:lstsq_unstructured}, in addition to the brevity of the implementation, is that the expression for $~D$ in \cref{prop:tsrlstsqprop} provides insight into when the inference of $\bar{\tens{T}}$ in \cref{eq:genericminprob} has a unique solution.
This is evidenced by the following concise result, analogous to \cite[Theorem~2.3]{mcquarrie2023popinf}.

\begin{corollary}\label{cor:uniqueness}
    Let $~D$, $~R$, and $\ell(\bar{~O})$ be as in \cref{prop:tsrlstsqprop}. Let $~\Theta\in\mathbb{R}^{N_s\times p'}$ be the matrix with entries $\Theta_{sx} = (~\nu_s)_x$. The inference problem $\argmin\,\ell\lr{\bar{~O}}$ in \eqref{eq:vectorizedpOpInf-unstructured} has a unique solution $\bar{~O}=\cvec_{23}\,\bar{\tens{T}}$ if and only if both $~\Theta$ and $(\cvec_{23}\,\hat{\tens{Y}})\trp\in\mathbb{R}^{N_tN_s\times r}$  have full column rank.
\end{corollary}
\begin{proof}
    Similar to \Cref{prop:tsrlstsqprop}, write
    $~D = \lr{\cvec_{13}\cvec_{24}\,\tens{K}}\trp \in\mathbb{R}^{N_tN_s\times rp'}$, where the tensor $\tens{K}\in\mathbb{R}^{r\times N_t\times p'\times N_s}$ has components $\mathrm{K}_{b\alpha xs} = (\hat{~Y}_s)_{b\alpha}\Theta_{sx}$.
    Note that $\tens{~K}$ ``factors'' as the combination tensor-Hadamard product of  $\hat{\tens{Y}}$ and $~1\otimes ~\Theta\trp$: it is an outer product in the indices $b,x$ and a Hadamard product in the indices $\alpha,s$.  This implies that, for any $~v\in\mathbb{R}^{rp'}$ in the kernel of $~D$, (Einstein summation assumed)
    \begin{align*}
        \lr{\cmat_{12}\,~D~v}_{\alpha s} &= \lr{\cmat_{12}\,\left[\cmat_{34}\,~D:\cmat_{12}\,~v^\intercal\right]}_{\alpha s} =  \hat{\mathrm{Y}}_{b\alpha s}\Theta_{sx}\lr{\cmat_{12}\,~v}^{bx} \\
        &= \hat{\mathrm{Y}}_{b\alpha s}\Theta_{sx}\lr{\mathrm{v}^{ay}~e_a~e_y\trp}^{bx} = \mathrm{v}^{ay}\hat{\mathrm{Y}}_{b\alpha s}\Theta_{sx}\delta^b_a\delta^x_y \\
        &= \lr{\hat{\mathrm{Y}}_{a\alpha s} \mathrm{v}^{ay}}\Theta_{sy} =  \hat{\mathrm{Y}}_{a\alpha s} \lr{\mathrm{v}^{ay}\Theta_{sy}} = 0.
    \end{align*}
    If $~\Theta$ has full column rank, the first equality on the last line shows that $(\cvec_{23}\,\hat{\tens{Y}})\trp$ is rank-deficient.  Similarly, if $(\cvec_{23}\,\hat{\tens{Y}})\trp$ has full column rank, the second equality on the last line shows that $~\Theta$ is rank-deficient. Therefore, $~D$ has full column rank whenever these matrices do, while if either is rank deficient, then so is $~D$.
\end{proof}

\begin{remark}[Nonlinear terms]
    \label{remark:nonlinear}
    Tensor parametric OpInf can be extended in a straightforward manner to account for quadratic (or higher-order) polynomial terms $\tens{H}(~\mu)\!:\!\lr{~q\otimes~q}$ with affine parameter dependence appearing in~\cref{eq:ODE}, such as those seen in \cite{mcquarrie2023popinf}. In this case, a tensor $\tens{T}$ satisfying $\tens{H}(~\mu) = \tens{T}~\mu'$ must have order 4, as $\tens{H}$ (if not matricized) is of order 3. Since higher-order terms are not necessary for the numerical examples presented later, such extensions are left for future work.
\end{remark}

The results of this section demonstrate how solutions to the tensor parametric OpInf problem are global minima and how the tensor reformulation clearly and concisely recovers previous parametric OpInf formulations.
Before discussing how this approach also enables a novel extension to systems with Hamiltonian structure, an example is presented to illustrate its practical utility.

\subsection{Example: Heat equation}
\label{sec:heat-equation}
Let $\Omega\subset\mathbb{R}^d$ where $d\in\{1,2,3\}$ and consider the following initial value problem with homogeneous Dirichlet boundary conditions,
\begin{align}\label{eq:heat}
    \left\{\phantom{-}
    \begin{aligned}
        &\dot{q}(~x,t) = \nabla\cdot\lr{c(~x,~\mu)\nabla q(~x,t)}, &~x \in \Omega \times (0,t_f],\\
        &q(~x,0 ) = q_0(~x),  &~x \in \Omega,\\
        & q(~x,t) = 0, &~x \in \partial\Omega \times (0,t_f],
    \end{aligned}
\right.
\end{align}
where $q:\Omega\times(0,t_f)\to\mathbb{R}$ is the continuous system state, $c:\Omega\times\mathbb R^{p} \to \mathbb R$ is the parameterized thermal conductivity coefficient, and $q_0:\Omega\to\mathbb{R}$ is a given initial condition. Note that $q$ depends implicitly on the conductivity parameters $~\mu$. Let $\Omega = \bigcup_{i=1}^{p}\Omega_i$ be a non-overlapping decomposition of the spatial domain and consider the piecewise-continuous $c(~x,~\mu)$ given by
\begin{align*}
    c(~x,~\mu)
    = \left[~x\in\Omega_1\right]\,\mu_1
    + \left[~x\in\Omega_2\right]\,\mu_2
    + \ldots
    + \left[~x\in\Omega_p\right]\,\mu_p,
    \qquad
    ~\mu = [\mu_1\,\,\mu_2\,\,\cdots\,\,\mu_p]\trp,
\end{align*}
where $[S]$ is the Iverson bracket \cite{knuth1992notes}, i.e., the indicator function of the statement $S$.

A finite-element-based FOM for~\cref{eq:heat} can be readily constructed using a continuous Galerkin discretization in space.
Let $\Omega_h := \{K_i\}_{i=1}^{N_E}$ be a conforming triangulation of the domain $\Omega$ containing $N_E$ elements.
The Galerkin weak formulation of \cref{eq:heat} with this spatial discretization is the following:
find $\dot q_h \in V_h$ such that for all $v_h \in V_h$,
\begin{align}\label{eq:heatfem}
    (\dot q_h, v_h)_{\Omega_h} = -(c(\cdot, ~\mu)\nabla q_h, \nabla v_h)_{\Omega_h},
\end{align}
where {$(\cdot,\cdot)_{\Omega_h}$ denotes the $L^2$ inner product on $\Omega_h$ and} $V_h$ is the $H^1(\Omega)$-conforming finite element space
\begin{align*}
    V_h := \{ v_h \in H^1(\Omega) : \left.v_h\right|_K\in P_1(K), \enspace \forall K \in \Omega_h;  \left.v_h\right|_{\partial \Omega} = 0\},
\end{align*}
with $P_1(K)$ representing the space of linear polynomials on element $K$.
Given a basis $\{\phi_i\}_{i=1}^{N}$ for $V_h$, inserting the approximation $q(~x,t) \approx q_h(~x,t) := \sum_{i=1}^{N}q_i(t)\phi(~x)$ into~\cref{eq:heatfem} leads to the system of $N$ ODEs $~M\dot{~q}=~A(~\mu)~q$ with matrix entries
\begin{subequations}
\begin{align}
    \label{eq:heatfom-indices}
    \begin{aligned}
    (~M)_{ij}
    &= (\phi_i, \phi_j)_{\Omega_h},
    \\
    (~A(~\mu))_{ij}
    &= - (c(\cdot,~\mu)\nabla\phi_i,\nabla\phi_j)_{\Omega_h}
    = -\sum_{x=1}^{p}\mu_x(\nabla\phi_i, \nabla\phi_j)_{\Omega_x}.
    \end{aligned}
\end{align}
The corresponding tensorized FOM $~M\dot{~q} = (\tens{T}~\mu')~q$ has tensor entries
\begin{align}
    \label{eq:heatfom-indices-tensor}
    \mathrm{T}_{ijk}
    = \lr{[~x\in \Omega_k]\nabla\phi_i, \nabla\phi_j}_{\Omega_h}
    = \lr{\nabla\phi_i, \nabla\phi_j}_{\Omega_k}.
\end{align}
\end{subequations}
Given snapshot matrices $~Q_1,\ldots,~Q_{N_s}$, \Cref{alg:normal_eqns_unstructured} or \Cref{alg:lstsq_unstructured} can now be used with data $~Y_s = ~U\trp~M~Q_s$ and $~Z_s = ~U\trp~MD_t(~Q_s)$ to infer a reduced tensor $\bar{\tens{T}}\approx\hat{\tens{T}}$, producing a parametric OpInf ROM $\dot{\hat{~q}} = \lr{\bar{\tens{T}}~\mu}\hat{~q}$.

\begin{remark}[Tensor symmetry]
The full-order tensor $\tens{T}$ in \eqref{eq:heatfom-indices-tensor} is symmetric in its first two indices, reflecting the fact that the heat equation \eqref{eq:heat} is an $L^2$-gradient flow. However, \cref{alg:normal_eqns_unstructured} and \cref{alg:lstsq_unstructured} do not enforce this symmetry as a constraint, i.e., the learned tensor $\bar{\tens{T}}$ is not necessarily symmetric in its first two indices.
The next section develops a tensor parametric OpInf algorithm with additional constraints that preserve symmetric (or anti-symmetric) tensor structure.
\end{remark}

 \section{Tensor parametric operator inference for Hamiltonian systems}\label{sec:Hamiltonian}
The concise tensor formulation developed in \Cref{sec:generic} enables a principled extension to parametric Hamiltonian systems, where preservation of additional algebraic structure in the learned tensor $\bar{\tens{T}}\approx\hat{\tens{T}}$ is necessary in order to enable the energy-conservative and symplectic characteristics of Hamiltonian dynamics. \Cref{sec:hamiltonian-background} summarizes the key aspects of Hamiltonian model reduction, and \Cref{sec:hamiltonian-algorithm} develops an analog to \Cref{alg:normal_eqns_unstructured} that accounts for the additional desired structure. An application to wave phenomena is outlined in \Cref{sec:wave}.

\subsection{Hamiltonian model reduction}\label{sec:hamiltonian-background}
Canonical Hamiltonian systems are defined by the joint state variable $~y = [~q\trp\enspace ~p\trp]\trp\in\mathbb{R}^{2N}$, a Hamiltonian functional $H:\mathbb{R}^{2N}\to\mathbb{R}$ representing a notion of energy, and a skew-adjoint block matrix $
    ~J
= -~J\trp
    \in \mathbb{R}^{2N\times 2N}
$ encoding the symplectic structure:
\begin{align}\label{eq:ClassicHamiltonian}
    \dot{~y} = \{~y, H(~y)\} = ~J\nabla H(~y),
    \qquad
    ~J = \left[\begin{array}{cc}
         ~0 & ~I \\
        -~I & ~ 0
    \end{array}\right].
\end{align}
Here, $\{F,G\} = \nabla F\cdot~J\nabla G$ is a particular Lie algebra realization on functions known as the (canonical) Poisson bracket, and notation has been slightly abused by vectorizing this bracket over the components of the state $~y$.
A core feature of Hamiltonian systems is that their solutions conserve energy for all time, which can be seen by computing the rate of change in $H$,
\begin{align*}
    \dot{H}(~y)
    = \dot{~y}\trp\nabla H(~y)
&= \nabla H(~y)\trp~J\trp\nabla H(~y)
    = -\nabla H(~y)\trp~J\nabla H(~y)
    \\
    &= \lr{\nabla H(~y)\trp~J\trp\nabla H(~y)}\trp
    = \nabla H(~y)\trp~J\nabla H(~y),
\end{align*}
which shows that $\dot{H}(~y) = 0$ for all $~y$.
This reflects the fact that Hamiltonian systems yield symplectic gradient flows whose solutions are confined to individual level sets of $H$ \cite{Marsden_BOOK_1998}.

When Hamiltonian systems are discretized with the finite element method, it is typical to incorporate a symmetric positive definite mass matrix $~M\in\mathbb{R}^{2N\times 2N}$, leading to models of the form
\begin{align}\label{eq:massedHamiltonian}
    ~M\dot{~y}
    = ~L~M^{-1}\nabla H(~y),
    \qquad
    ~L\trp = -~L.
\end{align}
To expose the Hamiltonian structure of this system, define $\bar{~J} = ~M^{-1}~L$ and write
\begin{align}\label{eq:mgradHamiltonian}
    \dot{~y}
    = ~M^{-1}~L~M^{-1}\nabla H(~y)
    = \bar{~J}\,\Mnabla H(~y),
\end{align}
where $\Mnabla H = ~M^{-1}\nabla H$ is the gradient of $H$ with respect to the weighted inner product $~a^*~b=~a\trp~M~b$ induced by $~M$ (see \Cref{app:morthonormal}).
The adjoint of $\bar{~J}$ in this inner product is given by
\begin{align*}
    \bar{~J}^{*}
    = (~M^{-1}~L)^{*}
    = ~M^{-1}(~M^{-1}~L)\trp ~M
    = ~M^{-1}~L\trp~M^{-\mathsf{T}} ~M
    = -~M^{-1}~L =-\bar{~J}.
\end{align*}
Therefore, $\bar{~J}$ is skew-adjoint with respect to $~M$, hence \eqref{eq:mgradHamiltonian} is a Hamiltonian system in the $~M$-weighted inner product.  Additionally, $H$ is conserved along solutions:
\begin{align*}
    \dot{H}(~y)
    = \dot{~y}^*\Mnabla H(~y)
    = \lr{\bar{~J}\,\Mnabla H(~y)}^*\Mnabla H(~y)
    = -\Mnabla H(~y)^*\lr{\bar{~J}\,\Mnabla H(~y)} = 0.
\end{align*}

\begin{remark}[Block diagonal mass]
It is often the case that $~L = ~J~M$ and that the mass matrix $~M$ can be written as
$
    ~M
    = \mathrm{blockdiag}(~M_q,~M_q),
$
where $~M_q\in\mathbb{R}^{N\times N}$ is a symmetric positive definite mass matrix corresponding to $~q$ and $~p$.
In this case, $~J~M = ~M~J$ and, hence, $\bar{~J} = ~M^{-1}~L = ~M^{-1}~J~M = ~J$.
For the remainder of this section, $~J$ is written instead of $\bar{~J}$ and it is assumed that $~J^{*} = - ~J$.
\end{remark}

The effectiveness of Hamiltonian systems in modeling non-dissipative phenomena has motivated a significant amount of work into their model reduction \cite{Afkham:2017,Buckfink:2021,GRUBER2023116334,gruber2024vc,hesthaven2022structure,peng2016symplectic,Sharma:2023,sharma2022hamiltonian}.  The key observation is that Hamiltonian systems are governed by a variational principle, and violation of this principle at the reduced-order level leads quickly to unphysical results, especially in the predictive regime where ROMs are intended to provide computational savings (see, e.g., \cite[Figure 10]{GRUBER2023116334}).
Effective ROMs for Hamiltonian systems---whether intrusive or non-intrusive---must therefore preserve the Hamiltonian structure, meaning that the reduced dynamics should be governed by an appropriate reduced Hamiltonian function $\hat{H}:\mathbb{R}^{2r}\to\mathbb{R}$ and reduced skew-adjoint matrix $\hat{~J}:\mathbb{R}^{2r \times 2r}$.\footnote{In fact, $\hat{~J}=\hat{~J}(~y)$ can generally be a matrix field, and not all skew-adjoint $\hat{~J}$ define valid Hamiltonian systems.  More details can be found in, e.g., \cite{GRUBER2023116334,gruber2024vc}.}

Structure preservation is equally critical in ROMs for \emph{parametric} Hamiltonian systems, i.e., those governed by a parameterized Hamiltonian $H = H(~y,~\mu)$. In this case, the ROM must be defined in terms of the gradient of an appropriate reduced parameterized Hamiltonian.
To this end, consider the Hamiltonian
\begin{align}
    \label{eq:theHamiltonian}
    H(~y,~\mu)
    = \frac{1}{2}~y^*~A(~\mu)~y + f(~y,~\mu),
\end{align}
defined in terms of a nonlinear function $f:\mathbb{R}^{2N}\times\mathbb{R}^{p}\to\mathbb{R}$ and a matrix field $~A:\mathbb{R}^{p}\to\mathbb{R}^{2N\times 2N}$ with affine parametric dependence.
The (canonical) Hamiltonian system generated by $H$ takes the form
\begin{align}
    \label{eq:mgradCanonicalHamiltonian}
    \dot{~y}
    = ~J\Mnabla H(~y,~\mu)
    = ~J(~A(~\mu)~y+\Mnabla f(~y,~\mu)).
\end{align}
It may be assumed that $~A(~\mu)$ is self-adjoint, that is, $~A(~\mu) = ~A(~\mu)^{*} = ~M^{-1}~A\trp(~\mu)~M$, since only the self-adjoint part of $~A(~\mu)$ contributes to the value of $H$ (see \Cref{thm:Madjoint,thm:selfadjoint}).
Since $~A$ depends affinely on $~\mu$, it follows as before that $~A(~\mu) = \tens{T}~\mu'$ for some constant tensor $\tens{T}\in\mathbb{R}^{2N\times 2N\times p'}$ and a vector $~\mu' = ~\theta(~\mu)\in\mathbb{R}^{p'}$.
Additionally, because $~A(~\mu)$ is self-adjoint, it follows that $\tens{T}$ is self-adjoint in its first two indices, denoted $\tens{T}^{*} = \tens{T}$.
The key to the tensor parametric Hamiltonian OpInf developed here will be an inference procedure analogous to \cref{alg:normal_eqns_unstructured} which explicitly takes this symmetry into account.

To see the constraints that Hamiltonian structure places at the reduced-order level, consider \cref{eq:mgradCanonicalHamiltonian} with $f\equiv 0$ and a known initial condition:
\begin{align}\label{eq:HODE}
    \dot{~y}
= ~J~A(~\mu)~y
    = ~J \lr{\tens{T}~\mu'}~y,
    \qquad ~y(0,~\mu) = ~y_0(~\mu).
\end{align}
Given a trial space basis $~U\in\mathbb{R}^{2N\times 2r}$ which is metric-compatible in the sense that $~U^*~U = \hat{~M}~U\trp~M~U = ~I$ for some $\hat{~M}\in\mathbb{R}^{2r\times 2r}$ defining an inner product on the reduced coordinates (typically $\hat{~M} = ~I$), making the approximation $\tilde{~y} \coloneqq ~U\hat{~y} \approx ~y$ with $\hat{~y} = \hat{~y}(t,~\mu)\in\mathbb{R}^{2r}$ and performing $~M$-orthogonal projection in \cref{eq:HODE} results in the following (intrusive) ROM:
\begin{align}\label{eq:naiveROM}
    \dot{\hat{~y}} = ~U^* ~J \lr{\tens{T}~\mu'} ~U\hat{~y},
    \qquad \hat{~y}(0,~\mu) = ~U^* ~y_0(~\mu).
\end{align}
Unfortunately, the ROM~\cref{eq:naiveROM} is not Hamiltonian with respect to the reduced inner product $\hat{~a}^*\hat{~b}=\hat{~a}\trp\hat{~M}\hat{~b}$ since $(~U^* ~J)^* = ~J^* ~U = -~J~U \neq -~U^* ~J$, which implies that the symplectic structure of the original system is lost.  Moreover, the natural reduced Hamiltonian functional $\hat{H}(\hat{~y},~\mu') := H(~U\hat{~y},~\mu)$ is not conserved: for any $~\mu'\in\mathbb{R}^{p'}$,
\begin{align*}
    \dot{\hat H}(\hat{~y},~\mu')
    = \dot{\hat{~y}}^*\nabla^{\hat{M}}\!\hat H(\hat{~y})
    &= \lr{~U^* ~J \lr{\tens{T}~\mu'} ~U\hat{~y}}^* ~U^* \lr{\tens{T}~\mu'}~U\hat{~y}
    \\
    &=  -\lr{\lr{\tens{T}~\mu'}~U\hat{~y}}^*~J~U~U^*\lr{\tens{T}~\mu'}~U\hat{~y}
    \neq 0.
\end{align*}
In other words, the structural and conservative properties of the FOM~\cref{eq:HODE} are not automatically inherited by the projection-based ROM~\cref{eq:naiveROM}, and solutions of \cref{eq:naiveROM} are liable to exhibit artificial energetic profiles and stability issues.

There are two primary approaches to circumventing this defect: modifying the reduced basis $~U$, or modifying the equation form~\cref{eq:naiveROM}.  The first approach was initiated by the work~\cite{peng2016symplectic}, which introduced the Proper Symplectic Decomposition (PSD) for building Hamiltonian ROMs.  This involves constructing a basis $~U\in\mathbb{R}^{2N\times 2r}$ with the desirable $~J$-equivariance property $~U^* ~J = \hat{~J}~U^*$, where $\hat{~J} = -\hat{~J}^{*} \in \mathbb{R}^{2r\times 2r}$.
The advantage of this property is that a variationally consistent and canonical Hamiltonian ROM follows directly from Galerkin projection onto the span of $~U$,
\begin{align}\label{eq:PSDROM}
    \dot{\hat{~y}}
    = ~U^* ~J\lr{\tens{T}~\mu'}~U\hat{~y}
    = \hat{~J}~U^*\lr{\tens{T}~\mu'}~U\hat{~y}
    = \hat{~J}\lr{\hat{\tens{T}}~\mu'}\hat{~y},
\end{align}
where the reduced tensor $\hat{\tens{T}} \in \mathbb{R}^{2r\times 2r\times p'}$ is (assuming $\hat{~M} = ~I$) symmetric in its first two indices, written $\hat{\tens{T}} =\hat{\tens{T}}\trp$, and $(\hat{\tens{T}}~\mu')\hat{~y}=\nabla \hat{H}\lr{\hat{~y},~\mu}$ is the (Euclidean) gradient of the reduced Hamiltonian,
\begin{align}
    \label{eq:reduced-hamiltonian}
    \hat{H}(~y,~\mu') = (1/2)\hat{~y}\trp(\hat{\tens{T}}~\mu')\hat{~y}.
\end{align}
This intrusively constructed ROM exactly preserves the reduced Hamiltonian $\hat{H}$, since
\begin{align*}
    \dot{\hat{H}}(\hat{~y},~\mu')
    = \dot{\hat{~y}}\trp \nabla \hat{H}(\hat{~y},~\mu')
    = \hat{~y}\trp(\hat{\tens{T}}~\mu')\trp \hat{~J}\trp(\hat{\tens{T}}~\mu') \hat{~y}
    = - \hat{~y}\trp(\hat{\tens{T}}~\mu')\trp \hat{~J}(\hat{\tens{T}}~\mu') \hat{~y} = 0.
\end{align*}
Mirroring this structure, an appropriate OpInf ROM for \cref{eq:HODE} that incorporates the dynamics induced by the Hamiltonian used above is
\begin{align}
    \label{eq:hamiltonianopinfROM}
    \dot{\hat{~y}}
    = \hat{~J}\lr{\bar{\tens{T}}~\mu'}\hat{~y},
    \qquad \hat{~y}(0,~\mu) = ~U^*~y_0(~\mu),
\end{align}
as long as the inferred tensor $\bar{\tens{T}}$ is constrained to satisfy the symmetry property $\bar{\tens{T}}\trp = \bar{\tens{T}}$.
Similar to the previous computation, it is straightforward to verify that the ROM~\cref{eq:hamiltonianopinfROM} exactly preserves the reduced Hamiltonian $\bar{H}(\hat{~y},~\mu) = (1/2)\hat{~y}\trp(\bar{\tens{T}}~\mu')\hat{~y}$, which approximates $\hat{H}(\hat{~y},~\mu)$.
The next section extends the work of \Cref{sec:generic} to infer Hamiltonian ROMs of the form \cref{eq:hamiltonianopinfROM} in a non-intrusive fashion.

The alternative to designing the basis matrix $~U$ to satisfy the potentially restrictive $~J$-equivariance property is to modify the model form~\cref{eq:naiveROM}.
An effective way of accomplishing this while retaining variational consistency was presented in~\cite{gruber2024vc}, which applies Petrov--Galerkin projection onto the test space $~J~U$.  Applying $(~J~U)^*$ to both sides of \cref{eq:HODE} yields
\begin{align}\label{eq:conROM}
    \dot{\hat{~y}} = \hat{~J}\invtrp\nabla \hat{H}(\hat{~y},~\mu')
    = \hat{~J}\invtrp \lr{\hat{\tens{T}}~\mu'}\hat{~y},
\end{align}
which satisfies $\hat{\tens{T}}\trp=\hat{\tens{T}}$ and is symplectic since $\hat{~J}\invtrp = \lr{~U^*~J~U}\invtrp = -\hat{~J}^{-1}$ is constant in the state.\footnote{The fact that $\hat{~J}\invtrp$ is constant implies that the Jacobi identity reduces to the symmetry condition $\hat{~J}\invtrp = -\hat{~J}^{-1}$.}
Like~\cref{eq:PSDROM}, this ROM conserves the reduced Hamiltonian $\hat{H}\lr{\hat{~y},~\mu}$.

\subsection{Enforcing symmetries in tensor parametric operator inference}\label{sec:hamiltonian-algorithm}
A key feature of the canonical Hamiltonian ROMs \cref{eq:PSDROM}--\cref{eq:conROM} is the preservation of the symmetry condition $\hat{~A}(~\mu)\trp = \hat{~A}(~\mu)$ governing the reduced Hamiltonian $\hat{H}$.  In the case of affine parameter dependence, this reduces to the symmetry constraint $\hat{\tens{T}}\trp = \hat{\tens{T}}$ on the first two indices of the tensor representation $\hat{\tens{T}}~\mu' = \hat{~A}(~\mu)$.
The natural extension of the tensor OpInf problem \cref{eq:rODE-opinf} to this setting is the following constrained regression problem:
\begin{align}
    \label{eq:constrainedminimization}
    \argmin_{\bar{\tens{T}}}\frac{1}{2}\sum_{s=1}^{N_s}\nn{\dot{\hat{~Q}}_s - \hat{~J}\lr{\bar{\tens{T}}~\mu_s'}\hat{~Q}_s }^2
    \quad
    \mathrm{s.t.}
    \quad
    \bar{\tens{T}}\trp = \bar{\tens{T}},
\end{align}
where, as in the setting of \Cref{sec:generic}, $\hat{~Q}_s,\dot{\hat{~Q}}_s\in\mathbb{R}^{2r\times N_t}$ are reduced state snapshot matrices and (approximate) reduced state time derivative matrices corresponding to parameter samples $~\mu_s\in\mathbb{R}^{p'}$, $s=1,\ldots,N_s$.
The following theorem provides a compact strategy for solving this problem.

\begin{theorem}\label{thm:popinf-Hamiltonian}
    Let $\hat{~X}_s\in\mathbb{R}^{r \times r}, \hat{~Y}_s\in\mathbb{R}^{r \times N_t}, \hat{~Z}_s\in\mathbb{R}^{r \times N_t}$, and $~\nu_s\in\mathbb{R}^{p'}$ for $s = 1,\ldots,N_s$. Then solutions to the convex minimization problem
    \begin{align*}
            \argmin_{\bar{\tens{T}}}\frac{1}{2}\sum_{s=1}^{N_s}\nn{\hat{~Z}_s-\hat{~X}_s\lr{\bar{\tens{T}}~\nu_s}\hat{~Y}_s}^2 \quad \mathrm{s.t.}\quad\bar{\tens{T}}~\eta = \pm \lr{\bar{\tens{T}}~\eta}\trp \quad \forall~\eta\in\mathbb{R}^{p'},
    \end{align*}
    can be computed by solving the linear system
    \begin{align}\label{e:HOPINF}
        \begin{aligned}
        \sum_{s=1}^{N_s}\left[~\nu_s~\nu_s\trp\otimes_K\lr{\hat{~X}_s\trp\hat{~X}_s\baroplus_K\hat{~Y}_s\hat{~Y}_s\trp}\right]\,&\vect\,\bar{\tens{T}} \\= \sum_{s=1}^{N_s}~\nu_s\,\otimes_K\,&\cvec_{12}\lr{\hat{~X}_s\trp\hat{~Z}_s\hat{~Y}_s\trp\pm \hat{~Y}_s\hat{~Z}_s\trp\hat{~X}_s},
        \end{aligned}
    \end{align}
where $\vect\,\bar{\tens{T}}=\cvec_{12}\lr{\cvec_{12}\,\bar{\tens{T}}}$ is a total column-wise flattening of the tensor $\bar{\tens{T}}\in\mathbb{R}^{r\times r \times p'}$, $~X\otimes_K~Y = \rvec_{13}\rvec_{24}\lr{~X\otimes~Y}$ denotes the usual Kronecker product of matrices, and $~X\baroplus_K~Y \coloneqq ~X\otimes_K~Y+~Y\otimes_K~X$ denotes a symmetric Kronecker sum.
\end{theorem}

\begin{proof}
Let $~\Lambda\in\mathbb{R}^{r\times r\times p'}$ be a tensor of Lagrange multipliers and define the Lagrangian
\begin{align*}
    L\lr{\bar{\tens{T}},~\Lambda} = \frac{1}{2}\sum_{s=1}^{N_s}\nn{\hat{~Z}_s-\hat{~X}_s\lr{\bar{\tens{T}}~\nu_s}\hat{~Y}_s}^2 + \IP{~\Lambda}{\bar{\tens{T}}\mp \bar{\tens{T}}\trp} ,
\end{align*}
where, as before, $\bar{\tens{T}}\trp$ indicates transposition of the first two indices.
Differentiating this expression, it follows that
\begin{align*}
    dL\lr{\bar{\tens{T}},~\Lambda} &= -\sum_{s=1}^{N_s}\IP{\hat{~Z}_s-\hat{~X}_s\lr{\bar{\tens{T}}~\nu_s}\hat{~Y}_s}{\hat{~X}_s\lr{d\bar{\tens{T}}\,~\nu_s}\hat{~Y}_s} \\
    &\qquad \qquad+ \IP{d~\Lambda}{\bar{\tens{T}}\mp \bar{\tens{T}}\trp} + \IP{~\Lambda}{d\bar{\tens{T}}\mp d\bar{\tens{T}}}\trp \\
    &= \IP{d\bar{\tens{T}}}{~\Lambda\mp ~\Lambda\trp - \sum_{s=1}^{N_s}\hat{~X}_s\trp\lr{\hat{~Z}_s -\hat{~X}_s\lr{\bar{\tens{T}}~\nu_s}\hat{~Y}_s}\hat{~Y}_s\trp\otimes~\nu_s} \\
    &\qquad\qquad+ \IP{d\hat{~\Lambda}}{\bar{\tens{T}}\mp \bar{\tens{T}}\trp} \\
    &= \IP{d\bar{\tens{T}}}{\nabla_{\bar{\tens{T}}}L} + \IP{d~\Lambda}{\nabla_{~\Lambda}L}.
\end{align*}
Setting these gradients to zero yields the Euler--Lagrange equations for $L$,
\begin{subequations}
\begin{align}
    ~\Lambda\mp ~\Lambda\trp
    &= \sum_{s=1}^{N_s}\hat{~X}_s\trp\lr{\hat{~Z}_s-\hat{~X}_s\lr{\bar{\tens{T}}~\nu_s}\hat{~Y}_s}\hat{~Y}_s\trp\otimes~\nu_s \label{eq:Tlag}
    \\
    \bar{\tens{T}}\mp\bar{\tens{T}}\trp &= ~0. \label{eq:Tsym}
\end{align}
\end{subequations}
In particular, the Lagrange multiplier $~\Lambda$ can be eliminated by symmetrizing/anti-symmetrizing \cref{eq:Tlag}, leading to
\begin{align*}
    \begin{aligned}
        \sum_{s=1}^{N_s}\left[\hat{~X}_s\trp\lr{\hat{~Z}_s-\hat{~X}_s\lr{\bar{\tens{T}}~\mu_s}\hat{~Y}_s}\hat{~Y}_s\trp\right.\hspace{5cm} \\
        \pm \left. \hat{~Y}_s \lr{\hat{~Z}_s-\hat{~X}_s\lr{\bar{\tens{T}}~\nu_s}\hat{~Y}_s}\trp\hat{~X}_s\right]\otimes~\nu_s = ~0,
    \end{aligned}
\end{align*}
which, upon rearranging and applying~\cref{eq:Tsym}, yields the tensor equation
\begin{align*}
    \begin{aligned}
        \sum_{s=1}^{N_s} \lr{\hat{~X}_s\trp\hat{~X}_s\lr{\bar{\tens{T}}~\nu_s}\hat{~Y}_s\hat{~Y}_s\trp + \hat{~Y}_s\hat{~Y}_s\trp \lr{\bar{\tens{T}}~\nu_s}  \hat{~X}_s\trp\hat{~X}_s}\otimes~\nu_s \hspace{2cm} \\ =\sum_{s=1}^{N_s}\lr{\hat{~X}_s\trp\hat{~Z}_s\hat{~Y}_s\trp\pm \hat{~Y}_s\hat{~Z}_s\trp\hat{~X}_s}\otimes~\nu_s.
    \end{aligned}
\end{align*}
This is a Sylvester equation in the first two indices of $\bar{\tens{T}}$, so applying the ``vec trick''~\cite{schacke2004kronecker}
yields a partial vectorization of dimension $r^2\times p'$,
\begin{align*}
    \begin{aligned}
        \sum_{s=1}^{N_s}\lr{\hat{~X}_s\trp\hat{~X}_s\baroplus_K\hat{~Y}_s\hat{~Y}_s\trp}\,\cvec_{12}\lr{\bar{\tens{T}}}~\nu_s~\nu_s\trp =\sum_{s=1}^{N_s}\cvec_{12}\lr{\hat{~X}_s\trp\hat{~Z}_s\hat{~Y}_s\trp\pm\hat{~Y}_s\hat{~Z}_s\trp\hat{~X}_s}~\nu_s\trp.
    \end{aligned}
\end{align*}
Vectorizing again finally yields the equivalent matrix-vector system in the $r^2p'$ unknowns of $\vect\,\bar{\tens{T}}\coloneqq \cvec_{12}\lr{\cvec_{12}\,\bar{\tens{T}}}$,
\begin{align*}
    \begin{aligned}
        \sum_{s=1}^{N_s}\left[~\nu_s~\nu_s\trp\otimes_K\lr{\hat{~X}_s\trp\hat{~X}_s\baroplus_K\hat{~Y}_s\hat{~Y}_s\trp}\right]\,\vect\,\bar{\tens{T}}\hspace{4cm}\\
        = \sum_{s=1}^{N_s}~\nu_s\otimes_K\cvec_{12}\lr{\hat{~X}_s\trp\hat{~Z}_s\hat{~Y}_s\trp\pm \hat{~Y}_s\hat{~Z}_s\trp\hat{~X}_s},
    \end{aligned}
\end{align*}
since $\cvec_{12}\lr{~u~v\trp} = \rvec_{12}\lr{~v~u\trp} = ~v\otimes_K~u$.
\end{proof}

Solving \cref{e:HOPINF} with $~X_s = \hat{~J}$, $~Y_s = \hat{~Q}_s$, $~Z_s = \dot{\hat{~Q}}_s$, and $~\nu_s = ~\mu_s$ results in a symmetric tensor $\bar{\tens{T}}$ for defining the ROM~\cref{eq:hamiltonianopinfROM}.
An example implementation for setting up and solving \cref{e:HOPINF} is provided in \Cref{alg:normal_eqns_Hamiltonian}, which relies on a basic result about the exchange of Kronecker products and vectorization (see \cref{lem:slicktrick}).

\begin{algorithm}[t]
\caption{A NumPy/SciPy implementation for inferring a structured tensor $\bar{\tens{T}}=\pm\bar{\tens{T}}\trp$ using the linear system \cref{e:HOPINF} derived in \Cref{thm:popinf-Hamiltonian}.}
\label{alg:normal_eqns_Hamiltonian}
\vspace{.25cm}
\begin{minted}{python}
import numpy as np
import scipy.linalg as la

def infer_Tbar_with_symmetry(nus, Xs, Ys, Zs, symmetric: bool = True):
    p, Ns = nus.shape
    r, Nt, Ns = Ys.shape  # or Zs.shape

    XsTXs = np.einsum("kis,kjs->ijs", Xs, Xs)
    YsYsT = np.einsum("iks,jks->ijs", Ys, Ys)
    nusnusT = np.einsum("xs,ys->xys", nus, nus)

    Btsr = np.einsum("xys,ijs,kls->xyijkl", nusnusT, XsTXs, YsYsT)
    Btsr += Btsr.transpose(0, 1, 4, 5, 2, 3)  # tensorized Kronecker sum
    Bhat = Btsr.transpose(0, 2, 4, 1, 3, 5).reshape((r*r*p, r*r*p), order="C")

    Ctsr = np.einsum("kis,kas,jas,xs->ijx", Xs, Zs, Ys, nus)
    Ctsr += (1 if symmetric else -1) * Ctsr.transpose(1, 0, 2)
    Chat = Ctsr.flatten(order="F")

    vecT = la.solve(Bhat, Chat, assume_a="sym")
    return vecT.reshape((r, r, p), order="F")
\end{minted}
\vspace{0.25cm}
\end{algorithm}

\subsection{Example: Wave equation}\label{sec:wave}
This section describes an application with canonical Hamiltonian structure to which \Cref{alg:normal_eqns_Hamiltonian} may be applied.
Consider the following initial boundary value problem for a wave on the bounded domain $\Omega\subset\mathbb{R}^{d}$, $d\in\{1,2,3\}$, with variable wave speed:
\begin{align}\label{eq:wave}
    \left\{\phantom{-}
    \begin{aligned}
        &\ddot{y}(~x,t) = \nabla\cdot\lr{c(~x,~\mu)^2\nabla y(~x,t)}, & ~x \in \Omega \times (0,t_f],\\
        & y(~x,t) = 0, &~x \in \partial\Omega\times (0,t_f],\\
        & y(~x,0) = y_0(~x), &~x \in \Omega,\\
        & \dot{y}(~x, 0) = 0, & ~x\in \Omega.
    \end{aligned}
    \right.
\end{align}
It is straightforward to check that the system~\cref{eq:wave} is Hamiltonian with the canonical variables $q = y$, $p = \dot{y}$ and Hamiltonian functional
\begin{align*}
    H(q, p, ~\mu)
    = \frac{1}{2}\int_\Omega p^2(~x)\,d~x
    + \frac{1}{2}\int_\Omega c^2(~x,~\mu)\nn{\nabla q(~x)}^2\,d~x.
\end{align*}
Suppose $\Omega$ can be partitioned into non-overlapping subdomains $\Omega = \cup_{i=1}^p\Omega_i$ representing different materials or media with wave speeds. Then $c^{2}:\Omega\times \mathbb{R}^{p}\to\mathbb{R}$ is given by
\begin{align*}
    c(~x,~\mu)^{2}
    = \left[~x\in\Omega_1\right]\,\mu_1^2
    + \left[~x\in\Omega_2\right]\,\mu_2^2
    + \ldots
    + \left[~x\in\Omega_p\right]\,\mu_p^2,
    \qquad
    ~\mu = [\mu_1\,\,\mu_2\,\,\cdots\,\,\mu_p]\trp.
\end{align*}
With this choice, \cref{eq:wave} is a locally linear, but globally nonlinear, Hamiltonian wave equation whose solution propagates with different speeds in each subdomain.

An appropriate FOM for \cref{eq:wave} can be constructed via the Hamiltonian-preserving mixed finite element scheme from~\cite{SANCHEZ2021113843}.
Given a triangulation $\Omega_h = \{K_i\}_{i=1}^{N_E}$ of $\Omega$, define the spaces
\begin{align*}
    \begin{aligned}
        W_h &\coloneqq \{w_h \in L^2(\Omega): \left.w_h\right|_K = P_k(K), \forall K \in \Omega_h\}, \\
        V_h &\coloneqq \{~v_h \in H(\text{div},\Omega): \left.~v_h\right|_K = RT_k(K), \forall K \in \Omega_h; \left.~v_h\right|_{\partial \Omega} = 0\} ,
    \end{aligned}
\end{align*}
where $RT_k(K)$ is the order-$k$ Raviart--Thomas space $[P_k(K)]^2 \oplus ~x P_k(K) $ on element $K$ and $P_k(K)$ denotes the space of polynomials of degree $k$ on element $K$. The semi-discrete weak form can then be stated as follows: find $ (\dot q_h, \dot p_h) \in W_h \times W_h$ such that
\begin{subequations}\label{eq:mixedscheme}
\begin{align}
    \label{eq:qhpha}
    \lr{\dot q_h, w_h}_{\Omega_h}
    &= \lr{p_h,w_h}_{\Omega_h}
    \qquad\,\,\,\,
    \forall w_h \in W_h,
    \\
    \label{eq:qhphb}
    \lr{\dot p_h,w_h}_{\Omega_h}
    &= \lr{\nabla\cdot ~\sigma_h, w_h}_{\Omega_h}
    \quad
    \forall w_h \in W_h,
\end{align}
where $~\sigma_h \in V_h$ is a weak representation of $c^2\nabla q$ and therefore satisfies
\begin{align}\label{eq:sigmah}
    \lr{\frac{1}{c(~\mu)^2}~\sigma_h,~\xi_h}_{\Omega_h} + \lr{q_h,\nabla\cdot~\xi_h}_{\Omega_h} = 0 \qquad \forall ~\xi_h \in V_h.
\end{align}
\end{subequations}
The discrete Hamiltonian for this scheme is defined in analogy with the continuous case:
\begin{align}\label{eq:discreteH}
    H_h(q_h,p_h) = \frac{1}{2}\lr{p_h,p_h}_{\Omega_h} + \frac{1}{2}\lr{\frac{1}{c(~\mu)^2} ~\sigma_h, ~\sigma_h}_{\Omega_h}.
\end{align}
The mixed finite-element scheme~\cref{eq:mixedscheme} can be written in a matrix-vector form suitable for tensorization. It is shown in \cref{app:systems} that using $~M_V$ and $~M_W$ to denote the mass matrices associated with the spaces $V_h$ and $W_h$, respectively, and collecting $N$ degrees of freedom of finite element approximations of $q_h$ and $p_h$ into the vector $~y = [\,~q\trp\enspace~p\trp\,]\trp\in\mathbb{R}^{2N}$,
the system formed from \cref{eq:mixedscheme}--\cref{eq:sigmah} can be expressed as
\begin{align}
    \label{eq:blockqp}
    \begin{aligned}
    \dot{~y}
    = \left[\begin{array}{c}
        \dot{~q}\\
        \dot{~p}
    \end{array}\right]
    &= \left[\begin{array}{ c  c }
    ~0 & ~I \\
    -~I & ~0
    \end{array}\right]\left[\begin{array}{ c  c }
        ~M_W^{-1}~S\trp~M_V(~\mu)^{-1}~S & ~0 \\
        ~0 & ~I
    \end{array}\right]
    \left[\begin{array}{c}
        ~q\\
        ~p
    \end{array}\right]
= ~J~A\lr{~\mu}~y.
    \end{aligned}
\end{align}
Importantly, the mass matrix $~M_V(~\mu)$ depends affinely on $~\mu^{-2} = [\,\mu_1^{-2}\,\cdots\,\mu_p^{-2}\,]\trp\in\mathbb{R}^{p}$
and $~A(~\mu) = \mathrm{blockdiag}(~A_1(~\mu), ~I)$ where $~A_1(~\mu) = ~M_W^{-1}~S\trp~M_V(~\mu)^{-1}~S\in\mathbb{R}^{N\times N}$.
This is a canonical Hamiltonian system of the form \cref{eq:mgradHamiltonian} or, equivalently, \cref{eq:HODE}, with respect to the joint mass matrix $~M = \mathrm{blockdiag}(~M_W,~M_W)$ and the discrete Hamiltonian
\begin{align*}
    \begin{aligned}
        H_h(~y,~\mu)
        &= \frac{1}{2}\langle ~p, ~p \rangle_{~M_W} + \frac{1}{2}\langle ~q, ~M_W^{-1}~S\trp ~M_V(~\mu)^{-1} ~S~q \rangle_{~M_W} = \frac{1}{2}\left\langle~y,~A(~\mu)~y\right\rangle_{~M},
    \end{aligned}
\end{align*}
where $\langle ~a, ~b \rangle_{~M_W} = ~a\trp ~M_W ~b$ denotes the discrete inner product on $W_h$.

Because the affine parametric expansion of $~M_V(~\mu)$ is in terms of $~\mu^{-2}$, it is reasonable to approximate the parametric dependence of $~A_1(~\mu)$ as affine in the element-wise square $~\mu^2 = [\,\mu_1^2\,\cdots\,\mu_p^2\,]\trp$ so that $~A_1(~\mu)\approx\tens{T}_1~\mu^2$ for some constant tensor $\tens{T}_1\in\mathbb{R}^{N\times N\times p}$.
This also defines a tensor $\tens{T}\in\mathbb{R}^{2N \times 2N \times p}$ with $\tens{T}~\mu' \approx ~A(~\mu)$ with $~\mu' = [\,(~\mu^2)\trp,\, 1]\trp\in\mathbb{R}^{p+1}$ and entries
\begin{align*}
    \mathrm{T}_{ijk}
    = [k\neq 1][i\leq N][j\leq N]\,(\tens{T}_1)_{ijk} + [k=1][i>N][j=i],
\end{align*}
which amounts to inserting the $N\times N$ identity alongside the approximation for $~A_1(~\mu)$ through the last index of $\tens{T}$.
It follows that the model $\dot{~y}=~J(\tens{T}~\mu')~y$ is an affine parametric approximation to the original wave equation \cref{eq:wave} which retains its Hamiltonian structure, provided that $\tens{T}^*=\tens{T}$ is self-adjoint in its first two indices.

\begin{remark}
    It is possible to express the Hamiltonian FEM scheme \cref{eq:mixedscheme}--\cref{eq:discreteH} in a way that exactly retains the affine parametric structure in \cref{eq:sigmah} from $~M_V(~\mu)$. However, doing so requires simulating the intermediate variable $~\sigma_h$ and precludes the standard Hamiltonian expression $\dot{~y}=~J\Mnabla H(~y)$.  We have chosen to retain the approximate affine dependence $~A_1(~\mu)\approx\tens{T}_1~\mu^2$ both for this reason and because it enables a demonstration of the proposed OpInf procedure in a case where the parametric dependence in the FOM is not fully affine.
\end{remark}

Given an $~M$-orthonormal, $~J$-equivariant basis $~U\in\mathbb{R}^{2N\times 2r}$,
\Cref{alg:normal_eqns_Hamiltonian} can be applied to learn a reduced model
\begin{equation}\label{eq:hamOpInfROM1}
    \dot{\hat{~y}}
    = ~J\lr{\bar{\tens{T}}~\mu'}\hat{~y},
    \qquad
    \hat{~y}(0,~\mu)
    = ~U^*~y_0(~\mu),
\end{equation}
in terms of $\bar{\tens{T}}=\bar{\tens{T}}\trp$.
Alternatively, since the parametric dependence of $~A(~\mu)$ appears only in its upper-left block via $~A_1(~\mu)$, one may construct a block-structured ROM
\begin{align}
    \label{eq:hamOpInfROM2}
    \dot{\hat{~y}}
    = \left[\begin{array}{c}
        \dot{\hat{~q}}\\
        \dot{\hat{~p}}
    \end{array}\right]
    =
    \left[\begin{array}{ c  c }
        ~0 & ~I \\
        -~I & ~0
    \end{array}\right]\left[\begin{array}{ c  c }
        \bar{\tens{T}}_1~\mu & ~0 \\
        ~0 & \bar{~A}_2
    \end{array}\right]
    \left[\begin{array}{c}
        \hat{~q}\\
        \hat{~p}
    \end{array}\right],
    \qquad
    \hat{~y}(0,~\mu)
= ~U^*~y_0(~\mu),
\end{align}
which requires a symmetric tensor $\bar{\tens{T}}_1=\bar{\tens{T}}_1\trp$ inferred with \cref{alg:normal_eqns_Hamiltonian} and a symmetric matrix $\bar{~A}_2=\bar{~A}_2\trp$ inferred with standard Hamiltonian OpInf \cite{GRUBER2023116334,sharma2022hamiltonian}.
The numerical results reported in \Cref{sec:numerics} use \cref{eq:hamOpInfROM2}, but additional experiments were performed using the more general \cref{eq:hamOpInfROM1} without significant or qualitative differences to the outcomes.

 \section{Numerical results}
\label{sec:numerics}
This section applies tensor parametric OpInf to the parameterized heat and wave equation models introduced in earlier sections.
Numerical experiments for each problem are conducted in one and two spatial dimensions, comparing both intrusive and non-intrusive ROMs. Accuracy is assessed using the relative $L^2$-error in state approximations: if $~Q_s^\text{FOM}$ and $~Q_s^\text{ROM}$ are matrices of the FOM and ROM solutions, respectively, at the same parameter value and over $N_t$ time instances with uniform time step $\Delta t > 0$, the relative error for a collection of $\tilde{N}_s$ ROM solutions $~Q^\text{ROM} = \{~Q_{1}^\text{ROM}, \ldots, ~Q_{\tilde{N}_s}^\text{ROM}\}$ compared to the FOM solutions $~Q^\text{FOM} = \{~Q_{1}^\text{FOM}, \ldots, ~Q_{\tilde{N}_s}^\text{FOM}\}$ is given by
\begin{align}
    \label{eq:l2err}
    RL^2\left(~Q^\text{FOM},~Q^\text{ROM}\right)
    = \sqrt{\frac{\sum_{s} \|~Q_{s}^\text{FOM} - ~Q_{s}^\text{ROM}\|_{~M}^2}{
    \sum_{s} \|~Q_{s}^\text{FOM}\|_{~M}^2}},
    \quad
    \|~Q\|_{~M}^2 = \mathrm{trace}(~Q\trp~M~Q).
\end{align}
For a given basis $~U$, the projection error is defined by replacing $~Q_s^\text{ROM}$ with $~U~U^*~Q_s^\text{FOM}$ in~\cref{eq:l2err}.
We also examine the signed elementwise errors $(~Q^\text{FOM}_{s})_{ij} - (~Q^\text{ROM}_{s})_{ij}$ and corresponding absolute errors for select trajectories.  As discussed in \Cref{rem:regularization}, we have chosen not to explicitly regularize the OpInf learning problems; further performance improvements, such as improvements in numerical conditioning, are possible with a careful choice of regularization parameter.
Codes to reproduce the numerical experiments are publicly available at \href{https://github.com/sandialabs/HamiltonianOpInf/tree/parametric}{\texttt{github.com/sandialabs/HamiltonianOpInf/tree/parametric}}.

\subsection{Heat equation}
This section applies the structure-agnostic tensor parametric OpInf methods of \Cref{sec:generic} to the parameterized heat equation~\cref{eq:heat}. An $~M$-orthonormal basis matrix $~U \in \mathbb R^{N\times r}$ is constructed through weighted proper orthogonal decomposition (POD) \cite{berkooz1993pod,graham1999vortexsheddingPOD,sirovich1987turbulence} as follows.
Let $~Y = [~R~Q_1 \enspace ~R~Q_2 \enspace \cdots \enspace ~R~Q_{N_s}] \in \mathbb R^{N\times N_tN_s}$, where $~Q_s$ is the FOM state snapshot matrix corresponding to training parameter value $~\mu_s$ and $~R$ is the upper triangular Cholesky factor of the mass matrix, $~R\trp ~R = ~M$.
For a fixed reduced dimension $r \ll N$, the basis matrix is then given by $~U = ~R^{-1}\tilde{~U}_{r}$, where $~Y = \tilde{~U}~\Sigma ~V\trp$ is the singular value decomposition (SVD) of $~Y$ and $\tilde{~U}_{r} = \tilde{~U}_{:,:r}$ denotes the first $r$ columns of $\tilde{~U}$ (the $r$ principal left singular vectors of $~Y$).
It follows that $~U^*~U = ~U\trp ~M ~U = ~U\trp ~R\trp ~R ~U = \tilde{~U}_{r}\trp \tilde{~U}_{r} = ~I$, which is the desired orthonormality condition. \cref{alg:normal_eqns_unstructured}, which relies on solving the normal equations, and \cref{alg:lstsq_unstructured}, which employs an SVD-based least-squares solver, are applied to learn the tensor operator $\bar{\tens{T}}\approx\hat{\tens{T}}$ in \eqref{eq:rODE-opinf}.  The results are denoted ``OpInf (normal)'' and ``OpInf (lstsq)'', respectively.

\subsubsection{One spatial dimension}

First, consider the system~\cref{eq:heat} on the one-dimensional domain $\Omega = (0, 2\pi)$
with the parameterized heat conductivity coefficient
\begin{align*}
    c(x, ~\mu)
    = \Big[x \in \left(0, \tfrac{2\pi}{3}\right)\Big] \mu_1
    + \Big[x \in \left(\tfrac{2\pi}{3}, \tfrac{4\pi}{3}\right)\Big] \mu_2
    + \Big[x \in \left(\tfrac{4\pi}{3}, 2\pi\right)\Big] \mu_3,
    \qquad
    ~\mu
\in \mathbb{R}^3.
\end{align*}
The mesh $\Omega_h$ consists of $1{,}000$ uniform elements for a total of $N=999$ spatial degrees of freedom. To generate snapshot data for training, the FOM is integrated in time from $t_0 = 0$ using an implicit multi-step variable-order method, available in Python as \texttt{scipy.integrate.BDF} \cite{2020SciPy-NMeth}, until final time $t_f = 8$. The states are recorded at uniform intervals with step size $\Delta t = 0.008$, resulting in $N_t = 1{,}001$ snapshots. The initial condition for this experiment is chosen to be
\begin{align*}
    q(x,0) = \exp(-(x-\pi)^2)\sin\lr{\frac{x}{2}},
\end{align*}
which is independent of $~\mu$.
Snapshot data are generated randomly over $N_s = 80$ training parameters and $20$ testing parameters log-uniformly sampled from the interval $(0.01,1)^3$, and snapshot time derivatives are estimated with second-order finite differences.  Note that the chosen parameter sampling creates sharp kinks in the solution at the discontinuities of the conductivity coefficient $c$, provided that the components of $~\mu$ are widely separated (c.f.~\Cref{fig:1dheatsol}).  The ROMs are integrated in time using the same implicit method as the FOM to generate reduced state approximations.

\begin{figure}
    \centering
    \includegraphics[width=\linewidth]{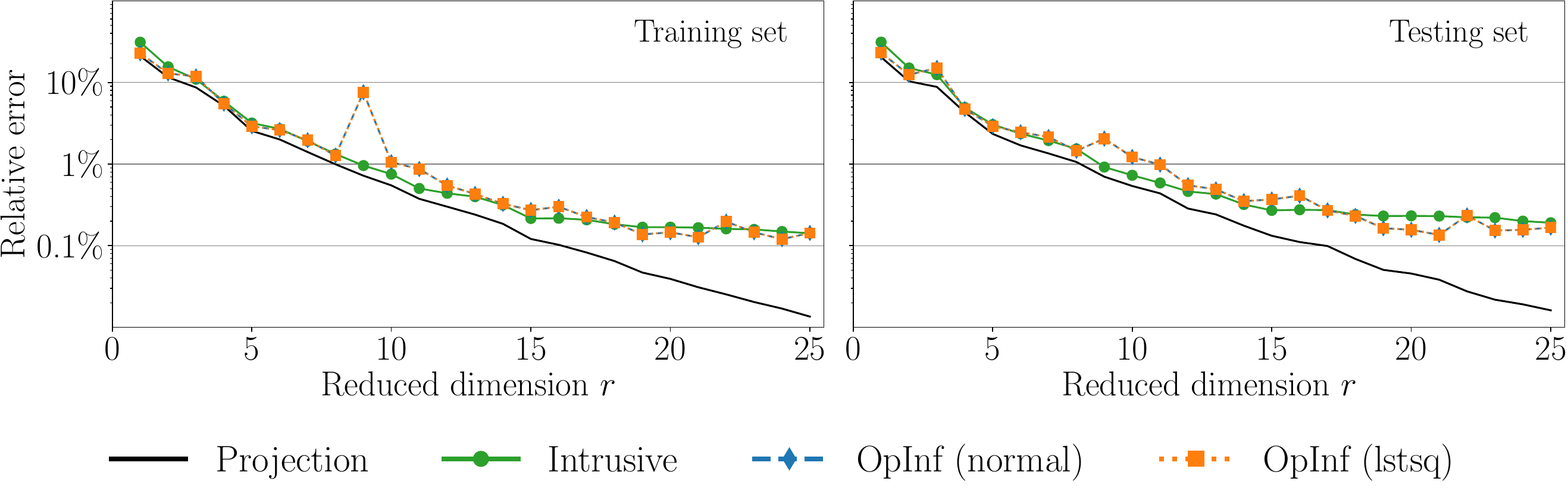}
    \caption{Relative $L^2$ error in ROM solutions to the 1D heat system and the projection error of the FOM snapshots as a function of the reduced dimension $r$, over all training parameter values (left) and testing parameter values (right).
}
    \label{fig:1dheattraintest}
\end{figure}

\Cref{fig:1dheattraintest} shows the relative error~\cref{eq:l2err} in the ROM solutions, along with the projection error, as functions of the basis size $r$ over all training and testing snapshots, respectively.
Over the training snapshots, the two OpInf ROMs exhibit identical error profiles, closely matching and sometimes improving upon that of the intrusive ROM. A small exception to this trend occurs at $r = 9$, where the OpInf ROM performs poorly for a small number of training parameters: in this case, the relative error is about $75.85\%$ at $~\mu \approx [0.011, 0.977, 0.422]\trp$ near the edge of the parameter domain, $19.25\%$ and $6.73\%$ at two other parameters, and less than $3\%$ at all other training parameters, with a median error of $1.06\%$. Differences between the classical and data-driven ROMs are expected, as OpInf generally does not pre-asymptotically recover the intrusive Galerkin ROM. Therefore, closure error in the learned operator can both improve (in the short term) and degrade ROM accuracy in the finite-data regime~\cite{gruber2024vc,peherstorfer2020sampling}. Similar error profiles are observed over the test snapshots, with the ROM error leveling off at about $0.2\%$ as $r$ increases.

\begin{figure}
    \centering
    \includegraphics[width=\linewidth]{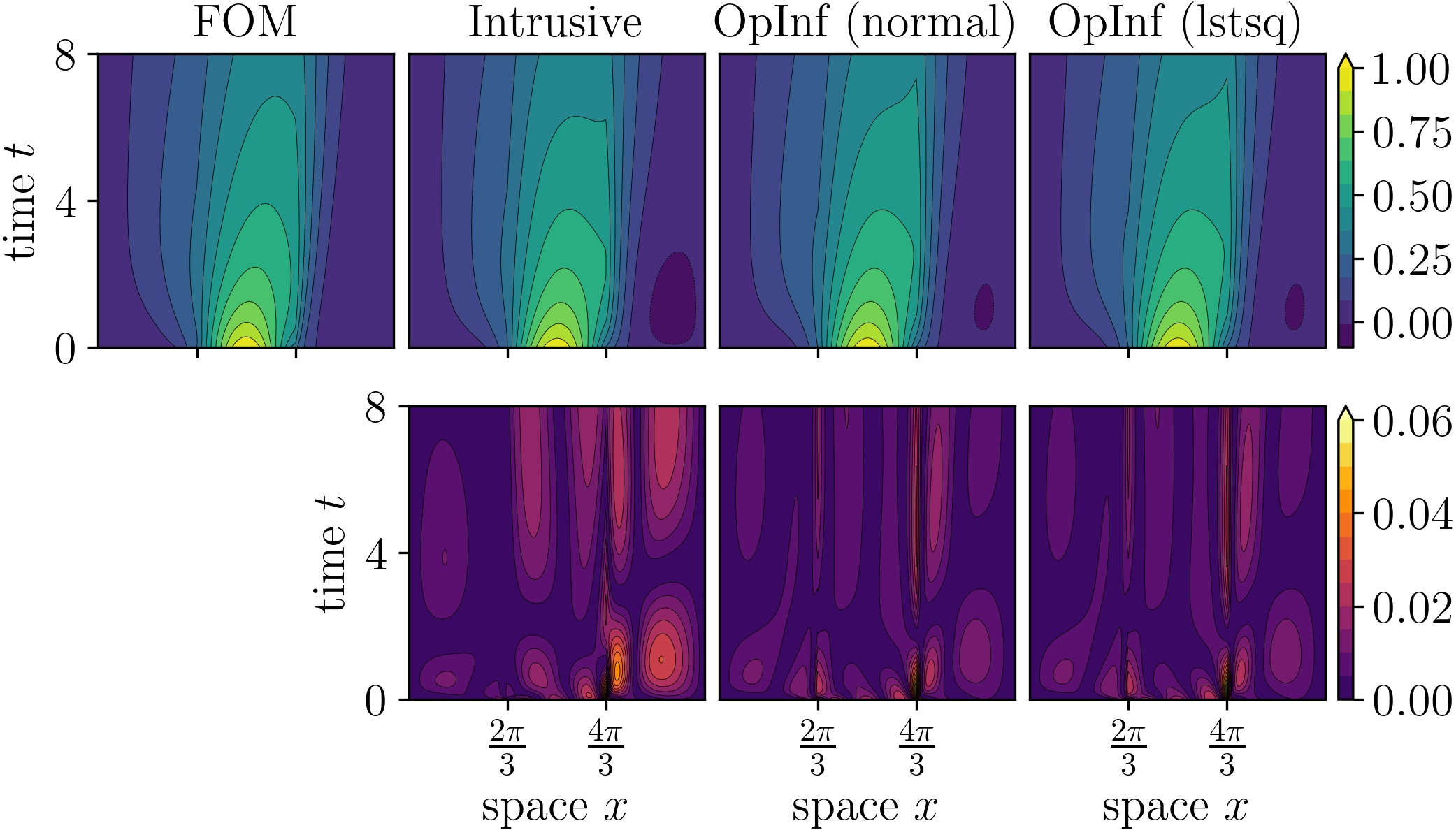}
    \caption{
    Top row: space-time contour plots of FOM and ROM solutions to the 1D heat equation. Bottom row: pointwise absolute error in the ROM solutions. The reduced dimension is $r=6$ and the parameters are $~\mu \approx [0.453, 0.163, 0.031]\trp$, which is not in the training set.
    }
    \label{fig:1dheatsol}
\end{figure}

A closer look is provided by \Cref{fig:1dheatsol}, which displays space-time contour plots of the FOM solution alongside solutions of ROMs with dimensionality $r=6$ at a single testing parameter.  The two OpInf ROMs have identical error profiles; indeed, the operators learned by the two approaches match nearly to machine precision, indicating that numerical conditioning is not problematic in this experiment. For all ROMs, errors are largest near $x = 2\pi/3$ and $x=4\pi/3$, the points in the domain where the conductivity coefficient is discontinuous. The relative errors in the ROM solutions are approximately $2.26\%$ for the OpInf ROMs and $3.34\%$ for the intrusive ROMs.

\subsubsection{Two spatial dimensions}

Consider now the square two-dimensional spatial domain $\Omega = (0, 2\pi) \times (0, 2\pi)$.
In this experiment, the parameterized heat conductivity coefficient is
\begin{align*}
    \begin{aligned}
    c(~x,~\mu) = \Bigl[ ~x \in (0,\pi)^2 \Bigr]\mu_1
    &+ \Bigl[ ~x \in (\pi,2\pi) \times (0,\pi) \Bigr]\mu_2\\
    &+ \Bigl[ ~x \in (0,\pi) \times (\pi, 2\pi) \Bigr]\mu_3
    + \Bigl[ ~x \in (\pi,2\pi)^2 \Bigr]\mu_4,
    \end{aligned}
    \qquad ~\mu \in \mathbb R^4.
\end{align*}
The mesh $\Omega_h$ is unstructured and contains $4{,}116$
simplicial elements, yielding $N=1{,}971$ spatial degrees of freedom independent of the boundary conditions of the problem. The terminal time is set to $t_f = 4$, with a time-step size of $\Delta t = 0.002$, resulting in $N_t=1{,}001$ points in time. As before, both the FOM and ROMs are time-integrated using SciPy's implicit multi-step method.
The initial condition is defined to be
\begin{align*}
    q(~x,0) =  \exp\big(-(x_1-\pi)^2-\-(x_2-\pi)^2\big)\sin\lr{\frac{x_1}{2}}\sin\lr{\frac{x_2}{2}}.
\end{align*}
Training data are generated over $N_s=80$ parameter values, while the test data consists of $20$ parameter values, all of which have entries randomly and log-uniformly sampled from $(0.1,1.0)^4$.

\begin{figure}
    \centering
    \includegraphics[width=\linewidth]{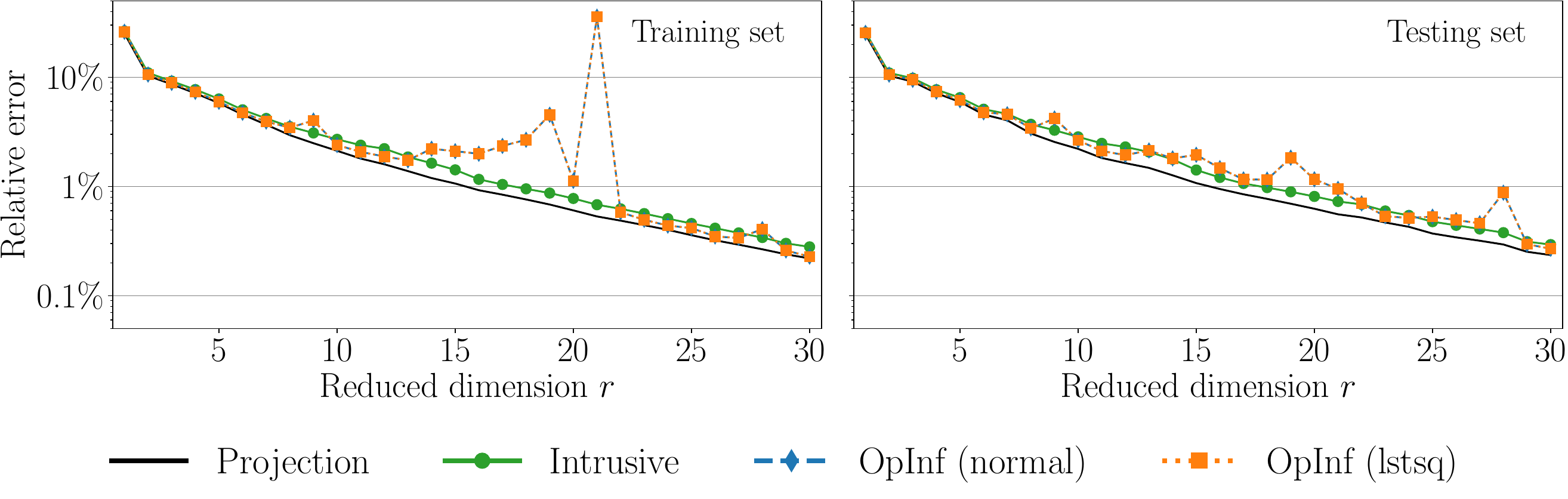}
    \caption{Relative $L^2$ error in ROM solutions to the 2D heat system and the projection error of the FOM snapshots as a function of the reduced dimension $r$, over all training parameter values (left) and testing parameter values (right).}
    \label{fig:2dheattraintest}
\end{figure}

\begin{figure}
    \centering
    \begin{subfigure}[b]{\textwidth}
        \centering
        \begin{minipage}{0.32\textwidth}
            \centering
            \phantom{)---}$t=0$\phantom{(}
            \includegraphics[width=\textwidth,trim=4cm 12cm 0 0, clip]{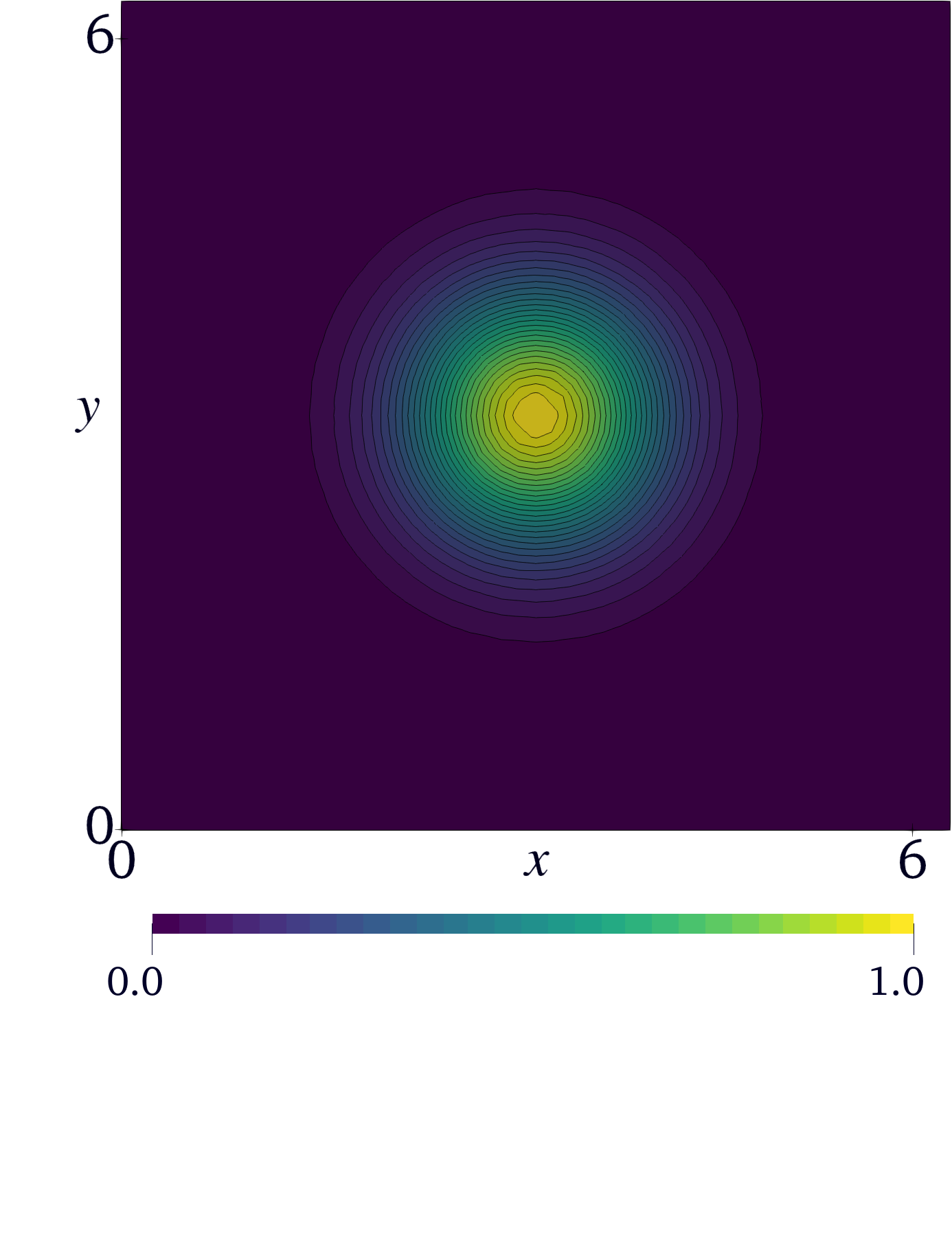}
            \vspace{0.2cm}
        \end{minipage}\begin{minipage}{0.32\textwidth}
            \centering
            \phantom{)---}$t=2$\phantom{(}
            \includegraphics[width=\textwidth,trim=4cm 12cm 0 0, clip]{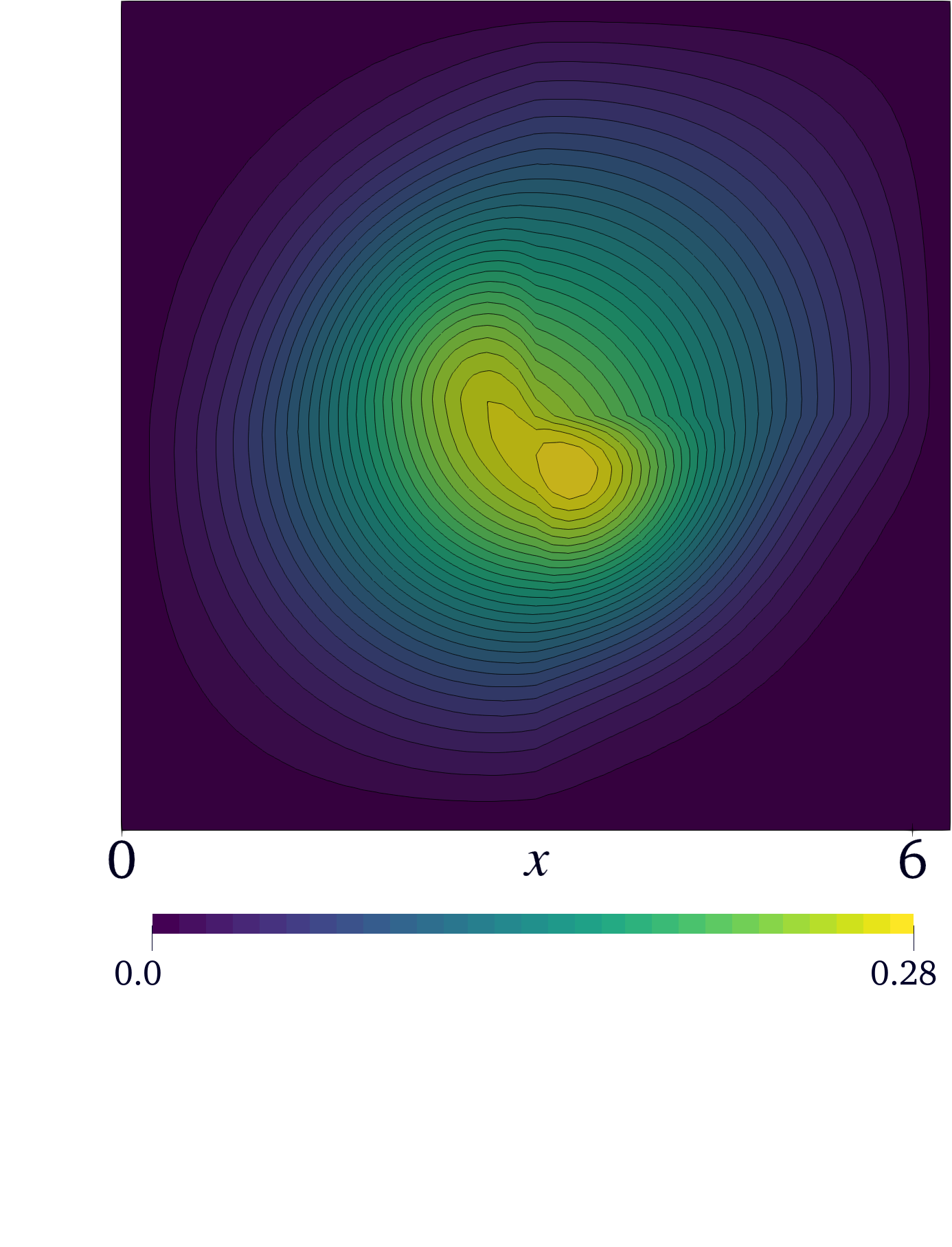}
            \vspace{0.2cm}
        \end{minipage}\begin{minipage}{0.32 \textwidth}
            \centering
            \phantom{)---}$t=4$\phantom{(}
            \includegraphics[width=\textwidth,trim=4cm 12cm 0 0, clip]{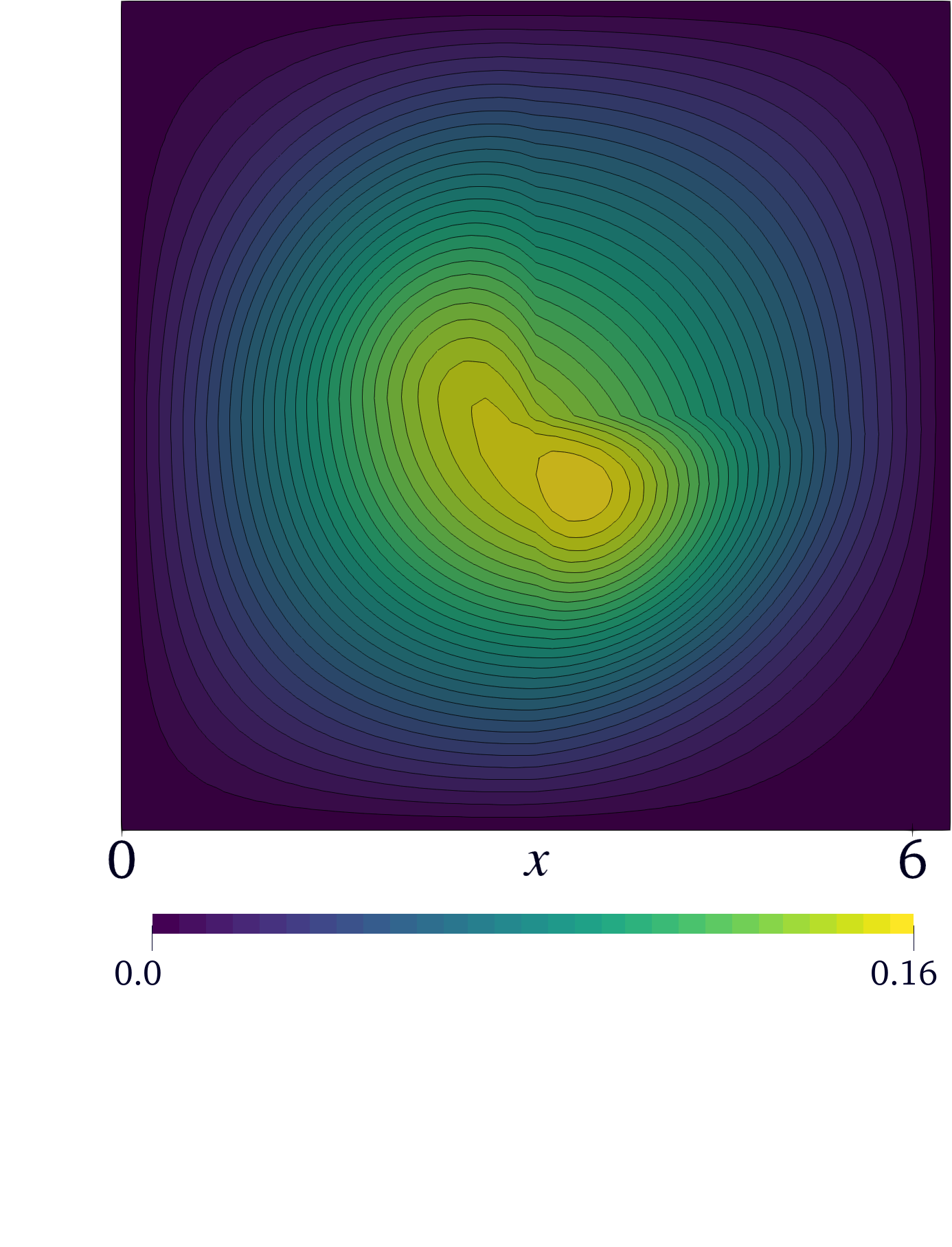}
            \vspace{0.2cm}
        \end{minipage}
\end{subfigure}\label{fig:2dheatsol}
    \\[-12pt]
    \begin{subfigure}[b]{\textwidth}
        \centering
        \begin{minipage}{0.32\textwidth}
            \centering
            \phantom{)---}Intrusive\phantom{(}
            \includegraphics[width=\textwidth, trim=4cm 12cm 0 0, clip]{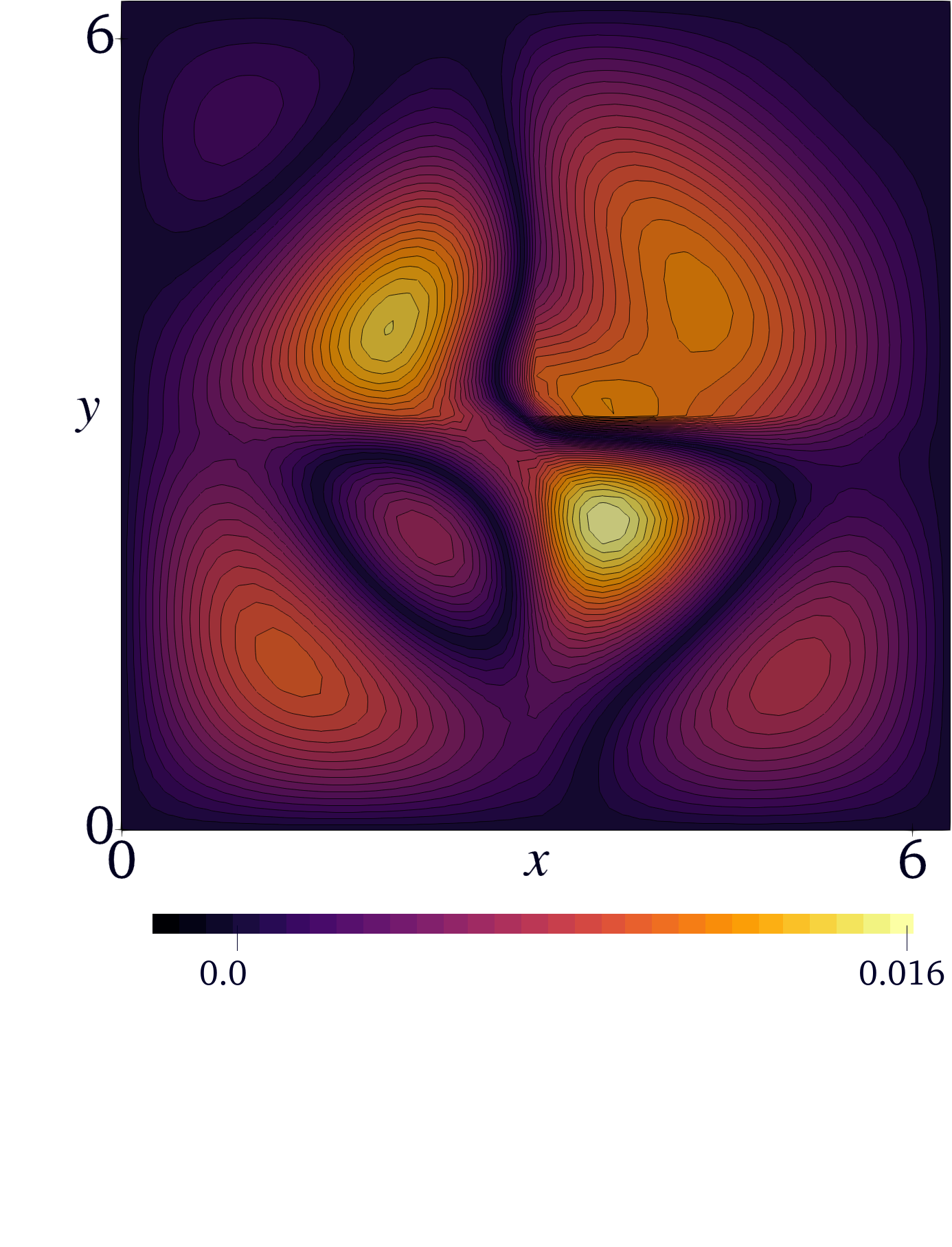}
            \vspace{0.2cm}
        \end{minipage}
        \begin{minipage}{0.32\textwidth}
            \centering
            \phantom{)---}OpInf (normal)\phantom{(}
            \includegraphics[width=\textwidth, trim=4cm 12cm 0 0, clip]{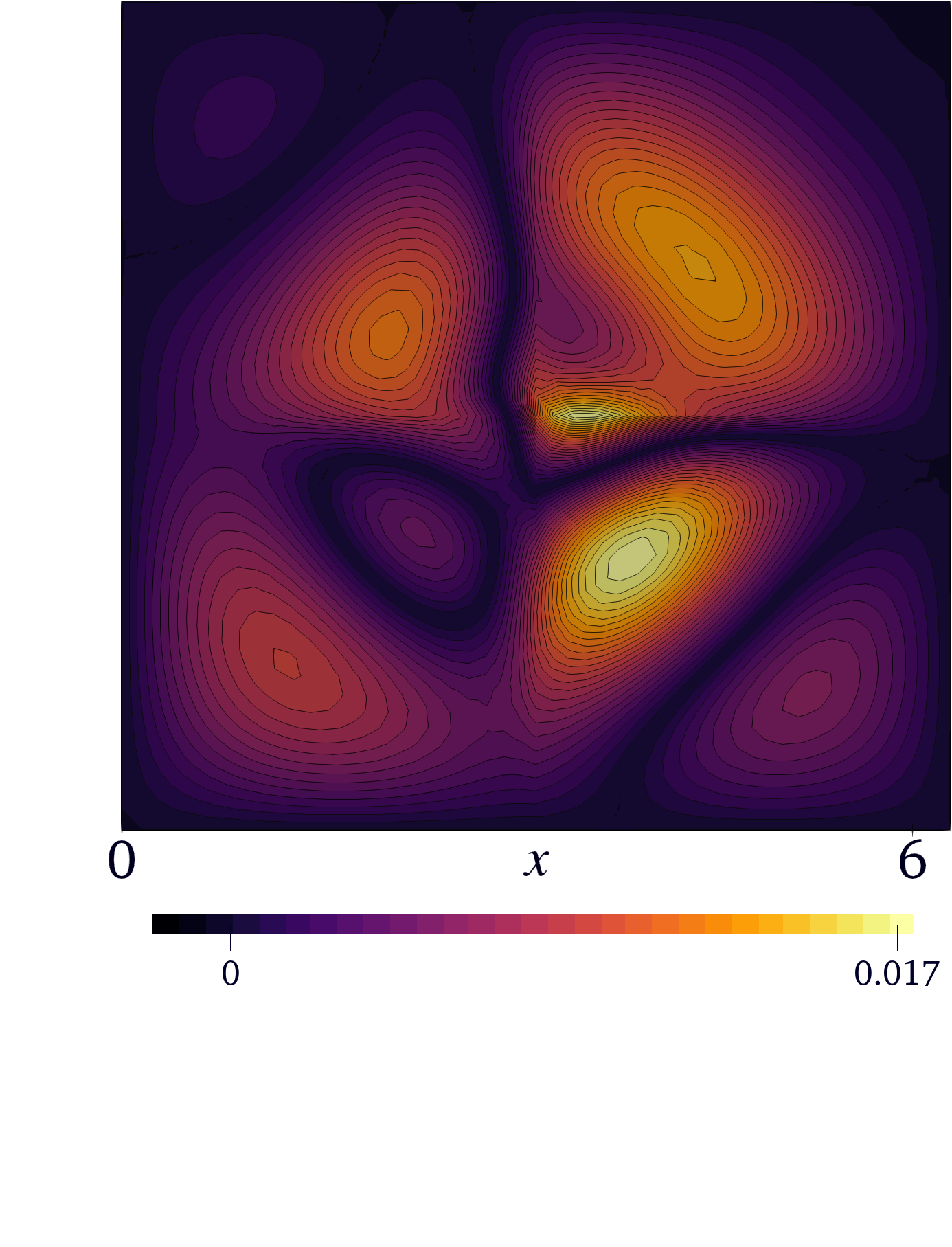}
            \vspace{0.2cm}
\end{minipage}\begin{minipage}{0.32\textwidth}
            \centering
            \phantom{)---}OpInf (SVD)\phantom{(}
            \includegraphics[width=\textwidth, trim=4cm 12cm 0 0, clip]{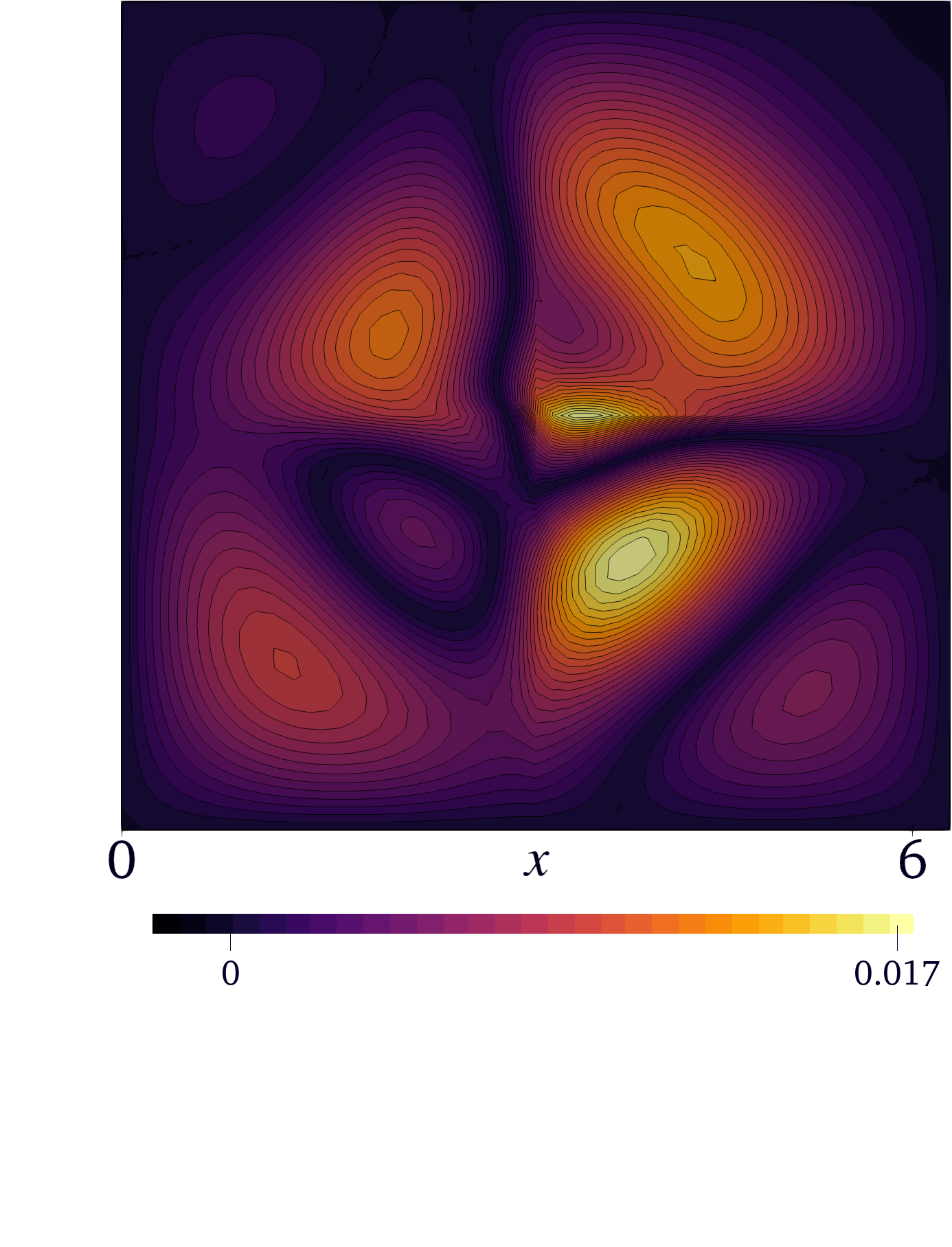}
            \vspace{0.2cm}
        \end{minipage}\label{fig:2dheaterr}
    \end{subfigure}
    \vspace{-0.75cm}
    \caption{FOM solution to the 2D heat system at different times (top) and pointwise absolute error in the ROM solutions at $t=4$ (bottom). The reduced state dimension is $r=8$ and the parameters are $~\mu\approx [0.319, 0.117, 0.179, 0.686]\trp$.
    }
    \label{fig:2dheat}
\end{figure}

\Cref{fig:2dheattraintest} displays the relative $L^2$-errors in the ROM solutions, along with the projection error, over the training and testing snapshots. As in the previous experiment, the OpInf ROMs produced by \Cref{alg:normal_eqns_unstructured} and \Cref{alg:lstsq_unstructured} are numerically equivalent, indicating reasonable conditioning in the OpInf problem.
As before, the OpInf ROM errors generally decrease as $r$ increases, with a few exceptions. In particular, at $r = 21$, the OpInf ROM training errors jump due to a large error at a single parameter value. For most choices of $r$, the OpInf ROM is comparable to the Galerkin ROM and nearly achieves the projection error.
\Cref{fig:2dheat} shows the time evolution of the FOM solution for a single testing parameter, as well as the corresponding absolute errors in the ROM solutions at the terminal time, with $r=8$. In this case, the OpInf ROMs produce solutions that are slightly more accurate, achieving a relative error of 4.19\%, compared to 4.55\% for the intrusive ROM solution.

\subsubsection{Symmetry and time derivative error}

\begin{figure}
    \centering
    \includegraphics[width=\textwidth]{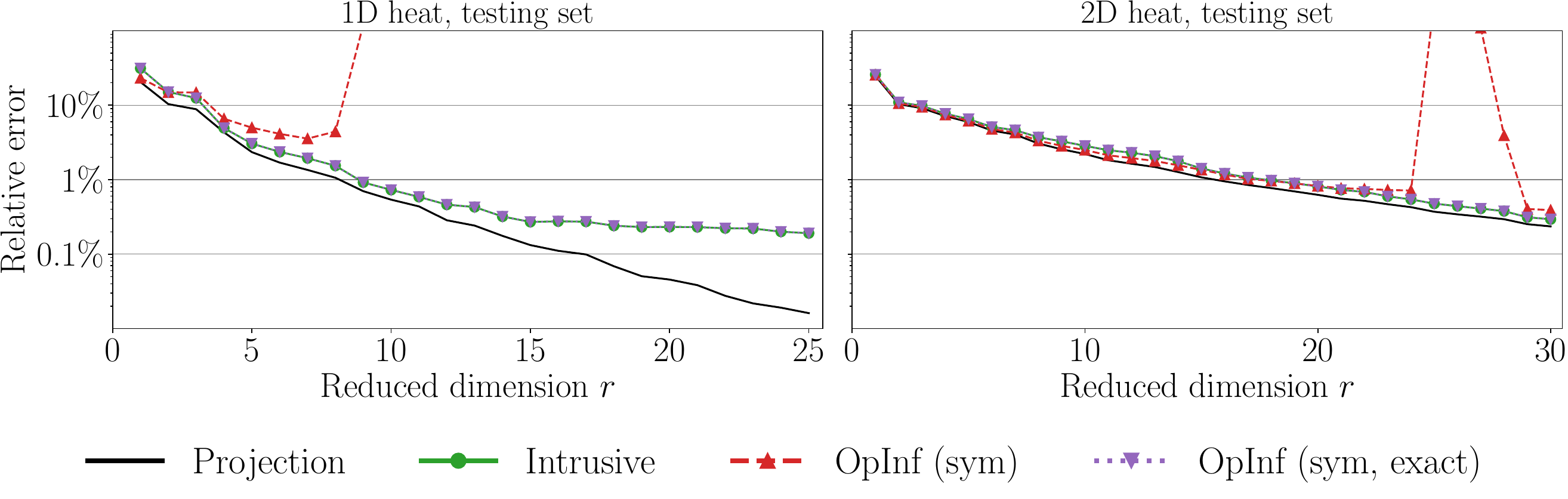}
    \caption{Relative testing error of symmetric ROMs for the heat problems in 1D (left) and 2D (right), as well as the projection error of the FOM snapshots, as a function of the reduced dimension $r$. The OpInf ROMs are inferred with \Cref{alg:normal_eqns_Hamiltonian}. When estimated time derivative data are used, the OpInf ROMs become unstable for some choices of $r$; with exact time derivative data, the inference recovers the intrusive ROM.}
    \label{fig:heat-symmetry}
\end{figure}

The symmetry-preserving inference in \Cref{alg:normal_eqns_Hamiltonian} is equally applicable to the present case of the heat equation, since the tensor operator $\hat{\tens{T}}$ governing the intrusive Galerkin ROM is guaranteed to be symmetric in its first two indices.  However, this restricts the space of allowable solutions to the OpInf problem by design, which may not be advantageous when there is unremovable closure error introduced by time derivative approximation.  \Cref{fig:heat-symmetry} displays the results of performing \Cref{alg:normal_eqns_Hamiltonian} on the 1-D and 2-D heat problems using both approximate and true time derivatives.
In the latter case, the true time derivative governing the Galerkin ROM is given by
\begin{align}
    \dot{\hat{~Q}}_{s}
    = ~U^{*}~A(~\mu_s)~U~U^{*}~Q_s.
\end{align}
Interestingly, for the heat problem, there is no obvious benefit to enforcing symmetry on the learned tensor operator when time derivatives are approximated.  This suggests that the closure error introduced by time derivative approximation is sufficiently disruptive to negate the benefits of the restriction bias imposed by symmetry, at least when no additional regularization is employed (see \Cref{rem:regularization}).  Conversely, the symmetry-preserving \Cref{alg:normal_eqns_Hamiltonian} exactly (or near-exactly) recovers the symmetric intrusive operator when exact (or near-exact, as in \cite{peherstorfer2020sampling}) time derivatives are employed.  This suggests that the benefits of symmetry-preservation\textemdash and restriction biases in general\textemdash are problem- and data-dependent.  In contrast, the next section presents an example where the structure-preservation inherent in \Cref{alg:normal_eqns_Hamiltonian} is critical to producing a parametric OpInf ROM which exhibits both physical consistency and good predictive performance in time.

\subsection{Wave equation}\label{sec:numerics-wave}

The structure-preserving procedure from~\Cref{sec:Hamiltonian} is now applied to the parameterized wave equation initial boundary value problem~\cref{eq:wave} in one and two spatial dimensions. In each case, the basis matrix is constructed via PSD with a mass-weighted variant of the cotangent lift algorithm from \cite{peng2016symplectic} as follows: Let $~U_W\in\mathbb{R}^{N\times r}$ be the $~M_W$-weighted POD basis of dimension $r$ for $~Y = \left[~R~Q_1 \enspace \cdots ~R~Q_{N_s} \enspace ~R~P_1 \enspace \cdots \enspace ~R~P_{N_s} \right]\in \mathbb R^{N\times 2N_tN_s}$, where $~Q_s,~P_s \in \mathbb{R}^{N\times N_t}$ are matrices containing the position and momentum FOM snapshots corresponding to training parameter $~\mu_s$, $s=1,\ldots,N_s$.
Then the PSD basis is $~U = \mathrm{blockdiag}(~U_W,~U_W)\in\mathbb{R}^{2N\times 2r}$, which is orthonormal with respect to $~M = \mathrm{blockdiag}(~M_W, ~M_W)$.

To demonstrate the effects of symmetry preservation, both \Cref{alg:normal_eqns_unstructured} and \Cref{alg:normal_eqns_Hamiltonian} are applied to produce ROMs with the block structure~\cref{eq:hamOpInfROM2}. Results with symmetry preservation ($\bar{\tens{T}}_1 = \bar{\tens{T}}_1\trp$ and $\bar{~A}_2 = \bar{~A}_2\trp$) are denoted ``H-OpInf'', while non-symmetric ROM results are labeled ``OpInf''.
Accuracy is assessed with the relative $L^2$-error \cref{eq:l2err} for the positional variable, as well as
the absolute error in the reduced Hamiltonian~\cref{eq:reduced-hamiltonian} as a function of time, $|\hat{H}(\hat{~y}(t),~\mu) - \hat{H}(\hat{~y}(0),~\mu)|.$

\subsubsection{One spatial dimension}

\begin{figure}
    \centering
    \includegraphics[width=\textwidth]{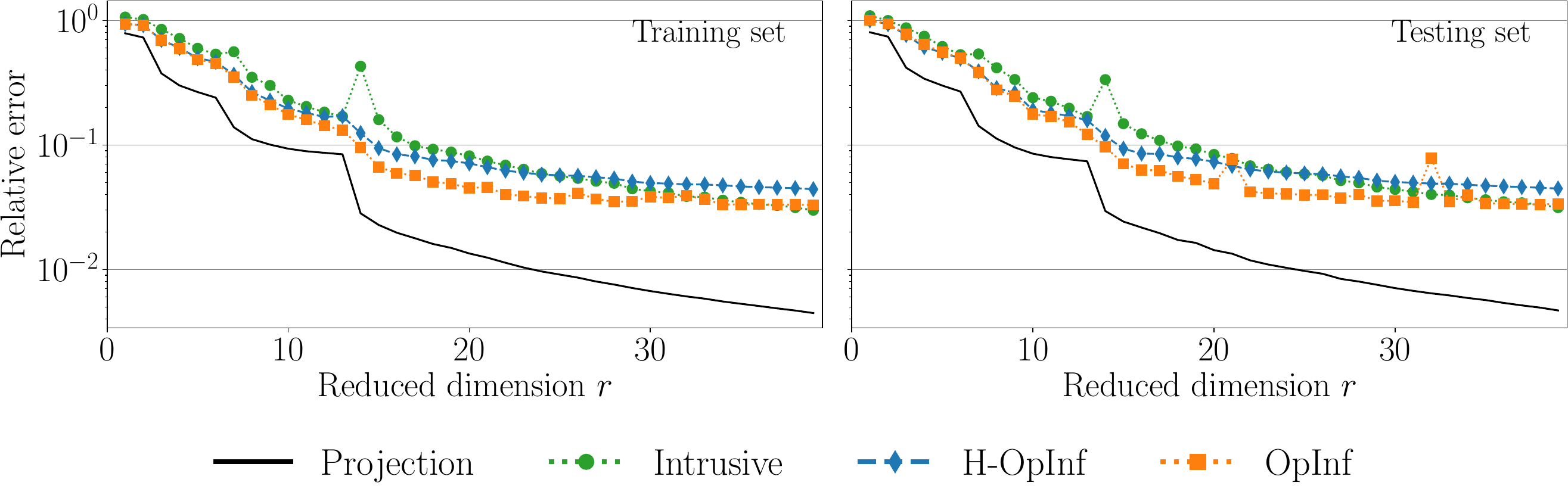}
    \caption{Relative $L^2$ error in ROM solutions to the 1D wave system and the projection error of the FOM snapshots as a function of the reduced dimension $r$, over all training parameter values (left) and testing parameter values (right).}
    \label{fig:1dwavetraintest}
\end{figure}

\begin{figure}
    \centering
    \includegraphics[width=\linewidth]{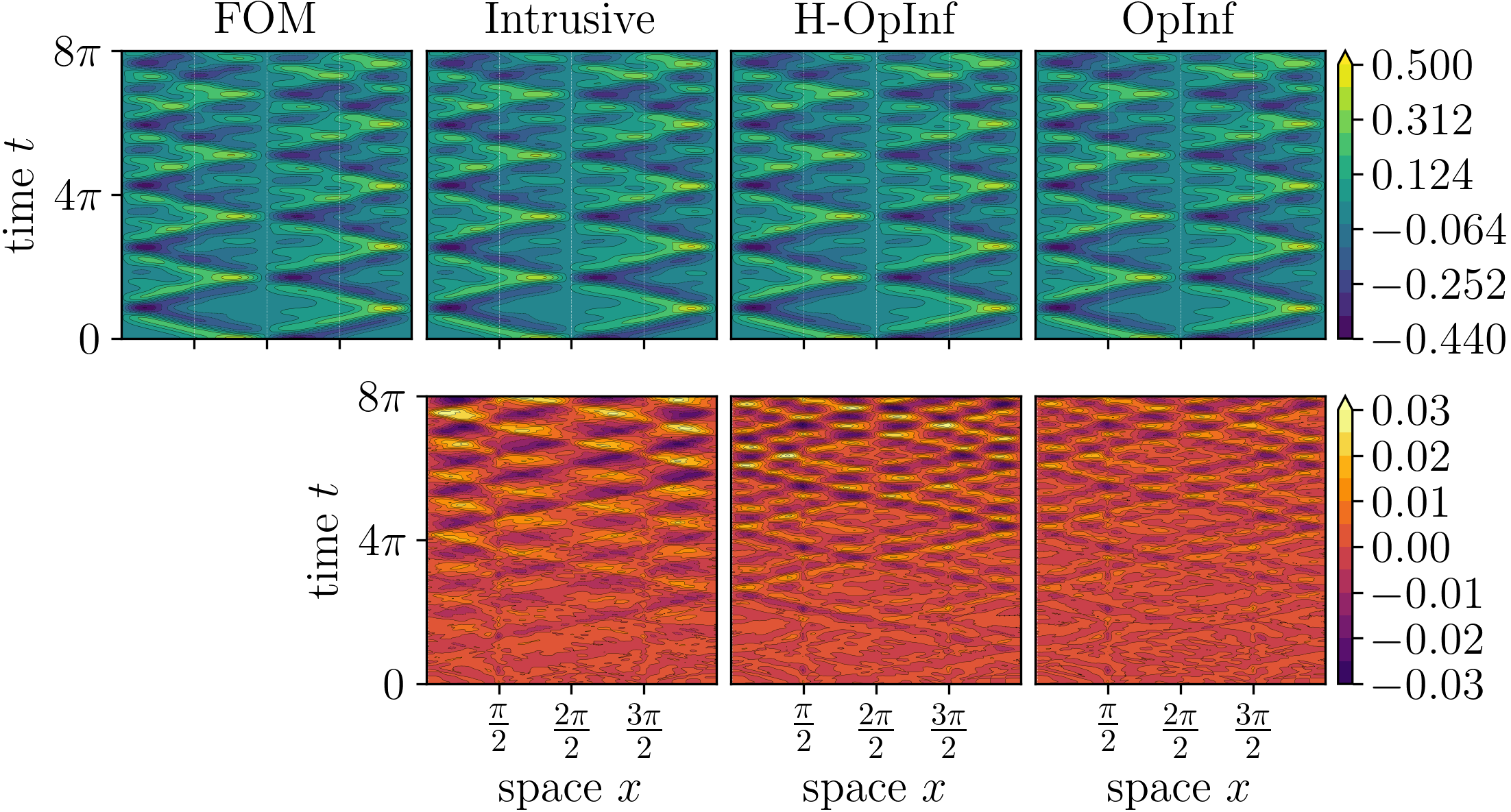}
    \caption{Space-time contour plots of the FOM and ROM solutions to the 1D wave equation (top) and signed pointwise error in the ROM solutions (bottom). The reduced state dimension is $r=18$ and the parameters are $~\mu \approx [1.064, 1.794, 1.724, 1.181]\trp$.
} \label{fig:1dwave}
\end{figure}

In this experiment, the system~\cref{eq:wave} is considered on the one-dimensional domain $\Omega = (0, 2\pi)$
with the parameterized wave speed given by
\begin{align*}
    \begin{aligned}
    c(x, ~\mu)^2
    = \Bigl[x \in (0, \tfrac{\pi}{2})\Bigr] \mu_1^2
    &+ \Bigl[x \in (\tfrac{\pi}{2}, \pi)\Bigr] \mu_2^2
    \\
    &+ \Bigl[x \in (\pi, \tfrac{3\pi}{2})\Bigr] \mu_3^2
    + \Bigl[x \in (\tfrac{3\pi}{2}, 2\pi)\Bigr]\mu_4^2,
    \end{aligned}
    \qquad
    ~\mu \in \mathbb R^4.
\end{align*}
The order of the space $W_h$ is set to 0, while the order of $V_h$ is chosen to be 1. The mesh~$\Omega_h$ is composed of 1000 uniform elements with $N=1000$ spatial degrees of freedom for each of the position and momentum variables.
Snapshot data are generated by time-integrating the FOM~\cref{eq:blockqp} using the energy-conserving implicit midpoint rule,
\begin{align*} \frac{~y^{n+1}-~y^n}{\Delta t}
    = ~J~A(~\mu)\frac{~y^{n+1} + ~y^n}{2},
    \qquad ~y^0(x) = ~y_0,
\end{align*}
until terminal time $t_f=8\pi$ with step size $\Delta t = \pi/100$ for a total of $N_t = 801$ points in time. The initial condition is set to
\begin{align*}
    y(x,0) = \exp(-(x-\pi)^2)\sin(x).
\end{align*}
Snapshot data are obtained over 40 training parameter values and 10 test values, with all parameter entries uniformly sampled from the interval $(0.8,2.4)^4$, yielding substantial but relatively controlled variations in wave profiles (c.f.~\Cref{fig:1dwave}).
The resulting ROMs are also integrated in time using the implicit midpoint rule.

\Cref{fig:1dwavetraintest} shows the relative errors in the ROM solutions and the projection error as functions of the reduced dimension $r$ for all training and testing snapshots, respectively.
The projection error decays slowly as $r$ increases, indicating that this is a difficult problem for linear model reduction. In both the training and testing sets, the non-symmetric OpInf ROMs tend to slightly outperform the intrusive and symmetric OpInf ROMs in an $L^2$ sense\textemdash a feature of their softly-enforced inductive bias toward the governing physics.
A more detailed comparison is given in \Cref{fig:1dwave}, which displays contour plots of the spacetime FOM solution alongside the corresponding ROM solutions for a reduced dimension of $r=12$ at a single testing parameter, as well as the signed error in the ROM solutions.
However, the non-symmetric OpInf ROMs fail to preserve the reduced Hamiltonian over time. This is demonstrated in~\Cref{fig:Herr}, which shows that the intrusive ROM and the symmetric OpInf ROM conserve the reduced Hamiltonian to near machine precision, while the non-symmetric OpInf ROM immediately alters the value of the reduced Hamiltonian. Over long integration times, the OpInf ROM gains energy erroneously, which can lead to un-physical and unstable solutions (c.f.~\Cref{fig:1dwaveblowup}).  This is the key issue faced by soft-constrained approaches such as unconstrained OpInf in the presence of systems with physical conservation laws. Indeed, while the block structure of \cref{eq:hamOpInfROM2} is helpful for non-symmetric OpInf, as unstructured non-symmetric ROMs of the form \cref{eq:hamOpInfROM1} tend to quickly become unstable, simply imposing block structure on the ROM is not enough to guarantee energy conservation.

\begin{figure}
    \centering
    \includegraphics[width=\textwidth]{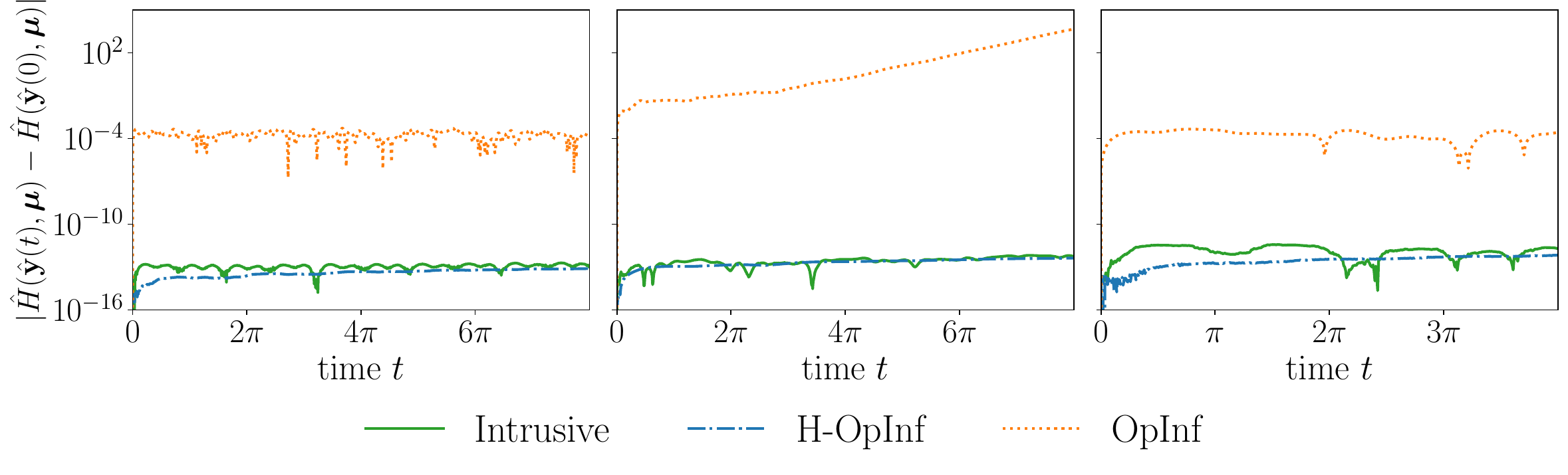}
    \caption{Absolute error in the reduced Hamiltonian versus time, for the ROM solutions to the wave equation displayed in~\Cref{fig:1dwave} (left), ~\Cref{fig:1dwaveblowup} (middle), and ~\Cref{fig:2dwave} (right).}
    \label{fig:Herr}
\end{figure}

\begin{figure}
    \centering
    \includegraphics[width=\textwidth]{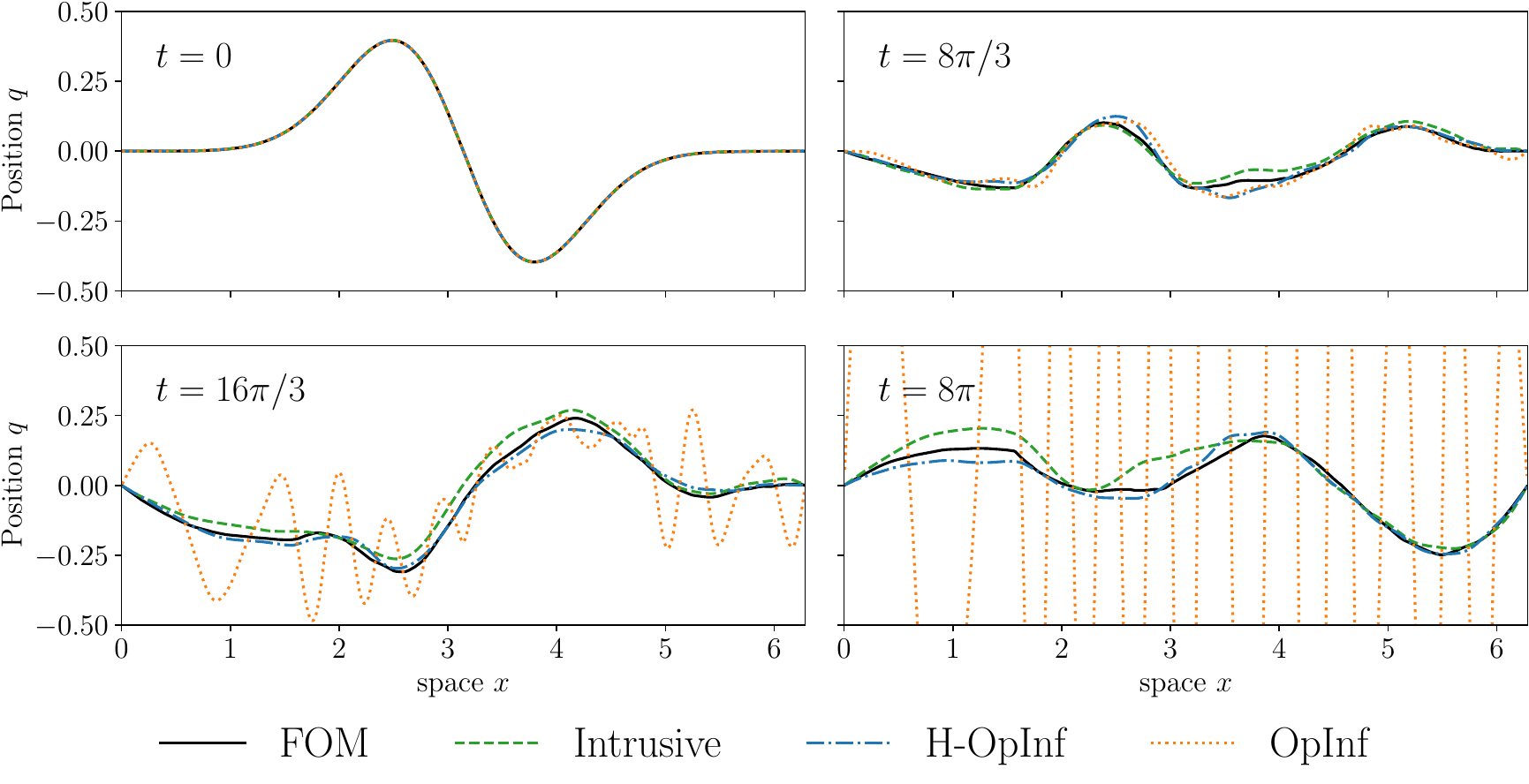}
    \caption{Emergence of non-physical oscillations and blow-up in a non-symmetric OpInf ROM solution to the 1D wave equation due to the lack of Hamiltonian structure preservation.}
    \label{fig:1dwaveblowup}
\end{figure}

To further illustrate the effects of the erroneous energy gain by non-symmetric OpInf ROMs, consider a related experiment in which the training parameters are sampled from the exaggerated range $(0.8,8)^4$ which produces extreme deviations in the wave over time.
\Cref{fig:1dwaveblowup} displays solutions of ROMs of dimension $r=30$ at a single
$~\mu\approx = [7.92, 1.53, 2.30, 1.96]\trp$.
The solution to the non-symmetric OpInf ROM exhibits nonphysical oscillations that progressively amplify over time. In contrast, the solutions from the intrusive and the H-OpInf ROM remain stable and approximately aligned with the FOM solution.
The error in the reduced Hamiltonian for each ROM is shown in~\Cref{fig:Herr}.

\subsubsection{Two spatial dimensions}
\begin{figure}
\centering
\includegraphics[width=\linewidth]{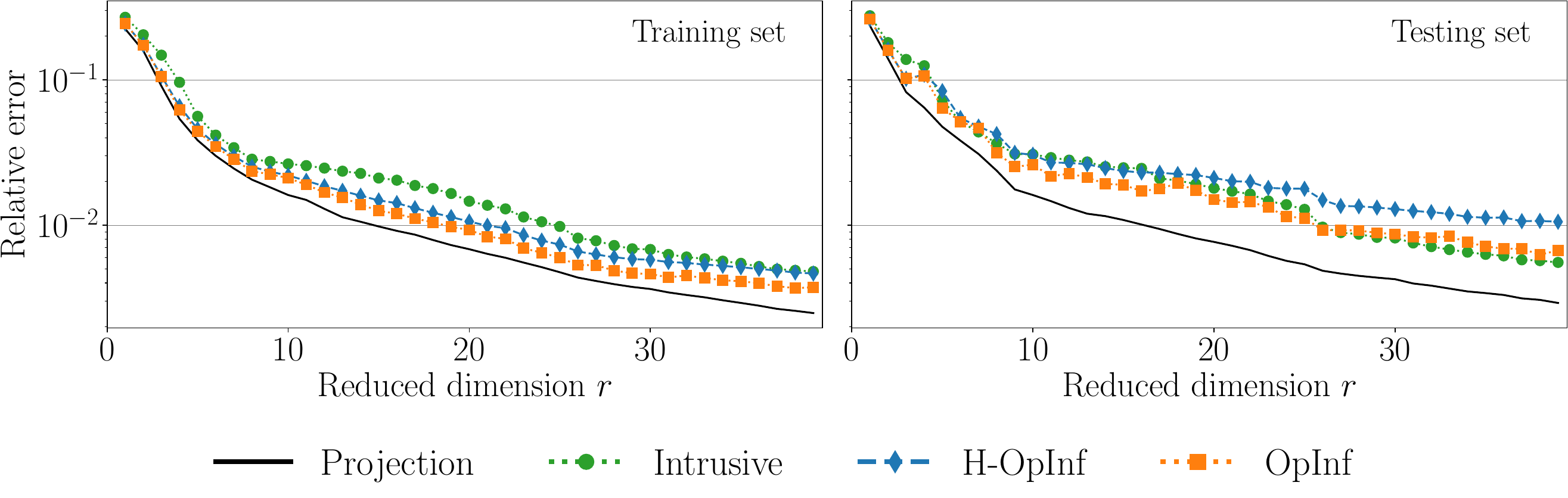}
\caption{Relative $L^2$ error in ROM solutions to the 2D wave system and the projection error of the FOM snapshots as a function of the reduced dimension $r$, over all training parameter values (left) and testing parameter values (right).
}
    \label{fig:2dwavetraintest}
\end{figure}
The previous experiment is now repeated for the two-dimensional domain $\Omega = (0, 2\pi)\times (0, 2\pi)$.
The wave speed is parameterized as
\begin{align*}
    \begin{aligned}
    c(~x,~\mu)^2 = \Bigl[ ~x \in (0,\pi)^2 \Bigr]\mu_1^2
    &+ \Bigl[ ~x \in (\pi,2\pi) \times (0,\pi) \Bigr]\mu_2^2\\
    &+ \Bigl[ ~x \in (0,\pi) \times (\pi, 2\pi) \Bigr]\mu_3^2
    + \Bigl[ ~x \in (\pi,2\pi)^2 \Bigr]\mu_4^2,
    \end{aligned}
    \quad
    ~\mu \in \mathbb{R}^4.
\end{align*}
For this experiment, the order $k$ of the finite-element spaces $W_h$ and $V_h$ is set to $2$. Using a unstructured mesh of 512 elements, this configuration results in $N=3072$ spatial degrees of freedom for both the position and momentum variables. A terminal time of $t_f= 4\pi$ with $\Delta t = \pi/500$ yields $N_t = 501$ points in time. The initial condition is given by
\begin{align*}
    y(~x,0) = \exp\big(-0.01(x_1-\pi)^2-0.01(x_2-\pi)^2\big)\sin\left(\frac{x_1}{2}\right)\sin\left(\frac{x_2}{2}\right).
\end{align*}
A total of $N_s=20$ training and $5$ testing parameters are sampled from the same parameter space as in the previous experiment.

\Cref{fig:2dwavetraintest} shows the relative error in the ROM solutions over the training and testing sets, together with the respective projection errors. The error profiles show similar trends as those in the one-dimensional experiment:
the non-symmetric OpInf ROMs tend to perform best in an $L^2$ sense, but fail to conserve the Hamiltonian. Note that the symmetric OpInf ROM tends to do slightly better than the intrusive ROM on the training set, a consequence of using a data-driven approach. On the other hand, the intrusive ROM generalizes slightly better to the testing set.

\begin{figure}
    \centering
    \begin{subfigure}[b]{0.9\textwidth}
        \centering
        \begin{minipage}{0.32\textwidth}
            \centering
            \phantom{)---}$t=0$\phantom{(}
            \includegraphics[width=\textwidth,trim=4cm 12cm 0cm 0, clip]{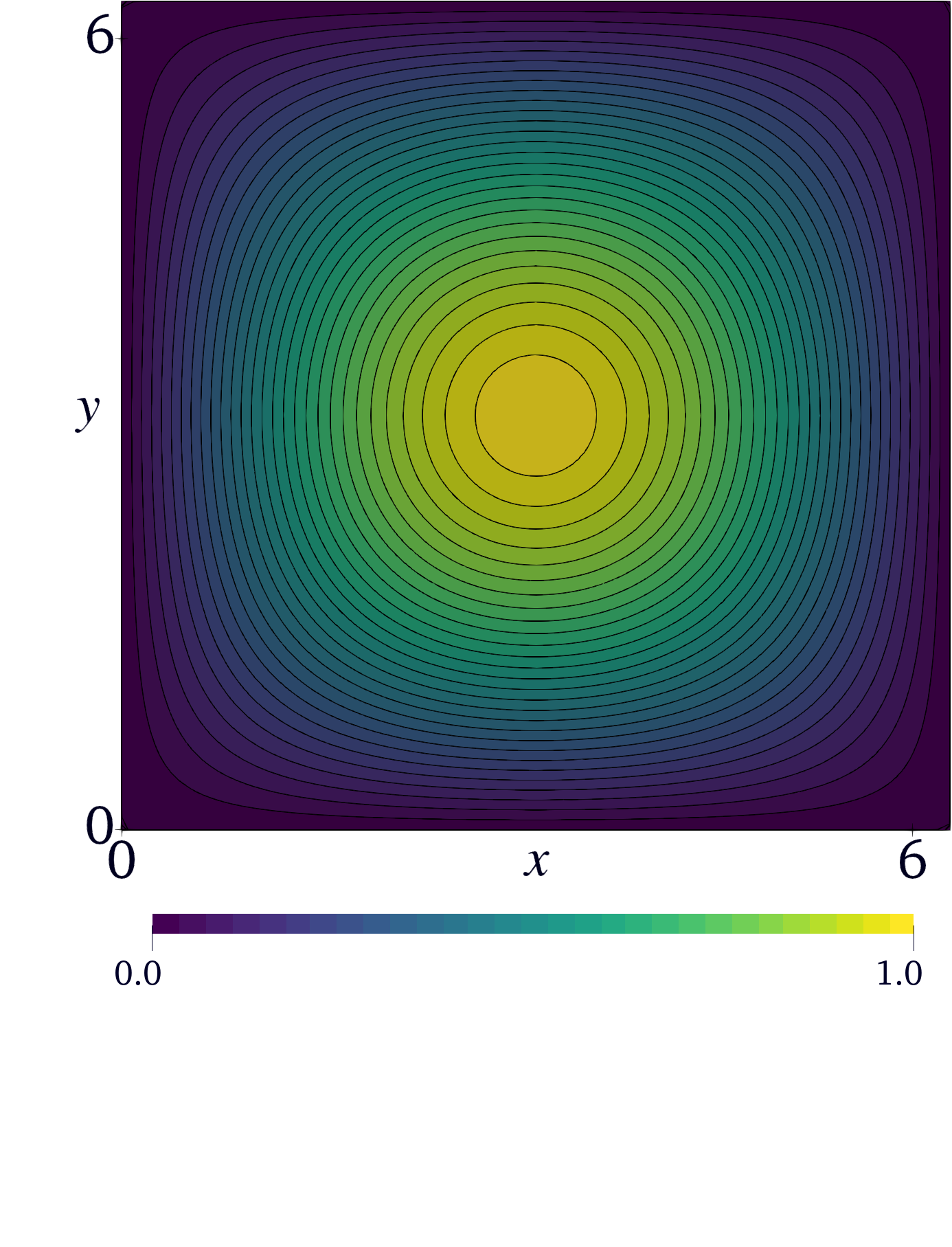}
            \vspace{0.2cm}
        \end{minipage}\begin{minipage}{0.32\textwidth}
            \centering
            \phantom{)---}$t=2\pi$\phantom{(}
            \includegraphics[width=\textwidth,trim=4cm 12cm 0 0, clip]{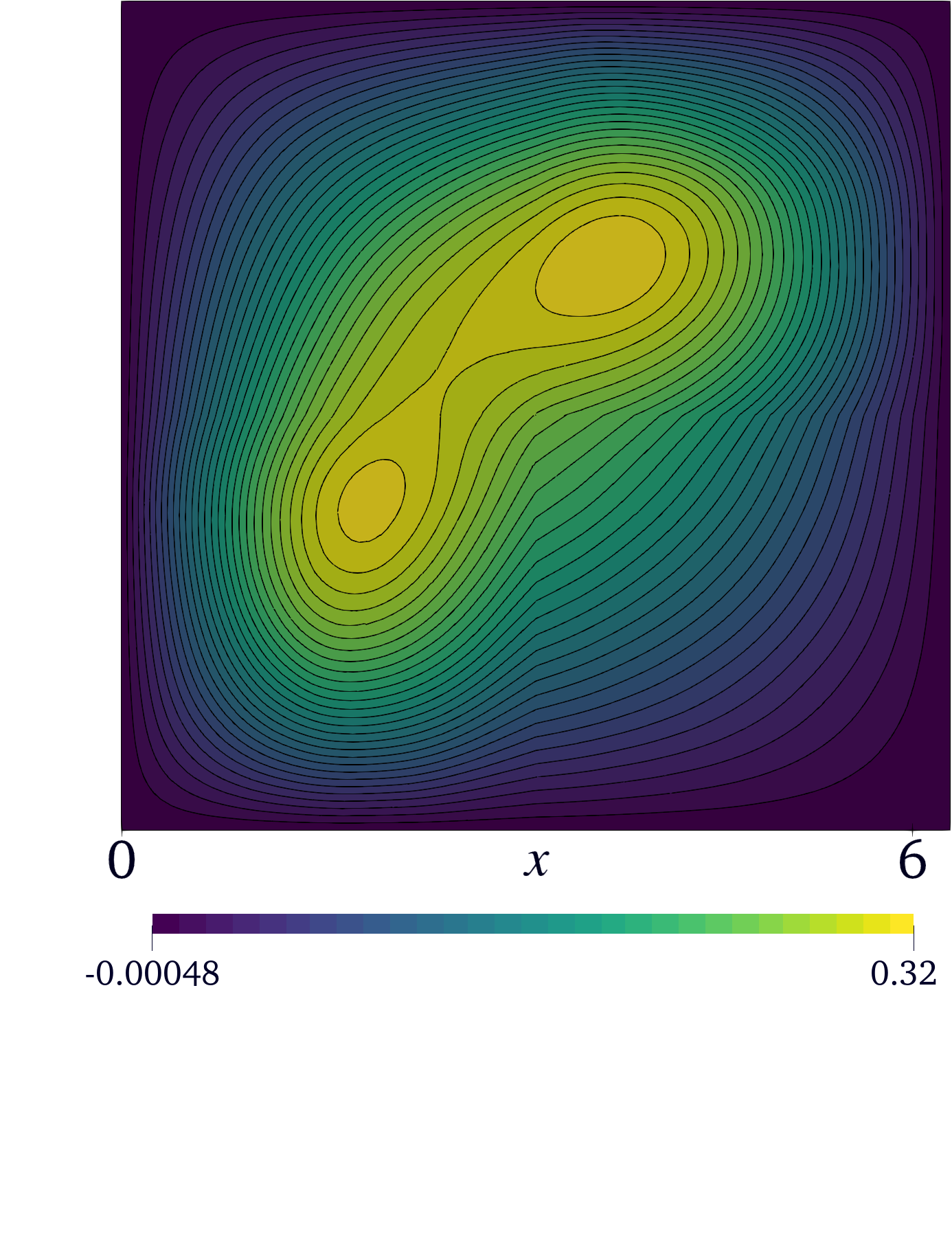}
            \vspace{0.2cm}
        \end{minipage}\begin{minipage}{0.32\textwidth}
            \centering
            \phantom{)---}$t=4\pi$\phantom{(}
            \includegraphics[width=\textwidth,trim=4cm 12cm 0 0, clip]{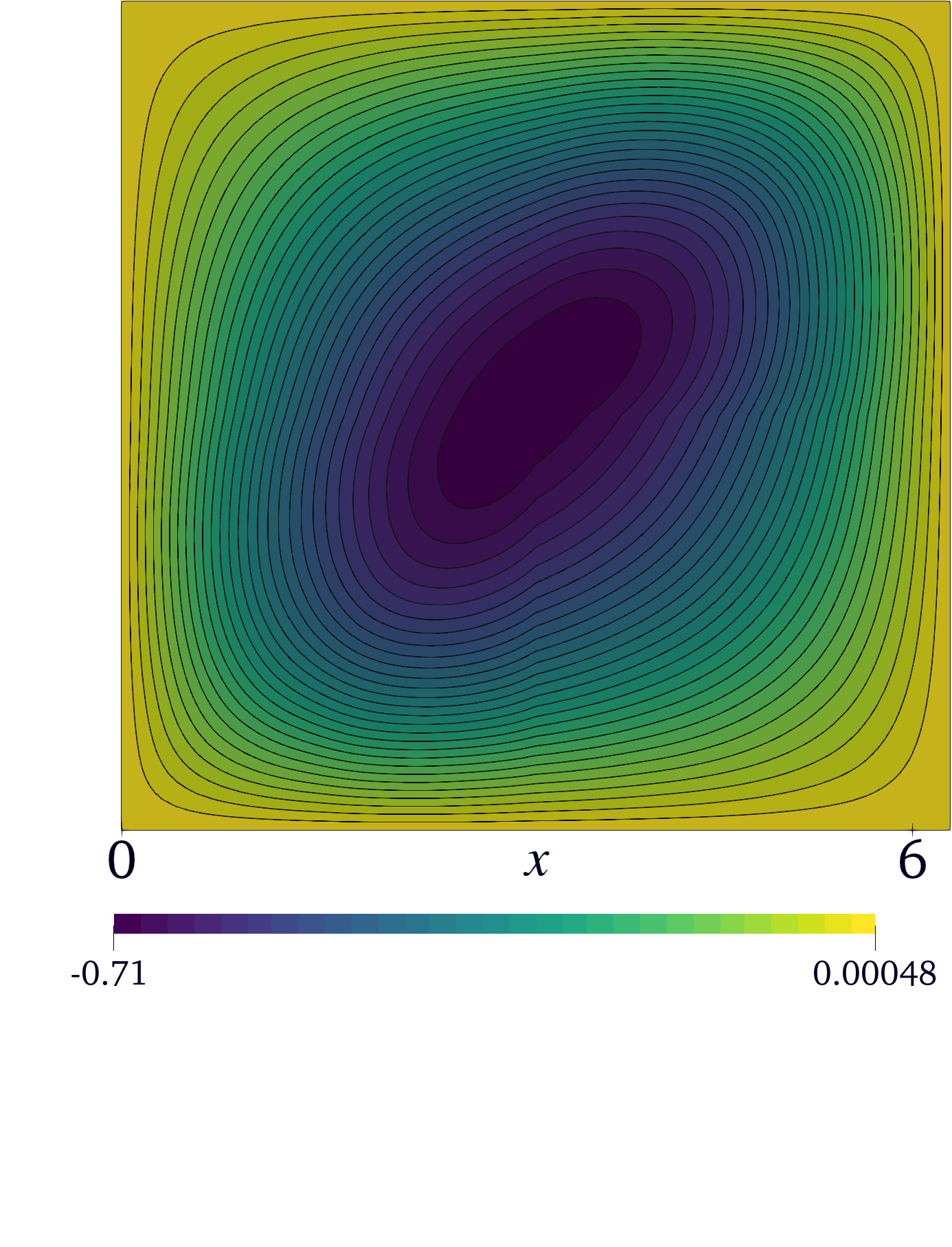}
            \vspace{0.2cm}
        \end{minipage}
\label{fig:2dwavesol}
    \end{subfigure}

    \begin{subfigure}[b]{0.9\textwidth}
        \centering
        \begin{minipage}{0.32\textwidth}
            \centering
            \phantom{)---}Intrusive\phantom{(}
            \includegraphics[width=\textwidth,trim=4cm 12cm 0 0, clip]{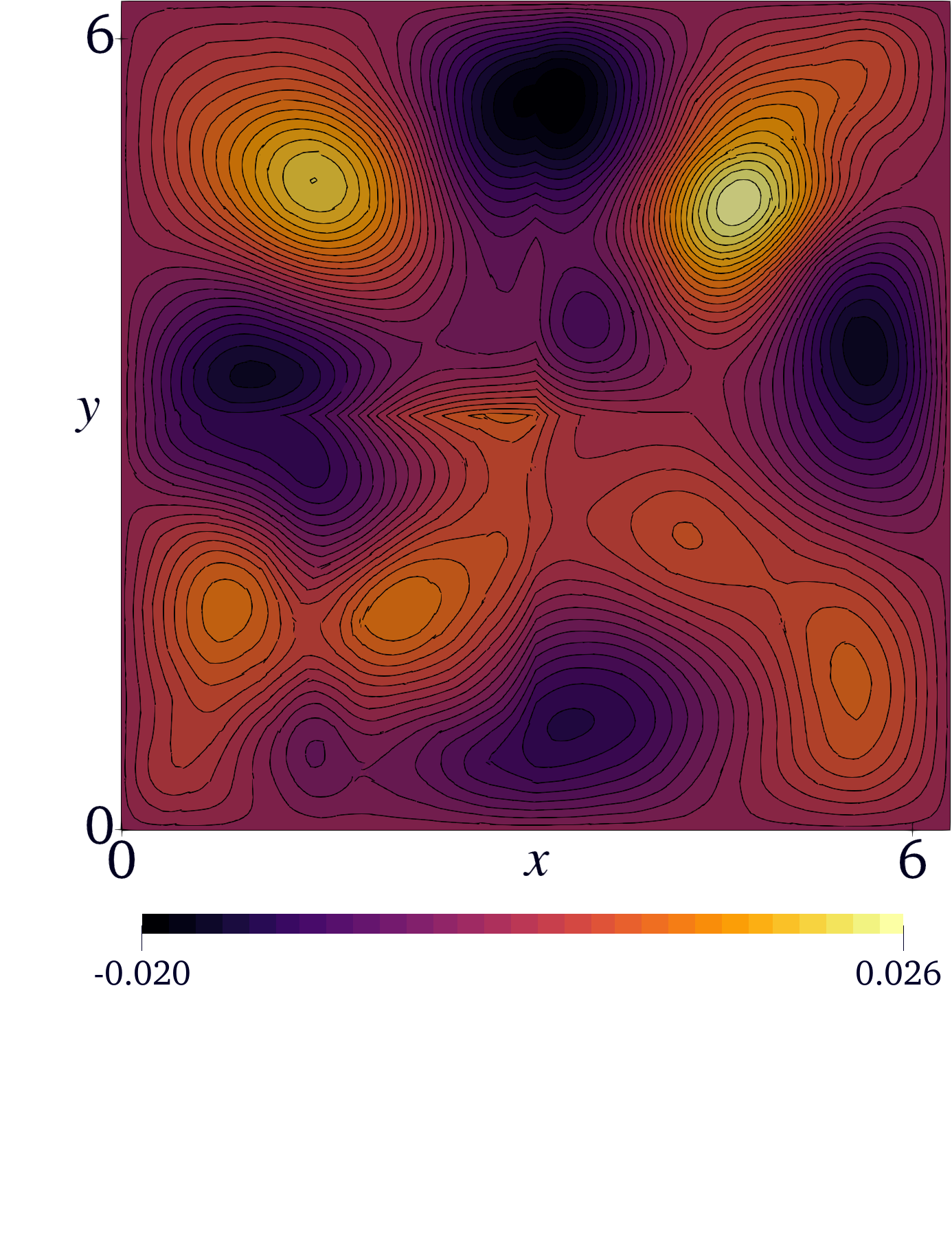}
            \vspace{0.2cm}
        \end{minipage}\begin{minipage}{0.32\textwidth}
            \centering
            \phantom{)---}H-OpInf\phantom{(}
            \includegraphics[width=\textwidth,trim=4cm 12cm 0 0, clip]{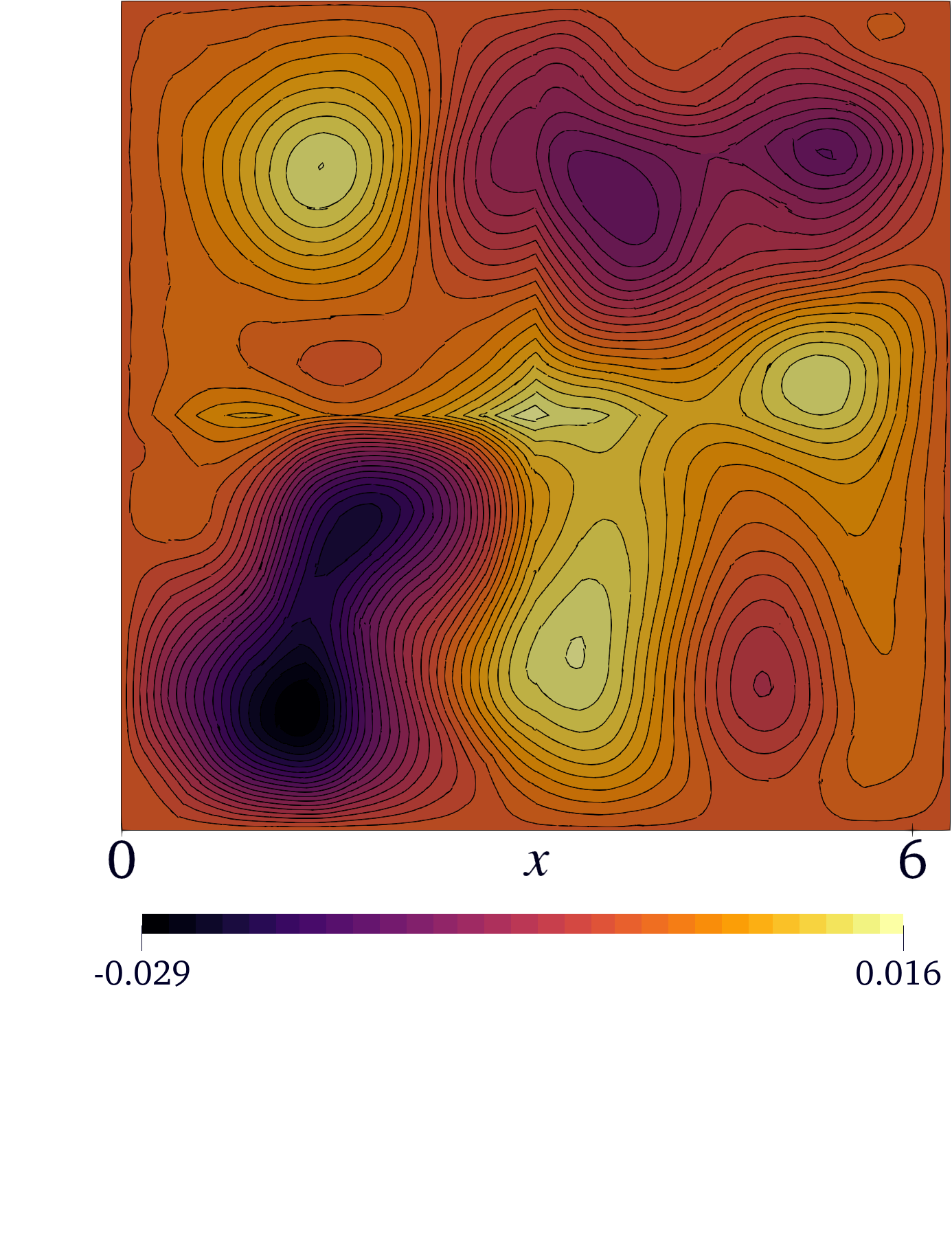}
            \vspace{0.2cm}
        \end{minipage}\begin{minipage}{0.32\textwidth}
            \centering
            \phantom{)---}OpInf\phantom{(}
            \includegraphics[width=\textwidth,trim=4cm 12cm 0 0, clip]{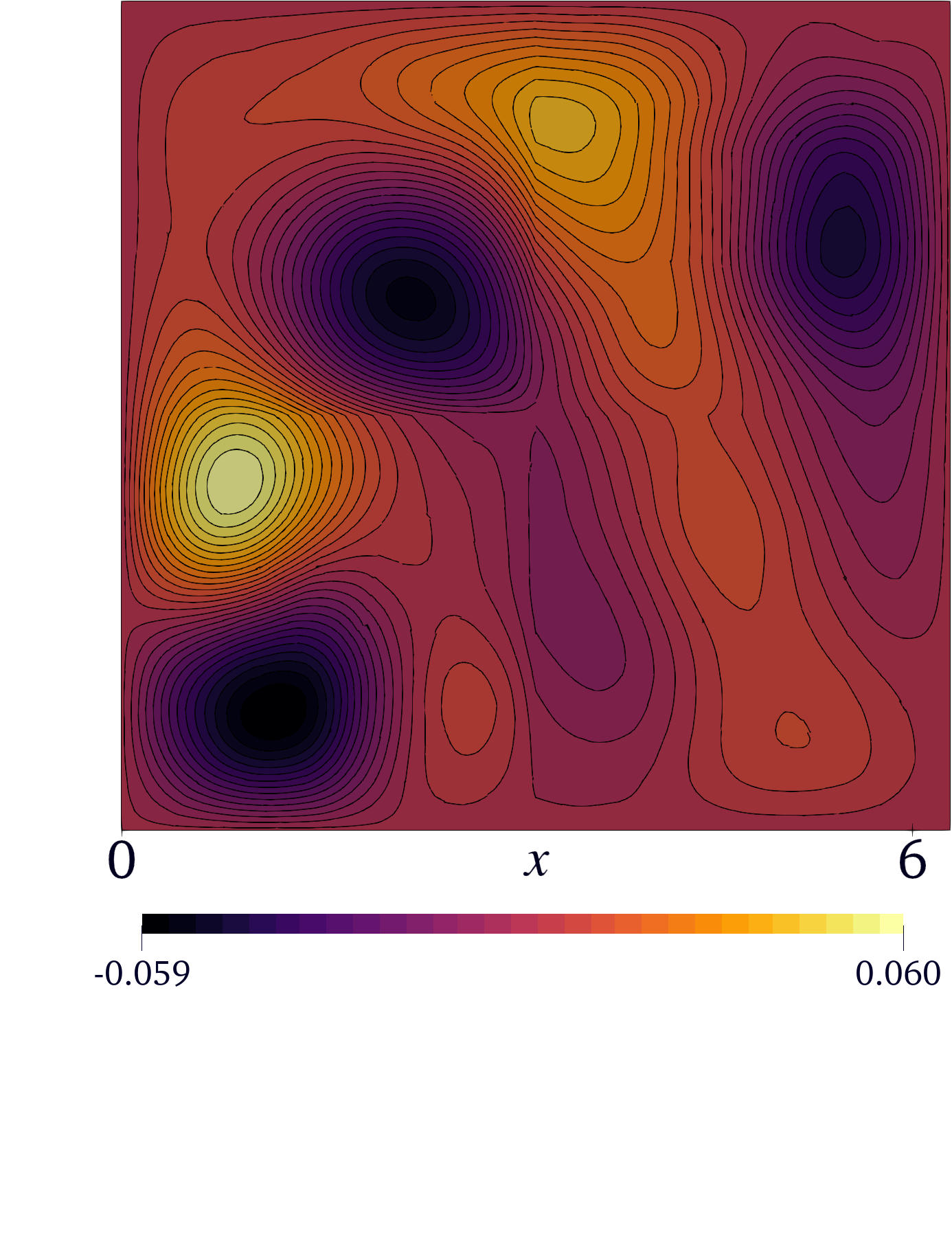}
            \vspace{0.2cm}
        \end{minipage}
\label{fig:2dwaveerr}
    \end{subfigure}
    \vspace{-0.75cm}
    \caption{Top row: FOM solution to the 2D wave system at different times. Bottom row: signed error in the ROM solutions at $t=4\pi$. Here, $r=18$ and $~\mu=[1.010, 1.946, 1.263, 1.093]\trp$.
}
    \label{fig:2dwave}
\end{figure}

Finally, \Cref{fig:2dwave} depicts the position variable at three different times for a single testing parameter, along with the corresponding signed errors in the ROM-approximated position at terminal time, with $r=18$. For this case, the intrusive ROM achieves the highest accuracy at terminal time, with a relative error of approximately 1.46\%. The H-OpInf and OpInf ROMs yield relative errors of 1.80\% and 2.70\%, respectively. The last frame of~\Cref{fig:Herr} shows again that the symmetric OpInf ROM and the intrusive ROM conserve the reduced Hamiltonian to machine precision, while the non-symmetric OpInf ROM exhibits persistent energy gain, as expected.

 \section{Conclusion}\label{sec:conclusion}

This paper presents a novel, tensorized operator inference procedure for non-intrusively learning reduced-order models corresponding to parametric systems of semi-discrete PDEs.  In its most basic form, the proposed method serves as a compact, easy-to-implement replacement for the state-linear component of existing parametric OpInf strategies. Additionally, it is shown that the advantages of this formulation enable an extension to the inference of operators with additional structure information, such as partial symmetries, which appear in the contexts of Hamiltonian systems and gradient flows. Applications to a heat equation with variable diffusion coefficient and a Hamiltonian wave equation with variable wave speed are presented, where the proposed OpInf leads to structure-preserving Hamiltonian ROMs which are energy-conserving and symplectic as required.

The structure-preserving inference procedure of \Cref{sec:Hamiltonian}, specifically in \Cref{thm:popinf-Hamiltonian} and \Cref{alg:normal_eqns_Hamiltonian}, is based on solving the normal equations corresponding to a particular minimization problem. It would be interesting to consider future extensions of this approach which avoid the potential for numerical ill-conditioning, i.e., to develop an extension of \Cref{alg:lstsq_unstructured} to the structure-preserving setting. This may have the secondary benefit of removing the need for a large, memory-intensive intermediate tensors that are required for symmetry enforcement in the current approach, as well as the potential for greater parallelism and preconditioning to improve convergence. Future work will also consider incorporating additional constraints on the learned operators beyond symmetry, as well as accounting for nonlinear terms arising from non-quadratic Hamiltonians, i.e., when $f \ne 0$ in \cref{eq:theHamiltonian}--\cref{eq:mgradCanonicalHamiltonian}.

\section*{Acknowledgments}

Support for this work was received through the U.S. Department of Energy, Office of Science, Office of Advanced Scientific Computing Research, Mathematical Multifaceted Integrated Capability Centers (MMICCS) program, under Field Work Proposal 22025291 and the Multifaceted Mathematics for Predictive Digital Twins (M2dt) project.  A.G. and S.A.M. acknowledge additional support from the John von Neumann Fellowship, a position sponsored by Sandia National Laboratories in conjunction with the Applied Mathematics Program of the U.S. Department of Energy Office of Advanced Scientific Computing Research.
Sandia National Laboratories is a multimission laboratory managed and operated by National Technology \& Engineering Solutions of Sandia, LLC, a wholly owned subsidiary of Honeywell International Inc., for the U.S.~Department of Energy's National Nuclear Security Administration under contract DE-NA0003525.
This paper describes objective technical results and analysis. Any subjective views or opinions that might be expressed in the paper do not necessarily represent the views of the U.S.~Department of Energy or the United States Government.

\appendix

\begin{appendices}

\section{Tensor algebra lemmata}\label{appendix:tensoralgebra}
This appendix lists tensor algebra conventions and proves minor results that are used in \Cref{sec:generic,sec:Hamiltonian}.

\begin{definition}[Frobenius inner product]
\label{def:frobenius}
Consider the Cartesian product $X = \mathbb{R}^{N_1\times\cdots\times N_n}$.
For order-$n$ tensors $\tens{A},\tens{B}\in X$ with entries $\textrm{A}_{i_1,\ldots,i_n}$ and $\textrm{B}_{i_1,\ldots,i_n}$, respectively, the \emph{Frobenius inner product} $\IP{\cdot}{\cdot}:X\times X\to\mathbb{R}$ is the bilinear operator
\begin{align*}
    \IP{\tens{A}}{\tens{B}}
= \sum_{i_1=1}^{N_1}\sum_{i_2=1}^{N_2}\cdots\sum_{i_n=1}^{N_n}
    \textrm{A}_{i_1,\ldots,i_n}\textrm{B}_{i_1,\ldots,i_n}.
\end{align*}
\end{definition}

\begin{lemma}[Properties of the Frobenius inner product]
\label{thm:frobeniusproperties}
Let $\IP{\cdot}{\cdot}$ denote the Frobenius inner product, with the domain inferred from the context. For $~A\in\mathbb{R}^{\ell\times m}$, $~B\in\mathbb{R}^{m \times n}$, and $~C\in\mathbb{R}^{\ell\times n}$,
\begin{align*}
    \IP{~A~B}{~C} = \IP{~A}{~C~B\trp}.
\end{align*}
Moreover, if  $\tens{T}\in\mathbb{R}^{m\times n \times p'}$ and $~\nu\in\mathbb{R}^{p'}$, then
\begin{align*}
    \IP{\tens{T}~\nu}{~B}
    = \IP{\tens{T}}{~B\otimes~\nu},
\end{align*}
where $\otimes$ denotes the outer product.
\begin{proof}
Denoting the entries of $~A$, $~B$, and $~C$ with $A_{ij}$, $B_{ij}$, and $C_{ij}$, respectively,
\begin{align*}
    \IP{~A~B}{~C}
    = \sum_{i=1}^{\ell}\sum_{j=1}^{n}
    \lr{\sum_{k=1}^{m}A_{ik}B_{kj}}C_{ij}
    = \sum_{i=1}^{\ell}\sum_{k=1}^{m}A_{ik}
    \lr{\sum_{j=1}^{n}C_{ij}B_{kj}}
    = \IP{~A}{~C~B\trp}.
\end{align*}
Next, denoting the entries of $\tens{T}$ with $\mathrm{T}_{ijk}$ and the entries of $~\nu$ with $\nu_{i}$,
\begin{align*}
    \IP{\tens{T}~\nu}{~B}
    = \sum_{k=1}^{m}\sum_{j=1}^{n}
    \lr{\sum_{i=1}^{p'}\mathrm{T}_{kji}\nu_{i}}B_{kj}
    = \sum_{k=1}^{m}\sum_{j=1}^{n}\sum_{i=1}^{p'}\mathrm{T}_{kji}
    \lr{B_{kj}\nu_{i}}
    = \IP{\tens{T}}{~B\otimes~\nu}.
\end{align*}
\end{proof}
\end{lemma}

\begin{definition}[Tensor contraction]
\label{def:contraction}
Let $$\tens{A}\in\mathbb{R}^{N_1\times \cdots\times N_{n-2} \times X\times Y}
\quad\text{and}\quad\tens{B}\in\mathbb{R}^{Y\times X\times M_3 \times \cdots \times M_m}$$ with entries $\mathrm{A}_{i_1,\ldots,i_{n-2},i_x,i_y}$ and $\mathrm{B}_{k_y,k_x,k_3,\ldots,k_m}$, respectively.
The contraction $$\tens{A}:\tens{B}\in\mathbb{R}^{N_1\times\cdots\times N_{n-2}\times M_3\times\cdots\times M_m}$$ is the tensor of order $N + M - 2$ with entries
\begin{align*}
    (\tens{A}:\tens{B})_{i_1,\ldots,i_{n-2},k_3,\ldots,k_m}
    = \sum_{x=1}^{X}\sum_{y=1}^{Y}\mathrm{A}_{i_1,\ldots,i_{n-2},x,y}\mathrm{B}_{y,x,k_3,\ldots,k_m}.
\end{align*}
\end{definition}

\begin{lemma}[Outer product and contraction]
\label{thm:contractionproduct}
For $\tens{T}\in\mathbb{R}^{r\times r \times p'}$, $~B\in\mathbb{R}^{r\times r}$, and $~\nu\in\mathbb{R}^{p'}$,
\begin{align*}
    (\tens{T}~\nu)~B\otimes~\nu
    = \tens{T} : (~\nu\otimes~B\otimes~\nu),
\end{align*}
where $:$ is the contraction operator of \Cref{def:contraction}.
\begin{proof}
Let $\mathrm{T}_{ijk}$, $B_{ij}$, and $\nu_{i}$ denote the entries of $\tens{T}$, $~B$, and $~\nu$, respectively.
Working componentwise,
\begin{align*}
    \lr{(\tens{T}~\nu)~B\otimes~\nu}_{ijk}
    = \sum_{a=1}^{r}(\tens{T}~\nu)_{ia}(~B\otimes~\nu)_{ajk}
    = \sum_{a=1}^{r}\sum_{b=1}^{p'}\mathrm{T}_{iab}\nu_{b}B_{aj}\nu_{k}.
\end{align*}
On the other hand,
\begin{align*}
    \lr{\tens{T} : (~\nu\otimes~B\otimes~\nu)}_{ijk}
    = \sum_{a=1}^{r}\sum_{b=1}^{p'}\mathrm{T}_{iab}(~\nu\otimes~B\otimes~\nu)_{bajk}
    = \sum_{a=1}^{r}\sum_{b=1}^{p'}\mathrm{T}_{iab}\nu_{b}B_{aj}\nu_{k}.
\end{align*}
\end{proof}
\end{lemma}

\begin{lemma}[Vectorization and contraction]
\label{thm:contractionvectorization}
For tensors $\tens{T} \in \mathbb{R}^{r \times r \times p'}$ and $\tens{X}\in\mathbb{R}^{p'\times r \times r \times p'}$,
\begin{align*}
\cvec_{23}\,\lr{\tens{T}:\tens{X}}
    = \lr{\cvec_{23}\,\tens{T}}
    \lr{\rvec_{12}\cvec_{34}\,\tens{X}},
\end{align*}
where $:$ is the contraction operator of \Cref{def:contraction}.
\begin{proof}
Let $\mathrm{T}_{ijk}$ and $\mathrm{X}_{ijk\ell}$ denote the entries of $\tens{T}$ and $\tens{X}$, respectively.
Working componentwise,
\begin{align*}
    \lr{\cvec_{23}(\tens{T}:\tens{X})}_{i,(k-1)r + j}
    = (\tens{T}:\tens{X})_{ijk}
    = \sum_{a=1}^{r}\sum_{b=1}^{p'}\mathrm{T}_{iab}\mathrm{X}_{bajk}
\end{align*}
for each $1 \le i,j \le r$ and $1 \le k \le p'$.
On the other hand, for $1\le a \le r$ and $1 \le b \le p'$,
\begin{align*}
    (\cvec_{23}\,\tens{T})_{i,(b-1)r+a}
    &= \mathrm{T}_{iab},
    \\
    (\rvec_{12}\cvec_{34}\,\tens{X})_{(b-1)r+a,(k-1)r+j}
    = (\cvec_{34}\,\tens{X})_{b,a,(k-1)r+j}
    &= \mathrm{X}_{bajk}.
\end{align*}
Hence, the $(i,(k-1)r + j)$-th entry of $\lr{\cvec_{23}\,\tens{T}}\lr{\rvec_{12}\cvec_{34}\,\tens{X}}$ is given by
\begin{align*}
    \sum_{a=1}^{r}\sum_{b=1}^{p'}
    (\cvec_{23}\,\tens{T})_{i,(b-1)r+a}
    (\rvec_{12}\cvec_{34}\,\tens{X})_{(b-1)r+a,(k-1)r+j}
    =\sum_{a=1}^{r}\sum_{b=1}^{p'}\mathrm{T}_{iab}\mathrm{X}_{bajk},
\end{align*}
which completes the proof.
\end{proof}
\end{lemma}

\begin{lemma}[Iterated vectorization and the Kronecker product]\label{lem:slicktrick}
    Let $~A,~B,~C$ be matrices potentially with different dimensions, and let $\rvec_{135}=\rvec_{13}\rvec_{35}$.
    Then,
    \begin{align*}
    ~A\otimes_K~B\otimes_K~C = \rvec_{135}\rvec_{246}\lr{~A\otimes~B\otimes~C}.
    \end{align*}
\end{lemma}
\begin{proof}
    By the definition of the Kronecker product along with its associativity, it follows that
    \begin{align*}
        ~A\otimes_K\lr{~B\otimes_K~C} &= \rvec_{13}\rvec_{24}\lr{~A\otimes\rvec_{13}\rvec_{24}\lr{~B\otimes~C}} \\
        &= \rvec_{13}\rvec_{24}\lr{\rvec_{35}\rvec_{46}\lr{~A\otimes~B\otimes~C}}\\
        &= \rvec_{13}\rvec_{35}\rvec_{24}\rvec_{46}\lr{~A\otimes~B\otimes~C},
    \end{align*}
    which immediately implies the conclusion.
\end{proof}

\section{Adjoints and gradients}\label{app:morthonormal}
This appendix briefly outlines the relationship between adjoints and gradients with respect to different inner products.

\begin{lemma}[Weighted adjoint]\label{thm:Madjoint}
Suppose $~A\in\mathbb{R}^{N\times N}$ and $~x,~y\in\mathbb{R}^N$.  Then the adjoint $~A^*$ of $~A$ with respect to the inner product $\IP{~x}{~y}_{~M} = ~x\trp~M~y$ defined by the symmetric and positive definite matrix $~M\in\mathbb{R}^{N\times N}$ is $~A^* = ~M^{-1}~A\trp~M$.
\end{lemma}
\begin{proof}
    Observe that
    \begin{align*}
    \IP{~A~x}{~y}_{~M}
    = ~x\trp~A\trp~M~y
    = ~x\trp~M\lr{~M^{-1}~A\trp~M}~y
    = \IP{~x}{(~M^{-1}~A\trp~M)~y}_{~M}.
\end{align*}
\end{proof}

The adjoint of a non-square matrix $~U\in\mathbb{R}^{N\times r}$ is only slightly more delicate, requiring weighted inner products on both $\mathbb{R}^N$ and $\mathbb{R}^r$.

\begin{lemma}[Weighted non-square adjoint]
    Let $~U\in\mathbb{R}^{N\times r}$, $\hat{~x}\in\mathbb{R}^r$, $~y\in\mathbb{R}^N$, and suppose $~M\in\mathbb{R}^{N\times N}$ and $\hat{~M}\in \mathbb{R}^{r\times r}$ are symmetric positive definite.  Then the adjoint of $~U$ with respect to the $~M$ and $\hat{~M}$ inner products is $~U^* = \hat{~M}^{-1}~U\trp~M$.
\end{lemma}
\begin{proof}
    Similar to before, observe that
    \begin{align*}
        \IP{~U\hat{~x}}{~y}_{~M}
        = \hat{~x}\trp~U\trp~M~y
        = \hat{~x}\trp\hat{~M}\lr{\hat{~M}^{-1}~U\trp~M}~y
        = \IP{\hat{~x}}{(\hat{~M}^{-1}~U\trp~M)~y}_{\hat{~M}}.
\end{align*}
\end{proof}

The previous results connect the $~M$-adjoint to the Euclidean adjoint. The next result establishes the same connection for the gradient operator on scalar functions.

\begin{lemma}[Weighted gradient]
    Let $f:\mathbb{R}^N\to\mathbb{R}$ be a differentiable function and $~M\in\mathbb{R}^{N\times N}$ be a symmetric positive definite matrix.  Then the gradient of $f$ with respect to the $~M$-weighted inner product is $\Mnabla f = ~M^{-1}\nabla f$.
\end{lemma}
\begin{proof}
    For any $~x,~v\in\mathbb{R}^N$, it follows from Taylor's theorem that
    \begin{align*}
        df|_{~x}(~v)
        = ~v\trp\nabla f(~x)
        = ~v\trp~M~M^{-1}\nabla f(~x)
= \IP{~v}{~M^{-1} \nabla f(~x)}.
\end{align*}
    Since $~v$ is arbitrary, this implies that $\nabla f = ~M\Mnabla f$ and the conclusion follows.
\end{proof}

The final result shows why affine-parametric state matrices can be assumed to be self-adjoint in the Hamiltonian systems context.

\begin{lemma}[Self-adjointness in quadratic functionals]\label{thm:selfadjoint}
Let $~A,~M\in\mathbb{R}^{N\times N}$ and assume that $~M$ is symmetric and positive definite.
Only the self-adjoint part of $~A$ contributes to the value of the quadratic functional $H:\mathbb{R}^{N}\to\mathbb{R}$ given by $H(~y) = ~y^{*}~A~y = ~y\trp~M~A~y$.
Specifically, $H(~y)$ is proportional to $~y^{*}(~A + ~A^{*})~y$.

\begin{proof}
Decomposing $~A$ into adjoint and skew-adjoint parts,
\begin{align*}
    H(~y)
    &= \frac{1}{2}~y^{*}(~A + ~A^{*})~y + \frac{1}{2}~y^{*}(~A - ~A^{*})~y.
\end{align*}
Because $~M$ is symmetric,
\begin{align*}
    ~y^{*}~A~y
= \lr{~y\trp~M~A~y}\trp
    = ~y\trp~M~M^{-1}~A\trp~M~y
    = ~y^{*}~A^{*}~y,
\end{align*}
and hence $~y^{*}(~A - ~A^{*})~y = 0$.
\end{proof}
\end{lemma}

\section{Wave equation finite element model}
\label{app:systems}
This appendix derives the matrix-vector representations of the full-order finite element discretization used in \Cref{sec:numerics-wave} for the parameterized wave equation~\cref{eq:wave}. Suppose $\{\phi_i\}_{i=1}^{N_W}$ and $\{~\psi_i\}_{i=1}^{N_V}$ are basis functions for the spaces $W_h$ and $V_h$, respectively, and let $~q, ~p$, and $~\sigma$ denote the vectors containing the degrees of freedom of $q_h$, $p_h$, and $~\sigma_h$. By using the test function $w_h = ~\psi_j$ and the expansions $~\sigma_h = \sum_{i=1}^{N_V}\sigma_i(t)~\psi_i(~x)$ and $q_h = \sum_{i=1}^{N_W}q_i(t)\phi_i(~x)$ in~\cref{eq:sigmah}, it follows that
\begin{align*}
    \sum_{i=1}^{N_V} \left( \frac{1}{c(~\mu)^2} ~\psi_i, ~\psi_j \right)_{\Omega_h}\sigma_i = - \sum_{i=1}^{N_W}\left(\phi_i, \nabla\cdot~\psi_j\right)_{\Omega_h}q_i.
\end{align*}
The above equation can be converted to the matrix-vector form
\begin{align}\label{eq:sigmasolve}
    ~M_V(~\mu) ~\sigma = - ~S ~q,
\end{align}
where the matrices $~M_V(~\mu) \in \mathbb{R}^{N_V\times N_V}$ and $~S \in \mathbb{R}^{N_V \times N_W}$ have components
\begin{align*}
    \lr{~M_V(~\mu)}_{ij}
    = \lr{c(~\mu)^{-2}~\psi_i, ~\psi_j}_{\Omega_h}
= \sum_{k=1}^p \frac{1}{\mu_k^{2}} \lr{~\psi_i,~\psi_j}_{\Omega_k}
    \qquad
    \lr{~S}_{ji}
    = \lr{\phi_i, \nabla\cdot~\psi_j}_{\Omega_h}.
\end{align*}
Note that matrix $~M_V(~\mu)$ depends affinely on $~\mu$ with coefficients $~\mu' = [\,\mu_1^{-2}\,\cdots\,\mu_p^{-2}\,]\trp$.
Hence, $~M_V(~\mu) = \tens{T}~\mu'$ for the constant tensor $\tens{T}\in\mathbb{R}^{N_V\times N_V \times p}$ with entries
$\mathrm{T}_{ijk} = \lr{~\psi_i,~\psi_j}_{\Omega_k}$.
On the other hand, inserting the expansion $p_h = \sum_{i=1}^{N_W} p_i(t)\phi_i(~x)$ and the above expansion for $~\sigma_h$ into~\cref{eq:qhphb} and testing against $w_h=\phi_j$ yields
\begin{align*}
    \lr{~M_W\dot{~p}}_j = \sum_{i=1}^{N_W}\left(\phi_i, \phi_j\right)_{\Omega_h}\dot p_i =
    \sum_{i=1}^{N_V}
    \left(\nabla\cdot ~\psi_i, \phi_j\right)_{\Omega_h}\sigma_i = \lr{~S\trp~\sigma}_j.
\end{align*}
By substituting for $~\sigma$ from~\cref{eq:sigmasolve}, the equation can be written equivalently as
\begin{align}\label{eq:epvec}
    ~M_W\dot{~p}
    = -~S\trp~M_V(~\mu)^{-1}~S~q
    = -~S\trp\lr{\tens{T}~\mu'}^{-1}~S~q.
\end{align}
Finally, using the test function $w_h = \phi_j$ and the above expansions of $q_h$ and $p_h$, in \cref{eq:qhpha},
\begin{align}\label{eq:eqvec}
    \lr{~M_W\dot{~q}}_j
    = \sum_{i=1}^{N_W}\left(\phi_i, \phi_j\right)_{\Omega_h}\dot q_i
    = \sum_{i=1}^{N_W}\left(\phi_i, \phi_j\right)_{\Omega_h}p_i = (~M_W~p)_j.
\end{align}
Together, \cref{eq:epvec} and \cref{eq:eqvec} lead to the following block form of the FOM~\cref{eq:mixedscheme},
\begin{align*}
    \dot{~y} &=
    \left[\begin{array}{c}
        \dot{~q}\\
        \dot{~p}
    \end{array}\right]
    =
    \left[\begin{array}{ c  c }
    ~0 & ~I \\
    -~I & ~0
    \end{array}\right]\left[\begin{array}{ c  c }
        ~M_W^{-1}~S\trp~M_V(~\mu)^{-1}~S & ~0 \\
        ~0 & ~I
    \end{array}\right]
    \left[\begin{array}{c}
        ~q\\
        ~p
    \end{array}\right].
\end{align*}
As a final point, the expression for the discrete Hamiltonian in~\cref{eq:discreteH} follows from (suppressing the dependence on $~\mu$)
\begin{equation*}
\begin{split}
    \begin{aligned}
        H_h(q_h,p_h)
        &= \frac{1}{2}\left(\sum_{i=1}^{N_W}p_i\phi_i,\sum_{j=1}^{N_W}p_j\phi_j\right)_{\Omega_h} + \frac{1}{2}\left(\frac{1}{c^2(\cdot,~\mu)} \sum_{i=1}^{N_V}\sigma_i~\psi_i,\sum_{j=1}^{N_V}\sigma_j~\psi_j\right)_{\Omega_h}
        \\
        &= \frac{1}{2}\sum_{i=1}^{N_W}p_i\sum_{j=1}^{N_W}\left(\phi_i,\phi_j\right)_{\Omega_h}p_j + \frac{1}{2}\sum_{i=1}^{N_V}\sigma_i\sum_{j=1}^{N_V}\lr{\frac{1}{c(\cdot,~\mu)^2}~\psi_i,~\psi_j}_{\Omega_h}\sigma_j
        \\
        &= \frac{1}{2}\sum_{i=1}^{N_W}\sum_{j=1}^{N_W}p_i (~M_W)_{ij}p_j + \frac{1}{2}\sum_{i=1}^{N_V}\sum_{j=1}^{N_V}\sigma_i(~M_V)_{ij}\sigma_j
        \\
        &= \frac{1}{2} ~p\trp ~M_W ~p + \frac{1}{2}~\sigma\trp ~M_V ~\sigma
        \\
        &= \frac{1}{2}~p\trp ~M_W ~p + \frac{1}{2}(-~M_V^{-1} ~S ~q)\trp ~M_V (-~M_V^{-1} ~S ~q)
        \\
        &= \frac{1}{2}~p\trp ~M_W ~p + \frac{1}{2} ~q\trp ~S\trp ~M_V^{-1} ~S ~q
        \\
        &= \frac{1}{2}\langle ~p, ~p \rangle_{~M_W} + \frac{1}{2}\langle ~q, ~M_W^{-1}~S\trp ~M_V^{-1} ~S~q \rangle_{~M_W},
    \end{aligned}
\end{split}
\end{equation*}
where $\langle ~a, ~b \rangle_{~M_W} = ~a\trp ~M_W ~b$ denotes the $~M_W$-weighted inner product. \end{appendices}

\bibliographystyle{siamplain} \bibliography{references}

\end{document}